\documentclass[12pt,letterpaper]{amsart} 
\usepackage[utf8]{inputenc}

\usepackage[margin=3cm]{geometry} 

\usepackage{amssymb, amsmath, amscd, amsthm, mathtools}
\usepackage{mathrsfs}
\usepackage{mathdots}

\usepackage{tikz-cd}
\usepackage[all]{xy}
\usepackage{epstopdf} 

\usepackage{enumitem}
\allowdisplaybreaks
\usepackage{setspace}
\usepackage[bookmarks=true,bookmarksnumbered=true,colorlinks]{hyperref}
\definecolor{dartmouthgreen}{rgb}{0.05, 0.5, 0.06}  
\hypersetup{
    linkcolor=blue,         
    citecolor=dartmouthgreen,
}

\makeatletter

\newcommand{\nosubsection}[1]{%
  \par\addvspace{.7\linespacing\@plus.7\linespacing} 
  {\noindent\normalfont\bfseries #1} 
  \par\nopagebreak\addvspace{.5\linespacing} 
  \@afterindenttrue\@afterheading 
}

\def\subsection{\@startsection{subsection}{2}%
  \z@{.7\linespacing\@plus.7\linespacing}{.5\linespacing}%
  {\normalfont\bfseries}}
\makeatother

\numberwithin{equation}{section}

\theoremstyle{plain}
\newtheorem{thm}{Theorem}[section]
\newtheorem{lem}[thm]{Lemma}
\newtheorem{prop}[thm]{Proposition}
\newtheorem{cor}[thm]{Corollary}
\newtheorem{fac}[thm]{Fact}

\newtheorem*{customproof0}{Proof of Theorem~\ref{prop2}}

\theoremstyle{definition}
\newtheorem{rem}[thm]{Remark}
\newtheorem{defn}[thm]{Definition}
\newtheorem*{con}{Conjecture}
\newtheorem{con1}[thm]{Conjecture}

\newcommand{\Mp}{\widetilde{\Sp}}

\newcommand{\pair}[1]{\langle #1 \rangle}
\renewcommand{\1}{{\bf 1}}

\newcommand{\bb}{\bigbreak}

\def\Sym{\operatorname{Sym}}

\def\diag{\operatorname{diag}}

\def\Hom{\operatorname{Hom}}
\def\ind{\operatorname{ind}}
\def\Ind{\operatorname{Ind}}
\def\Irr{\operatorname{Irr}}

\def\Re{\operatorname{Re}}
\def\Res{\operatorname{Res}}

\def\Tr{\operatorname{Tr}}

\def\beq{\begin{equation}}
\def\eeq{\end{equation}}
\def\beqn{\begin{equation*}}
\def\eeqn{\end{equation*}}
\def\beqna{\begin{eqnarray}}
\def\eeqna{\end{eqnarray}}
\def\beqnan{\begin{eqnarray*}}
\def\eeqnan{\end{eqnarray*}}

\def\TT{\mathbb{T}}
\def\tb{\textbf}

\def\wt{\widetilde}
\def\vi{\varphi}

\def\mc{\mathcal}

\def\mf{\mathfrak}

\def\bs{\backslash}

\def\GL{\mathrm{GL}}

\def\O{\mathrm{O}}
\def\SO{\mathrm{SO}}
\def\SL{\mathrm{SL}}
\def\Sp{\mathrm{Sp}}
\def\G{\mathrm{G}}
\def\H{\mathrm{H}}
\def\D{\mathrm{D}}
\def\J{\mathrm{J}}
\def\I{\mathrm{I}}
\def\U{\mathrm{U}}
\def\P{\mathrm{P}}
\def\B{\mathrm{B}}
\def\R{\mathrm{R}}
\def\N{\mathrm{N}}
\def\M{\mathrm{M}}
\def\T{\mathrm{T}}

\def\K{\mathrm{K}}
\def\V{\mathrm{V}}

\def\A{\mathbb{A}}
\def\Ha{\mathbb{H}}
\def\Z{\mathrm{Z}}

\def\e{\ve}

\def\CC{\mathbb{C}}

\def\VV{\mathbb{V}}
\def\WW{\mathbb{W}}

\def\ZZ{\mathbb{Z}}
\def\z{\mathfrak{Z}}

\def\OO{\mathcal{O}}
\def\RR{\mathbb{R}}

\def\II{\mathbb{I}}
\def\col{\coloneqq}

\def\ss{\subset}

\def\k{\mathfrak{k}}

\def\la{\langle}
\def\ra{\rangle}

\def\y{\textbf{y}}

\def\w{\textbf{w}}
\def\t{\textbf{t}}

\def\ve{\varepsilon}

\title[Special periods and some non-tempered GGP conjecture]{SPECIAL PERIODS AND SOME NON-TEMPERED CASES OF THE GAN-GROSS-PRASAD CONJECTURE}
\begin{document}

\begin{abstract}
In this paper, we establish a relationship between special periods and special \(L\)-values of automorphic representations of classical groups, and prove the non-tempered global Gan--Gross--Prasad conjecture in several cases.

 Our approach consists of two main steps. First, inspired by Rallis' tower property, we study the interaction between special periods and the tower property for the genericity of global theta lifts. Second, we investigate the relationship between the analytic properties of \(L\)-functions and special periods via the Rankin--Selberg integral method.

 Combining these results with non-vanishing criteria for global theta lifts in terms of various \(L\)-values, we prove three explicit higher-corank families of non-tempered cases of the global
Gan–Gross–Prasad conjecture.
\end{abstract}

\dedicatory{Dedicated to the memory of Benedict Gross and Stephen Rallis}
\author{Jaeho Haan}
\address{Jaeho Haan}
\curraddr{Department of Mathematical Sciences, KAIST, 291 Daehak-ro, Yuseong-gu, Daejeon, 34141, South Korea}
\email{jaehohaan@gmail.com}
\author{Sanghoon Kwon}
\address{Sanghoon Kwon}
\curraddr{Department of Mathematical Education\\ Catholic Kwandong University \\ Gangneung 25601 \\ Republic of Korea}
\email{shkwon1988@gmail.com \\ skwon@cku.ac.kr}



\maketitle
\tableofcontents

\section{Introduction}

\nosubsection{Motivation and purpose of the paper}
Special values of automorphic $L$-functions lie at the heart of modern number theory and arithmetic geometry.
In the automorphic setting, a guiding principle is that distinguished period integrals should detect central
(or near-central) values of suitable $L$-functions.
A particularly influential manifestation of this principle is the Gan--Gross--Prasad (GGP) conjecture,
which relates the non-vanishing of Bessel and Fourier--Jacobi periods on classical groups
to central values of Rankin--Selberg $L$-functions.

The present paper proves several cases of the \emph{non-tempered} global GGP conjecture~\cite{GGP20}.
More precisely, we establish the conjecture for the following three group pairs:
\[
(\SO_{2n},\SO_3),\qquad (\Sp_{2n},\Mp_2),\qquad (\Mp_{2n},\Sp_2),
\]
in a natural and explicitly describable class of non-tempered situations, where the second factor is fixed
to be either the trivial representation or a global Weil representation, both of which are non-tempered automorphic representation.

To state the theorem cleanly, we begin by recalling the brief history and formulation of the GGP conjecture, which underlies the main results of this paper.
\nosubsection{Bessel and Fourier-Jacobi periods}

In the early 1990s, Gross and Prasad \cite{GP92,GP94} formulated a striking conjecture relating the non-vanishing of certain period integrals on special orthogonal groups to the non-vanishing of central values of certain tensor product \(L\)-functions, together with its local counterpart. Later, together with Gan \cite{GGP12, GGP20}, they extended this conjecture to all classical groups as well as metaplectic groups. This extended conjecture is now referred to as the global \textit{GGP conjecture}.

The global GGP conjecture involves two types of period integrals: Bessel periods and Fourier--Jacobi periods. Bessel periods arise from automorphic forms on orthogonal groups or hermitian unitary groups, whereas Fourier--Jacobi periods are associated with automorphic forms on metaplectic--symple\-ctic groups or skew-hermitian unitary groups. In this paper, we restrict our attention to orthogonal and metaplectic--symplectic groups; thus, we briefly recall the definitions of these two types of periods in this setting.

Let \(F\) be a number field and let \(F_v\) denote its local completion at a place \(v\) of \(F\).
Let \(\mathbb{A}\) be the adele ring of \(F\), and let \(\psi\) be a nontrivial unitary additive character of \(F \backslash \mathbb{A}\).

For $\ve \in \{\pm\}$, let $Z_m \subset Z_n$ be non-degenerate $m$ and $n$-dimensional $\ve$-Hermitian spaces over $F$ satisfying:
\begin{itemize}
\item $\ve \cdot (-1)^{\text{dim}Z_m^{\perp}}=-1$;
\item $Z_m^{\perp}$ in $Z_n$ is a split space.
\end{itemize}

Define \(G_n\) and \(G_m\) to be the connected components of the isometry groups of
\(Z_n\) and \(Z_m\), respectively. We assume that \(G_n\) and \(G_m\) are both quasi-split. For \(k = n,m\) and \(R = F\), \(\mathbb{A}\), or \(F_v\), define
\[
\widetilde{G}_k(R) \coloneqq
\begin{cases}
\text{the metaplectic double cover of } G_k(R),
& \text{if } G_k(R) \text{ is symplectic},\\[4pt]
G_k(R),
& \text{if } G_k(R) \text{ is SO}.
\end{cases}
\]
There is a canonical lifting of \(G_k(F)\) into \(\widetilde{G_k}(\mathbb{A})\).
Although \(\widetilde{G_k}\) is not an algebraic group, we shall treat it as one for notational convenience, at least in the introduction.

Consider $G_m$ as a subgroup of $G_n$ that acts trivially on the orthogonal complement of $Z_m$ within $Z_n$. Let $X$ be a maximal isotropic subspace of $Z_m^{\perp}$, and let $P$ be the parabolic subgroup of $G_n$ defined as the stabilizer of a complete flag of subspaces within $X$. Write $n-m=2r+1$ if $n-m$ is odd, and $n-m=2r$ if $n-m$ is even. Denote by $U$ the unipotent radical of $P$, on which $G_m$ acts through conjugation. Write \(H_{m,r} = U \rtimes G_m\) and \(\wt{H}_{m,r} = U \rtimes \wt{G}_m\).
In the case where \(n-m = 2r+1\), there exists a generic character
\[
\psi_U \colon U(F)\backslash U(\mathbb{A}) \to \mathbb{C}^{\times},
\]
depending on \(\psi\), which is stabilized by \(G_m(\mathbb{A})\).
Moreover, this character extends to a character of
\(H_{m,r}(F)\backslash H_{m,r}(\mathbb{A})\).

When $n-m=2r$, there exists a global generic Weil representation denoted as $\nu_{\psi^{-1},Z_m}$ of $\wt{H}_{m,r}(\A)$. This representation is realized on the Schwartz space $\mc{S} = \mc{S}(Y(\A))$, where $Y$ is a Lagrangian $F$-space of $Z_m$.

For each $f \in  \mc{S}$, the associated  theta series is defined by 
\[
\Theta_{\psi^{-1},Z_m}(h ,f)=\sum_{x\in Y(F)} \left( \nu_{\psi^{-1},Z_{m}}(h)f \right) (x), 
\quad
h \in \wt{H}_{m,r}(\A).
\]
The space of theta series provides an automorphic realization of $\nu_{\psi^{-1},Z_m}$ that is invariant under the action of $H_{m,r}(F)$. Moving forward, we will denote $H_{m,r}(F) \bs H_{m,r}(\mathbb{A})$ as $[H_{m,r}]$.

Let \(\pi_1\) and \(\pi_2\) be irreducible  cuspidal automorphic representations of
\(\widetilde{G}_n(\mathbb{A})\) and \(\widetilde{G}_m(\mathbb{A})\), respectively.
When both \(G_n\) and \(G_m\) are symplectic groups, suppose that exactly one of
\(\pi_1\) and \(\pi_2\) is genuine. 

We view $H_{m,r}$ as a subgroup of $G_n$ through the mapping $(u,g) \mapsto ug$. Depending on whether $n-m$ is odd or even, we define the \textit{Bessel periods} and the \textit{Fourier--Jacobi periods} for $\pi_1,\pi_2$ as the following integrals:
\begin{itemize}

\item If $n-m$ is odd, the Bessel period for $\vi_1 \otimes \vi_2 \in \pi_1 \boxtimes \pi_2$ is defined by
\[\mathcal{B}(\vi_1\otimes \vi_2) \col \int_{ [H_{m,r}]}\vi_1(ug)\vi_2(g)\psi_U^{-1}(u)dudg.
\]

\item If $n-m$ is even, the Fourier--Jacobi period for $\vi_1 \otimes \vi_2 \otimes \phi \in \pi_1 \boxtimes \pi_2 \boxtimes \nu_{\psi^{-1},Z_m} $ is defined by 
\[
\mathcal{FJ}(\vi_1\otimes \vi_2\otimes \phi) \col \int_{ [ H_{m,r}]}\vi_1(u\wt{g})\vi_2(\wt{g})\Theta_{\psi^{-1},Z_m}\left((u,\wt{g}),\phi\right)dudg.
\]
\end{itemize}
\par
In the last integral, \(\widetilde{g}\) denotes any lift of \(g\) to \([\widetilde{G}_{m}]\).
This choice is immaterial, since exactly two of
\(\pi_1\), \(\pi_2\), and \(\nu_{\psi^{-1},Z_m}\) are genuine.

When both $\pi_1$ and $\pi_2$ are non-cuspidal, the above Bessel and Fourier--Jacobi periods may diverge, and one must introduce suitable regularizations. Concrete definitions of regularized periods were given in~\cite{IY19} and~\cite{Zyd22} when $H_{m,r}$ is reductive, and in~\cite{H24} when $H_{m,r}$ is non-reductive.
We shall use the same notation $\mc{B}$ and $\mc{FJ}$ to denote the regularized Bessel and Fourier--Jacobi periods, respectively.\\

\nosubsection{Global GGP conjecture}

\noindent We now state the global GGP conjecture as formulated in the framework of Arthur, using the language of global $A$-parameters.

\begin{con}[{\cite[Conjecture~9.1]{GGP20}}]\label{ggp1}
Let $M \times N$ be a discrete global $A$-parameter of $G_n \times G_m$, and let $\pi_1 \boxtimes \pi_2$ be an irreducible discrete automorphic representation of $G_n(\A) \times G_m(\A)$ with $A$-parameter $M \times N$. Then the following assertions hold:

\begin{enumerate}
    \item Suppose that $n-m$ is odd.
    \begin{enumerate}
        \item If the Bessel period $\mc{B}$ on $\pi_1 \boxtimes \pi_2$ is nonzero, then $(M, N)$ is a relevant pair.
        \item Assume that $(M, N)$ is a relevant pair. Then the following are equivalent:
        \begin{enumerate}
            \item[(1)] There exists an irreducible discrete automorphic representation $ \pi_1' \boxtimes \pi_2' \in \Pi_{M \times N}^{\mathrm{rel}}$ such that $\mc{B}$ on $\pi_1' \boxtimes \pi_2'$ is nonzero.
            \item[(2)] The extended central $L$-value $L(s, M, N) \big|_{s=0}$ is nonzero.
        \end{enumerate}
    \end{enumerate}

    \item Suppose that $n-m$ is even.
    \begin{enumerate}
        \item If the Fourier--Jacobi period $\mc{FJ}$ on $\pi_1 \boxtimes \pi_2 \boxtimes \nu_{\psi^{-1}, Z_m}$ is nonzero, then $(M, N)$ is a relevant pair.
        \item Assume that $(M, N)$ is a relevant pair. Then the following are equivalent:
        \begin{enumerate}
            \item[(1)] There exists an irreducible discrete automorphic representation $ \pi_1' \boxtimes \pi_2' \in \Pi_{M \times N}^{\mathrm{rel}}$ such that $\mc{FJ}$ on $\pi_1' \boxtimes \pi_2' \boxtimes \nu_{\psi^{-1}, Z_m}$ is nonzero.
            \item[(2)] The extended central $L$-value $L(s, M, N) \big|_{s=0}$ is nonzero.
        \end{enumerate}
    \end{enumerate}
\end{enumerate}
\end{con}

For the precise definitions of global discrete $A$-parameters, relevant pairs, and global relevant $A$-packets $\Pi_{M \times N}^{\mathrm{rel}}$, we refer the reader to Sections~\ref{sec:A-param} and~\ref{relev}. We remark that when the relevant group is metaplectic, the associated $A$-parameter intrinsically depends on the choice of the additive character $\psi$. 

The extended central $L$-value $L(s,M,N)\big|_{s=0}$ appearing in parts (i)-(b) and (ii)-(b) of the conjecture may be viewed as a natural generalization of the standard central $L$-value $L(1/2, \pi_1 \times \pi_2)$; see \eqref{clv} for its precise formulation. Furthermore, the local Gan--Gross--Prasad conjecture predicts the uniqueness of the distinguished representation $\pi_1' \boxtimes \pi_2' \in \Pi_{M \times N}^{\mathrm{rel}}$ satisfying the non-vanishing period condition in the Conjecture.
\nosubsection{Progress on tempered cases}

\noindent 
When $M$ and $N$ are tempered $A$-parameters, they form a relevant pair, and hence Conjecture~\ref{ggp1} recovers the tempered GGP conjecture proposed in~\cite{GGP12} as a special case.
In this tempered setting, substantial progress has been made over the past decade.
Most notably, for unitary groups the conjecture has now been completely resolved via the relative trace formula method.
See~\cite{BPLZZ21,BPCZ22,BPC25,Zha14a,Zha14b} for the Hermitian unitary cases, and~\cite{Xu14,Xu16,BLX24} for the skew-Hermitian unitary cases.

By contrast, it remains unclear how to apply the relative trace formula approach to classical groups other than unitary groups.
Nevertheless, for such groups one direction of the conjecture in the tempered setting is known, by relating the relevant period integrals to regularized periods of certain residual Eisenstein series; see~\cite{GJR04,JZ20, Yam18,IY19, H}.
Although the converse direction is not established in general, it has been proved in the case
$G_n \times G_m = \SO_{2n+1} \times \SO_2$.
See~\cite{JS07} when $\SO_2$ is split, and~\cite{FM17,FM21,FM24,JZ20} when $\SO_2$ is anisotropic.
\nosubsection{Main theorem}

\noindent 
While many significant results have been established for the tempered cases of the Gan--Gross--Prasad (GGP) conjecture, much less is known regarding the non-tempered cases. For corank $1$ cases, specifically when $(G_n \times G_m) = (\SO_5 \times \SO_4)$, Gan, Gourevitch, and Szpruch (\cite{GG09, GS15}) investigated certain non-tempered settings where the relevant Arthur parameter $M \times N$ of $G_n \times G_m$ is given by
\[
(M, N) = (M_1 \oplus [1], N_0) \quad \text{and} \quad (M_2 \boxtimes [1], N_0).
\]
Here, $M_1$ and $M_2$ are tempered $A$-parameters, and $N_0$ is an orthogonal tempered $A$-parameter. These results were obtained by employing the seesaw identity within the framework of theta correspondence. 

In the context of unitary groups, for the pair $(G_n \times G_m) = (\U_3 \times \U_2)$, Gelbart, Rogawski, and Soudry \cite{GRoS97} established a relationship between Bessel periods and the pole of a certain $L$-function at $s=1$. As noted in \cite[Theorem~11.7]{GGP20}, this result implies the non-tempered GGP conjecture for the case where the $A$-parameter $M \times N$ is given by
\[
(M, N) = (\mathbb{I}_{\GL_1} \oplus M_0, [1]),
\]
where $M_0$ is a tempered $A$-parameter.

For groups of higher corank, however, there have been no relevant results to the best of the authors' knowledge. In general, higher corank cases present significant technical challenges because the period integrals involve integration over large unipotent subgroups, making the analysis considerably more complex than in the corank $0$ or $1$ instances. 

In this paper, we establish the conjecture for the following pairs of quasi-split groups under specific non-tempered settings:
\[
(G_n, G_m) \in \{ (\SO_{2n}, \SO_3), (\Mp_{2n}, \Sp_2), (\Sp_{2n}, \Mp_2) \}.
\]
Notably, our results encompass cases of arbitrarily high corank, significantly extending the scope of the conjecture beyond previously known instances. Moreover, while all existing literature is restricted to scenarios where only one of the parameters $(M,N)$ is non-tempered, we present the first treatment of the case where both parameters $(M,N)$ are non-tempered for the pair $(G_n, G_m) = (\Sp_{2n}, \Mp_2)$.

We state our results more precisely in the following theorem. For $d \in F^\times$, let $\chi_d$ denote the quadratic character of $F^{\times} \backslash \A^{\times}$ associated with the quadratic extension $F(\sqrt{d}) / F$.
\begin{thm}\label{thm:main}
Let $(G, H)$ be one of the following pairs of quasi-split groups:
\[
(G, H) \in \{ (\SO_{2n}, \SO_3), (\Sp_{2n}, \Mp_2), (\Mp_{2n}, \Sp_2) \}.
\]
We assume that $n \ge 2$ when $G = \SO_{2n}$, and $n \ge 1$ when $G \in \{ \Sp_{2n}, \Mp_{2n} \}$. 

Let $M \times N_0$ be a global discrete $A$-parameter for $G \times H$, where $N_0$ is explicitly defined as
\[
N_0 = 
\begin{cases} 
[1], & \text{when } H=\SO_3, \\ 
\chi_d \boxtimes [1], & \text{when } H=\Mp_2, \\ 
[2], & \text{when } H=\Sp_2. 
\end{cases}
\] Then the following assertions hold:
\begin{enumerate}
    \item Let $\pi_1 \boxtimes \pi_2 \in \Pi_{M \times N_0}^{\mathrm{rel}}$. Assume that $\pi_1$ is cuspidal and locally generic at some finite place $v$ of $F$. Suppose that the associated period is non-vanishing (i.e., $\mc{B}$ is non-zero on $\pi_1 \boxtimes \pi_2$ when $H=\SO_3$, or $\mc{FJ}$ is non-zero on $\pi_1 \boxtimes \pi_2 \boxtimes \nu_{\psi^{-1},Z_2}$ when $H \in \{\Sp_2, \Mp_2\}$). Then $(M,N_0)$ is a relevant pair. Furthermore, $M$ is a tempered $A$-parameter when $H \in \{\SO_3, \Mp_2\}$, and $M = M' \oplus [1]$ for some tempered $A$-parameter $M'$ of $\SO_{2n-1}$ when $H=\Sp_2$.
    
    \item Assume that $M$ is tempered when $H \in \{\SO_3, \Mp_2\}$, and $M = M' \oplus [1]$ for some tempered $A$-parameter $M'$ of $\SO_{2n-1}$ when $H=\Sp_2$. Then the following conditions are equivalent:
    \begin{enumerate}
        \item There exists a representation $ \pi_1' \boxtimes \pi_2' \in \Pi_{M \times N_0}^{\mathrm{rel}}$ such that the associated period is non-vanishing (i.e., $\mc{B}$ is nonzero on $\pi_1' \boxtimes \pi_2'$ when $H=\SO_3$, or $\mc{FJ}$ is nonzero on $\pi_1' \boxtimes \pi_2' \boxtimes \nu_{\psi^{-1},Z_2}$ when $H \in \{\Sp_2, \Mp_2\}$).
        \item The central $L$-value satisfies $L(s, M, N_0) \big|_{s=0} \neq 0$.
    \end{enumerate}
\end{enumerate}
\end{thm}
We summarize part (i) of the theorem in the following table:\bb

\begin{center}
\renewcommand{\arraystretch}{1.2}
\begin{tabular}{c|c|c|c}
\hline
$(G,H)$ & $N_0$ & \textbf{Relevant period} & \textbf{Condition on $M$} \\
\hline
$(\SO_{2n}, \SO_3)$ & $[1]$ & Bessel &  $M$ is tempered \\
$(\Sp_{2n}, \Mp_2)$ & $\chi_d \boxtimes [1]$ & Fourier--Jacobi & $M$ is tempered \\
$(\Mp_{2n}, \Sp_2)$ & $[2]$ & Fourier--Jacobi & $ M=M' \oplus [1]$ \\
\hline
\end{tabular}
\end{center}
\bb
For the precise definition of tempered $A$-parameters, we refer the reader to Section~\ref{subsec:discrete-A}. It should be emphasized that the $A$-parameter $N_0$ is strictly non-tempered in all three cases, and the parameter $M$ is additionally non-tempered when $H=\Sp_2$.

The assumption of local genericity at a single finite place is used to rule out any prior occurrences in the theta tower via the local vanishing theorem for twisted Jacquet modules. We further remark that although the general non-tempered GGP conjecture requires regularized periods, the specific cases treated in our main theorem involve cuspidal representations on the larger group. Thus, all period integrals appearing in the theorem are absolutely convergent.

Taken together, Theorem~\ref{thm:main} establishes a complete verification of the global non-tempered GGP conjecture for these specific pairs of classical groups of higher corank.

\nosubsection{The role of non-tempered representations and Arthur packets}

\noindent From the perspective of Arthur's framework, the non-tempered case of the GGP conjecture is of intrinsic importance. Global $A$-parameters are conjecturally defined as homomorphisms
\[
\psi \colon \mathcal{L}_F \times \SL_2(\CC) \longrightarrow {}^LG,
\]
where $\mathcal{L}_F$ denotes the hypothetical Langlands group. The presence of a non-trivial $\SL_2(\CC)$-factor is precisely what characterizes non-tempered Arthur parameters, distinguishing them from tempered ones. Consequently, non-tempered representations form an essential and natural part of global Arthur packets, rather than being pathological exceptions.

Within this framework, many automorphic representations arising as CAP representations, the residual spectrum, or via endoscopic transfer and global theta correspondence are expected---and in many cases known---to be non-tempered. In particular, the images of global theta lifts typically carry non-trivial $\SL_2(\CC)$-components in their associated $A$-parameters. Therefore, any attempt to fully understand the structure of Arthur packets through period integrals necessarily requires addressing the non-tempered setting.

The general GGP conjecture formulated in \cite{GGP20} reflects this perspective. It extends the original generic GGP conjecture \cite{GGP12} by replacing the central $L$-values of specific representations with a canonical $L$-value condition defined entirely in terms of global $A$-parameters. This reformulation highlights the fact that the relationship between periods and $L$-functions is fundamentally governed by the internal structure of Arthur packets---a structure whose interactions are significantly more complex and delicate in the non-tempered setting where generic representations typically do not even exist.

The non-tempered cases treated in this paper provide concrete examples in which one can explicitly trace how the $\SL_2(\CC)$-component of an $A$-parameter influences both the analytic behavior of the associated $L$-functions and the non-vanishing of special periods. We emphasize, however, that for the three group pairs considered here, one of the groups is of relatively small rank, and the corresponding non-tempered $A$-parameters involve $\SL_2(\CC)$-factors of limited dimension. As a result, our findings do not capture all possible phenomena that may arise in the full non-tempered GGP conjecture.

Nevertheless, these cases serve as natural and highly accessible testing grounds for the conjecture. They illustrate, in a concrete setting, how Arthur's $\SL_2(\CC)$-component enters into the delicate interplay between periods, theta correspondence, and special values of $L$-functions, and demonstrate how the general framework of \cite{GGP20} fits perfectly into Arthur's broader classification of automorphic representations.

\nosubsection{Strategy of proof: theta correspondence and Rankin--Selberg methods}

\noindent For clarity of exposition, in this subsection we explain the strategy of the proof only for the quasi-split group pair $(\SO_{2n}, \SO_3)$. This case already contains all the essential conceptual ingredients of the argument. The other cases treated in this paper follow the same general pattern, with appropriate technical modifications. To streamline our discussion in this introduction, we temporarily use the simplified notation $\Theta_{n,m}^{0}$ for the global theta lifting from $\SO_{2n}(\A)$ to $\Sp_{2m}(\A)$, and $\Theta_{n,m}^{1}$ for the global theta lifting from $\Sp_{2n}(\A)$ to $\SO_{2m}(\A)$.

\medskip

\noindent
\textbf{Step 1: From Bessel periods to theta lifts.}\\
Let $\pi$ be an irreducible cuspidal automorphic representation of $\SO_{2n}(\A)$.
The first step is to relate the non-vanishing of the special Bessel period
$\mathcal{B}$ on $\pi$ to the non-vanishing of a specific global theta lift.
By computing Whittaker--Fourier coefficients of global theta lifts via the standard
model of the Weil representation, we prove that
\[
\mathcal{B} \text{ is } \text{non-vanishing on } \pi
\quad \Longleftrightarrow \quad
\Theta_{n,n-1}^0(\pi) \neq 0 \text{ and generic}.
\]
This establishes a precise bridge between period integrals and the behavior of theta lifts
at the first occurrence range.

\medskip

\noindent
\textbf{Step 2: The first occurrence index and $A$-parameters.}\\
Once $\Theta_{n,n-1}^0(\pi)$ is known to be nonzero and generic, Rallis's tower property and Kudla's supercuspidal theorem allow us to recover strong information on the global $A$-parameter $M$
associated with $\pi$. The assumption that $\pi$ is locally generic at at least one finite place is used to rule out degenerate low-rank phenomena and to ensure that
the first occurrence index equals $n-1$. This yields part~(i) of Theorem~\ref{thm:main} in the case $(\SO_{2n},\SO_3)$.

\medskip

\noindent
\textbf{Step 3: From non-vanishing periods to poles of $L$-functions.}\\
Assume that $\mathcal{B} \text{ is } \text{non-vanishing on } \pi$.
Then, by Step~1, the theta lift $\Theta_{n,n-1}^0(\pi)$ is nonzero.
Applying the Rallis inner product formula in the first occurrence range,
together with the theory of local doubling zeta integrals,
we deduce that the standard $L$-function $L(s,\pi)$ has a pole at $s=1$.
Since $L(s,\pi)=L(s,M)$, this implies
\[
L(s,M)\ \text{has a pole at } s=1
\quad \Longrightarrow \quad
L(s,M,N_0)\big|_{s=0} \neq 0,
\]
which establishes the $(a)\Rightarrow(b)$ direction of part~(ii) of
Theorem~\ref{thm:main}.

\medskip

\noindent
\textbf{Step 4: Generic representations and an equivalence theorem.}\\
The converse implication $(a)\Leftarrow(b)$ in part~(ii) of
Theorem~\ref{thm:main} is substantially more delicate.
Its key input is Theorem~\ref{t2} which clarifies  the interaction between theta correspondence, periods and $L$-functions.

More precisely, we prove that for an irreducible \emph{generic} cuspidal
automorphic representation $\pi$ of $\SO_{2n}(\A)$,
the following five conditions are equivalent:
\begin{enumerate}
    \item the completed $L$-function $L(s,\pi)$ has a pole at $s=1$;
    \item the partial $L$-function $L^S(s,\pi)$ has a pole at $s=1$;
    \item the Bessel period $\mathcal{B}$ on $\pi$ is non-vanishing;
    \item the theta lift $\Theta_{n,n-1}^0(\pi)$ is nonzero and generic;
    \item the theta lift $\Theta_{n,n-1}^0(\pi)$ is nonzero.
\end{enumerate}

The implications
\[
\text{(iii)} \Rightarrow \text{(iv)} \Rightarrow \text{(v)} \Rightarrow \text{(i)}
\]
follow from the argument in \textbf{Step~1}, together with the Rallis inner product
formula~\cite{Yam14} in the first term range.

The remaining implications
\[
\text{(i)} \Rightarrow \text{(ii)} \Rightarrow \text{(iii)}
\]
are precisely where the genericity assumption on $\pi$ is required.

To establish the implication (i)$\Rightarrow$(ii), we use the trivial Ramanujan
bound of the generic representations to control the analytic behavior of local doubling zeta integrals. This allows us to deduce (ii) from (i).

To prove the implication (ii) $\Rightarrow$ (iii), we require a Rankin--Selberg integral theory for quasi-split $\SO_{2n}$. Although a Rankin--Selberg theory for $\SO_{2n} \times \GL_1$ (or more generally $\SO_{2n} \times \GL_l$) was developed in \cite{Kap15}, it does not encompass generic representations of quasi-split $\SO_{2n}$ associated with all generic characters; rather, it is restricted to a specific generic character. To accommodate all generic representations of $\SO_{2n}$, we extend this theory by introducing the notion of type $(d, c)$ for $\SO_{2n}$ (see Remark~\ref{kap}).

Specifically, we define a more general Rankin--Selberg integral \eqref{bsi}, prove a basic identity (Proposition~\ref{p1}) relating it to a Whittaker-model zeta integral, and establish its Euler product decomposition. This construction enables us to translate the analytic information regarding the pole of the $L$-function into the non-vanishing of the corresponding period, provided the representation is generic. Consequently, this yields the desired implication (ii) $\Rightarrow$ (iii).

\medskip

\noindent
\textbf{Step 5: Descent and the $(a)\Leftarrow(b)$ direction.}\\
Assume now that $L(s,M,N_0)\big|_{s=0} \neq 0$.
Equivalently, the standard $L$-function $L(s,M)$ has a pole at $s=1$.
At this stage, we do not yet have a specific automorphic representation
on which to apply the equivalence theorem in \textbf{Step~4}.

To overcome this difficulty, we invoke the global descent theory
of Ginzburg--Rallis--Soudry.
Since $M$ is a global tempered $A$-parameter, descent produces
an irreducible cuspidal \emph{generic} automorphic representation
$\pi' \in \Pi_M$ of $\SO_{2n}(\A)$ whose associated $A$-parameter is $M$.
Applying the equivalence theorem in \textbf{Step~4} to $\pi'$, the pole of
$L(s,\pi')=L(s,M)$ at $s=1$ implies that the Bessel period
$\mathcal{B}$ is non-vanishing on $\pi'$.
This completes the proof of the $(a)\Leftarrow(b)$ direction of part~(ii)
of Theorem~\ref{thm:main} in the case $(\SO_{2n},\SO_3)$.

\medskip

\medskip
\begin{rem}
The cases $(\Sp_{2n},\Mp_2)$ and $(\Mp_{2n},\Sp_2)$ follow the same overall conceptual
strategy as the orthogonal case discussed above.
However, their implementation in the Fourier--Jacobi setting is technically
 more involved.
In contrast to the Bessel case, the analysis of Fourier--Jacobi periods
requires a more delicate use of mixed models of the Weil representation,
together with a careful treatment of nontrivial metaplectic coverings. In particular, the computation of Whittaker--Fourier coefficients of global
theta lifts in these cases constitutes one of the main technical contributions of this paper.
\end{rem}
\nosubsection{A digression: genericity and the tower property of theta lifting}

\noindent
Before concluding the introduction, we isolate a collection of results that arise naturally in the course of our analysis but are of independent interest.

While the global theta correspondence is traditionally structured around Rallis' tower property concerning cuspidality, the analogous persistence of genericity within these towers remains incompletely understood, even for split reductive dual pairs. Although partial results regarding this phenomenon have appeared sporadically in the literature (e.g., \cite{Fu95, GRS97}), a comprehensive and systematic treatment for quasi-split dual pairs has been lacking to the best of our knowledge.
In this paper, we establish a precise and uniform description of how genericity behaves along the theta tower for quasi-split orthogonal, symplectic, and metaplectic groups. These results are summarized in Theorem~\ref{pp} below.

\begin{thm}[Theorems~\ref{a1} and~\ref{p6}, Theorems~\ref{b10} and~\ref{q1}]\label{pp}
Let $G_n$ be either $\SO_{2n}$ or $\Sp_{2n}$, and let $\pi$ be a generic cuspidal automorphic representation of $G_n(\A)$. Let $l(\pi)$ denote the first occurrence index of $\pi$ in the tower of theta liftings. Then for $\ve \in \{0,1\}$, the following statements hold:
\begin{enumerate}
    \item $n+\ve - 1 \le l(\pi) \le n+\ve$;
    \item $\Theta_{n,l(\pi)}^{\ve}(\pi)$ is generic;
    \item $\Theta_{n,l(\pi)+1}^{\ve}(\pi)$ is generic if and only if $l(\pi) = n+\ve - 1$;
    \item $\Theta_{n,i}^{\ve}(\pi)$ is non-generic for all $i \ge l(\pi) + 2$.\footnote{We thank Hiraku Atobe for pointing out that property (iv) can alternatively be deduced from the explicit description of the local theta correspondence for unramified representations, which follows from the cuspidal support theorem.}
\end{enumerate}
\end{thm}

While properties (i), (iii), and (iv) are likely well known to experts (see, e.g., \cite{GRS97}, \cite{Fu95}), we believe that (ii) is a novel and highly non-trivial property established in this work. Previously, this was known only implicitly for cases where $\pi$ is strongly generic for split $\SO_{2n}$ \cite{GRS97} (cf. \cite{Fu95}). Our proof relies on explicit computations of the Whittaker--Fourier coefficients of global theta lifts—carried out within suitable mixed models of the Weil representation—in conjunction with the Rallis inner product formula, local $L$-factor theory, and Rankin--Selberg integrals. The investigation of these properties served as the original starting point for our research. We hope that furnishing such a systematic reference will be of independent interest beyond the immediate scope of the present work.
\bb

\noindent Before closing the introduction, we briefly highlight several features of the paper.
\begin{enumerate}
    \item In the classical Langlands program, the group $\SO_n$ is typically
preferred over $\O_n$, since $\O_n$ is a disconnected linear algebraic
group and therefore lies outside the usual scope of the Langlands
framework, which focuses on connected reductive groups.
However, in the context of the theta correspondence, it is often more
natural to work with $\O_n$, as it arises naturally in the theory of
reductive dual pairs.
For this reason, and with a view toward possible future applications,
we state and prove several key results for both $\SO_n$ and $\O_n$.

Meanwhile, as emphasized in~\cite{AG17}, results for $\SO_n$ and $\O_n$
can often be transferred from one group to the other due to their close
structural relationship, although certain subtleties arise in passing
between $\O_n$ and $\SO_n$.

Accordingly, we first establish the relevant theorems and propositions for one of these groups, and then transfer them to the other via a suitable technical argument.
\item Since the reductive dual pairs $(\O_{2n},\Sp_{2m})$ and $(\O_{2n+1},\Mp_{2m})$ exhibit strong structural similarities, we treat them in a uniform manner by introducing the notation $\varepsilon \in \{\pm 1\}$, following the treatment of classical groups in~\cite{GI14, GQT14, Yam14}.
\item Using the machinery developed in this paper, one can also treat a case of the GGP conjecture for the group pair
$(\SO_{2n+1},\SO_2)$.
However, this situation falls into the tempered setting.
Since our primary focus is on the non-tempered case of the GGP conjecture, we do not pursue this direction here.
Moreover, the tempered case has already been studied extensively in the
literature.
See~\cite{JS07} when $\SO_2$ is split, and
\cite{FM17,FM21,FM24,JZ20} when $\SO_2$ is anisotropic.
\end{enumerate}

While our focus in this paper is on orthogonal, symplectic, and metaplectic dual pairs, the methods developed here are expected to extend to quasi-split unitary groups as well. Once established for unitary groups, this framework will encompass the results in \cite{GRoS97} as a special case. We note that some related computations for unitary groups have already appeared in the literature (e.g., \cite{Gr04, Mo24}). Although the unitary case is conceptually analogous, its inclusion would necessitate considerably more intricate notations and conventions. For the sake of clarity and exposition, we therefore confine our discussion in this paper to the orthogonal, symplectic, and metaplectic settings.

After completing the first draft of this paper, we became aware of the
work of B.~P.~J.~Wang~\cite{Wa25}, who developed a general theory relating
periods of dual reductive groups under the theta correspondence. Some of his results appear to overlap with certain computations
presented in Sections~\ref{sec:4} and~\ref{sec:6} of this paper.
However, it is not clear to us whether Theorems~\ref{a1} and~\ref{b10}
are covered by his framework, as our results involve more refined information, such as the precise type of the quasi-split orthogonal groups and generic characters appearing in the statements. It is also worth noting that K.~Morimoto \cite{Mo14} carried out a similar computation in the setting of the dual pair $(\mathrm{GSp}(4), \mathrm{GSO}(6))$.
\bigbreak

The paper is organized as follows.
In \S\ref{sec2} we fix notation and recall the background on classical groups, metaplectic covers,
and the relevant periods.
In \S\ref{sec:3}--\S\ref{sec:7} we develop the theta correspondence and Rankin--Selberg-theoretic machinery,
including the genericity tower property and the modified Rankin--Selberg integrals.
Finally, in \S\ref{sec8} we assemble these results to prove the non-tempered GGP cases described in Theorem~\ref{thm:main}.

\section{Notation and Preliminaries}\label{sec2}

In this section, we establish the basic notation, conventions, and background material utilized throughout this paper. 
In Section~\ref{subsec:2-1}, we fix general notation, while Sections~\ref{ortho} and~\ref{symp} review the definitions and basic properties of orthogonal, symplectic, and metaplectic groups. 
Section~\ref{geno} discusses the distinction between genuine and non-genuine representations, which is fundamental to the representation theory of metaplectic groups.

In Section~\ref{atr}, we introduce the notion of almost tempered representations, a class of representations that encompasses the tempered ones. 
Section~\ref{ltn} provides a brief review of the partial and completed $L$-functions associated with automorphic representations of classical groups. 
Section~\ref{heis} summarizes the construction of the Weil representation and its role in the theory of theta series. 
Finally, in Sections~\ref{spb} and~\ref{spf}, we recall the definitions of the Bessel and Fourier--Jacobi periods, respectively, which are the central objects of study in the present work.
\subsection{Basic notation} \label{subsec:2-1}
\begin{itemize}
    \item $F$: a global field of characteristic not equal to $2$.
    \item $\A$: the ring of adeles of $F$.
    \item $v$: a place of $F$.
    \item $F_v$: the completion of $F$ at $v$.
    \item $\OO_v$: the ring of integers in $F_v$.
    \item $|\cdot|$: the normalized absolute value on $F$ or $F_v$.
    \item $G(K)$: the group of $K$-points of an algebraic group $G$ defined over $F$ for an $F$-algebra $K$. (We often write $G$ instead of $G(K)$ when the context is clear.)
    \item $\Irr(G)$: the set of isomorphism classes of irreducible smooth representations of $G(F_v)$.
    \item $\pi^{\vee}$: the contragredient (smooth) dual of a representation $\pi$.
    \item $N_{E/F}, \Tr_{E/F}$: the norm and trace maps for a quadratic extension $E/F$.
    \item $\Ind_{P}^G$: the normalized  induction from a parabolic subgroup $P$ to $G$.
    \item $\ind_{P}^G$: the unnormalized compactly supported induction from a parabolic subgroup $P$ to $G$.
    \item $\mc{S}(X)$: the Bruhat--Schwartz space on a topological space $X$.
    \item $\mc{A}(G)$: the space of cuspidal automorphic forms on $G(\A)$.
    \item $\mc{A}^2(G)$: the space of square integrable automorphic forms on $G(\A)$.
    \item $\mf{R}$: the right translation action of $G(\A)$ on a function space defined on $G(\A)$.
    \item $[G]$: the automorphic quotient $G(F) \backslash G(\A)$.
    \item $\mu_n$: the algebraic group of $n$-th roots of unity.
    \item $\mathbb{G}_a, \mathbb{G}_m$: the additive and multiplicative group schemes over $F$, respectively.
    \item $M_{n \times m}$: the space of $n \times m$ matrices over $F$.
    \item $\mathrm{Id}_n$: the $n \times n$ identity matrix.
    \item $\mathrm{Id}$: the identity operator on a given vector space.
    \item $\wt{w}_k$: the anti-diagonal matrix in $M_{k \times k}$ with all anti-diagonal entries equal to $1$.
    \item $a^{\T}$: the transpose of a matrix $a$.
    \item $a^{\star}$: the involution defined by $a^{\star} = \wt{w}_{n} (a^{\T})^{-1} \wt{w}_{n}$ for $a \in \GL_n$.
    \item $Z_k$: the standard maximal unipotent subgroup of $\GL_k$.
    \item $S_k$: the subspace $\{s \in M_{k \times k} \mid \wt{w}_k s^{\T} \wt{w}_k = -s\}$.
    \item $\1$: the identity element of a group.
    \item $\II$: the trivial character.
\end{itemize}

We denote by $\chi_d$ the quadratic character of $F_v^\times$ (resp.~$\A^\times / F^\times$) associated to the quadratic extension $F_v(\sqrt{d}) / F_v$ (resp.~$F(\sqrt{d}) / F$) for $d \in F^\times$. We fix a nontrivial unitary additive character $\psi = \otimes_v \psi_v$ of $F \backslash \A$. For $\lambda \in F^\times$ (resp.~$F_v^\times$), we define $\psi_\lambda(x) \coloneqq \psi(\lambda x)$ (resp.~$\psi_{v,\lambda}(x) \coloneqq \psi_v(\lambda x)$) for $x \in \A$ (resp.~$x \in F_v$).

We fix the self-dual Haar measure $dx_v$ on $F_v$ with respect to $\psi_v$. Throughout the paper, for a connected algebraic group $H$ defined over $F$, we use the Tamagawa measure $dh$ on $H(\A)$, which is decomposed into the product of local Haar measures $dh_v$ on $H(F_v)$. These local measures are determined by the aforementioned self-dual measure $dx_v$ on the coordinates. In addition, we adopt the following normalization conventions:
\begin{itemize}
    \item For any unipotent algebraic group $U$ defined over $F$, the Tamagawa measure $du$ on $U(\A)$ is normalized such that $\text{vol}(U(F) \backslash U(\A), du) = 1$.
    \item For any maximal compact subgroup $K_v$ of $H(F_v)$, the local Haar measure $dk_v$ is normalized so that $\text{vol}(K_v, dk_v) = 1$.
    \item For the finite group $\mu_2$, we let $dt = \prod_v dt_v$ be the measure on $\mu_2(\A)$, where each local measure $dt_v$ on $\mu_2(F_v)$ is normalized so that $\text{vol}(\mu_2(F_v), dt_v) = 1$.
\end{itemize}
\subsection{Orthogonal groups} \label{ortho}

Many of the concepts and definitions discussed in Sections~\ref{ortho} and~\ref{symp} apply uniformly to both the local and global settings. To streamline the exposition, we let $K$ denote either the global field $F$ or its local completion $F_v$ at a place $v$.

Let $\Ha$ be the hyperbolic plane over $K$, the split quadratic space of dimension $2$. For $k \ge 1$, we set $\overline{V_{k}} \coloneqq \Ha^{\oplus k}$ and $\overline{V_{0}} \coloneqq 0$. Let $\{e_1, \dots, e_k\}$ and $\{e_1^*, \dots, e_k^*\}$ be bases of the isotropic subspaces of $\overline{V_k}$ satisfying
\[
(e_i, e_j)_{\overline{V_k}} = (e_i^*, e_j^*)_{\overline{V_k}} = 0, \quad (e_i, e_j^*)_{\overline{V_k}} = \delta_{ij}.
\]
For $1 \le i \le k$, let $X_i = \mathrm{Span}_K \{e_1, \dots, e_i\}$ and $X_i^* = \mathrm{Span}_K \{e_1^*, \dots, e_i^*\}$, so that $\overline{V_i} = X_i \oplus X_i^*$.

For arbitrary $c, d \in K^{\times}$, let $V_{c,d} \coloneqq K[X]/(X^2-d)$, and let $\e$ be the involution on $V_{c,d}$ induced by $a+bX \mapsto a-bX$. We denote the images of $1, X \in K[X]$ in $V_{c,d}$ by $e$ and $e'$, respectively. We regard $V_{c,d}$ as a $2$-dimensional vector space over $K$ equipped with the symmetric bilinear form
\[
(\alpha, \beta)_{V_{c,d}} \coloneqq c \cdot \Tr_{V_{c,d}/K}(\alpha \cdot \e(\beta)).
\]
Let $V_d$ be a $1$-dimensional $K$-vector space equipped with the bilinear form $(\alpha, \beta)_{V_d} \coloneqq 2d \alpha \beta$. To facilitate a uniform treatment, we use the same notation $e$ and $e'$ to refer to the elements $1$ and $0$ in $V_d$, respectively. Note that the discriminants of both $V_{c, d}$ and $V_d$ are given by $d \pmod{K^{\times 2}}$. 

For a $(2n+\ve)$-dimensional orthogonal space $V$ over $K$, let $\G_n^{\ve}(V) \coloneqq \O(V)$ be the orthogonal group and $\H_n^{\ve}(V) \coloneqq \SO(V)$ be its special orthogonal subgroup. 

Henceforth, we fix $c, d \in K^{\times}$ and for $\ve \in \{0, 1\}$, we set
\[
K_{\ve} = \begin{cases} N_{K(\sqrt{d})/K}(K(\sqrt{d})^\times), & \text{if } \ve=0, \\ K^{\times 2}, & \text{if } \ve=1, \end{cases} \quad \lambda_{\ve} = \begin{cases} c \pmod{K_0}, & \text{if } \ve=0, \\ d \pmod{K_1}, & \text{if } \ve=1, \end{cases}
\]
and define the $(2n+\ve)$-dimensional quadratic space $V_n^{\ve}$ as
\[
V^{\ve} \coloneqq \begin{cases} V_{c,d}, & \text{if } \ve=0, \\ V_d, & \text{if } \ve=1, \end{cases} \quad V_n^{\ve} \coloneqq V^{\ve} \oplus \overline{V_{n-1+\ve}}.
\]
When $\ve=0$, we say that $V_n^0$ has type $(d, c)$. The collection $\{V_r^{\ve} \mid r \ge 0\}$ forms a Witt tower of orthogonal spaces. We assume $V = V_n^{\ve}$ and denote $\G_n^{\ve}(V)$ simply by $\G_n^{\ve}$. While not explicitly shown in the notation, the group $\G_n^{\ve}$ depends on the choice of $c, d \in K^{\times}$.

Consider a flag of isotropic subspaces $X_{k_1} \subset X_{k_1+k_2} \subset \dots \subset X_{k_1+\dots+k_r} \subset V_n^{\ve}$. Its stabilizer is a parabolic subgroup $\P$ of $\G_n^{\ve}$ with Levi factor
\[
\M \simeq \GL_{k_1} \times \dots \times \GL_{k_r} \times \G_{n-(k_1+\dots+k_r)}^{\ve}.
\]
We denote by $\M_{n,r}^{\ve}$ and $\U_{n,r}^{\ve}$ the Levi factor and the unipotent radical, respectively, of the parabolic subgroup associated with the case $k_1 = \dots = k_r = 1$.  When $r=0$, we define $\M_{n,0}^{\ve}=\G_n^{\ve}$ and $\U_{n,0}^{\ve}=
\{\1\}$, the trivial group. Let $\B_{n-1+\ve}^{\ve} = \T_{n-1+\ve}^{\ve} \U_{n-1+\ve}^{\ve}$ be the Borel subgroup of $\G_n^{\ve}$, where $\T_{n-1+\ve}^{\ve} \coloneqq \M_{n, n-1+\ve}^{\ve}$. For each place $v$ of $F$, let $\K_{\ve,v}$ be the standard maximal compact subgroup of $\G_n^{\ve}(F_v)$, and set $\K_{\ve} = \prod_v \K_{\ve,v} \subset \G_n^{\ve}(\A)$.

\subsection{Symplectic and metaplectic groups} \label{symp}

Let $\Ha'$ be the symplectic hyperbolic plane over $K$, i.e., the split symplectic space of dimension $2$. For $r \ge 0$, let $W_r = (\Ha')^{\oplus r}$ be the $2r$-dimensional symplectic space equipped with a non-degenerate symplectic form $\la \cdot, \cdot \ra_{W_r}$. The collection $\{W_r \mid r \ge 0\}$ forms a Witt tower of symplectic spaces. 
Let $\{f_1, \dots, f_m, f_1^*, \dots, f_m^*\}$ be an ordered basis of $W_m$ satisfying
\[
\pair{f_i, f_j}_{W_m} = \pair{f_i^*, f_j^*}_{W_m} = 0, \quad \pair{f_i, f_j^*}_{W_m} = \delta_{ij}.
\]
For $1 \le k \le m$, let $Y_k = \mathrm{Span}_K \{f_1, \dots, f_k\}$ and $Y_k^* = \mathrm{Span}_K \{f_1^*, \dots, f_k^*\}$, so that $W_m = Y_m \oplus Y_m^*$. Furthermore, we set \[W_{m,k} = \mathrm{Span}_K \{f_{k+1}, \dots, f_m, f_{k+1}^*, \dots, f_m^*\},\] which yields the decomposition $W_m = Y_k \oplus W_{m,k} \oplus Y_k^*$.

Let $\J_m \coloneqq \Sp(W_m)$ be the symplectic group associated with $W_m$. Consider a flag of isotropic subspaces $Y_{k_1} \subset Y_{k_1+k_2} \subset \dots \subset Y_{k_1+\dots+k_r} \subset W_m$. Its stabilizer in $\J_m$ is a parabolic subgroup $\P' = \M' \U'$ with Levi factor
\[
\M' \simeq \GL_{k_1} \times \dots \times \GL_{k_r} \times \J_{m_0}, \quad m_0 = m - \sum_{i=1}^r k_i.
\]
We denote by $\M_{m,r}'$ and $\U_{m,r}'$ the Levi factor and the unipotent radical, respectively, associated with the case $k_1 = \dots = k_r = 1$.  When $r=0$, we define $\M_{m,0}'=\J_m$ and $\U_{m,0}'=\{\1\}$, the trivial group.  Let $\B_m' = \T_m' \U_m'$ be the Borel subgroup of $\J_m$, where $\T_m' \coloneqq \M_{m,m}'$. For each place $v$ of $F$, let $\K_v'$ be the standard maximal compact subgroup of $\J_m(F_v)$, and set $\K' = \prod_v \K_v' \subset \J_m(\A)$.

The metaplectic group $\wt{\J}_m(F_v)$ is the twofold central extension of $\J_m(F_v)$ by $\mu_2 = \{\pm 1\}$:
\[
\xymatrix{1 \ar[r] & \mu_2 \ar[r] & \wt{\J}_m(F_v) \ar[r]^-{pr_v} & \J_m(F_v) \ar[r] & 1}.
\]
Realizing the elements of $\wt{\J}_m(F_v)$ as pairs $(h', \ve') \in \J_m(F_v) \times \mu_2$, the multiplication is given by $(h_1', \ve_1')(h_2', \ve_2') = (h_1' h_2', \ve_1' \ve_2' \cdot \mf{c}(h_1', h_2'))$, where $\mf{c}$ is the Rao cocycle \cite{Rao93}. 

Let $Z \coloneqq \{ (z_v) \in \bigoplus_v \mu_2 \mid \prod_v z_v = 1 \}$. We define the adelic metaplectic group as $\wt{\J}_m(\A) \coloneqq (\prod_v' \wt{\J}_m(F_v)) / Z,$ where $\prod_v'$ is a restricted product with respect to some choice of hyperspecial subgroups of $\J_m(F_v)$ at unramified places $v$ of $F$.
Since $pr \coloneqq \prod_v pr_v$ is trivial on $\mu_2$, it yields the central extension
\[
\begin{aligned}
\xymatrix{1 \ar[r] & \mu_2 \ar[r] & \wt{\J}_m(\A) \ar[r]^-{pr} & \J_m(\A) \ar[r] & 1}.
\end{aligned}
\]
A fundamental result of Weil \cite{We64} asserts that $\J_m(F)$ splits canonically in $\wt{\J}_m(\A)$.  Although $\wt{\J}_m$ is not an algebraic group in the strict sense, we shall treat it as such, given that its $F_v$-points and $\A$-points are explicitly defined.

Let $R$ denote either $F_v$ or $\A$. If a subgroup $L \subset \J_m(R)$ splits in $\wt{\J}_m(R)$, we identify $L$ with its image under the canonical splitting. In particular, any maximal compact subgroup and any unipotent subgroup of $\J_m(R)$ splits canonically in $\wt{\J}_m(R)$ \cite[Appendix~I]{Mo95}. If $L$ does not split, its preimage $pr^{-1}(L)$ in $\wt{\J}_m(R)$ is denoted by $\wt{L}$.

The double covering of a Levi factor $\M' \simeq \GL_{k_1} \times \dots \times \GL_{k_r} \times \J_{m_0}$ is given by the amalgamated product
\[
\wt{\M}' \coloneqq \wt{\GL}_{k_1} \times_{\mu_2} \dots \times_{\mu_2} \wt{\GL}_{k_r} \times_{\mu_2} \wt{\J}_{m_0},
\]
where $\wt{\GL}_k$ is the twofold covering defined by $(g_1, \ve_1')(g_2, \ve_2') = (g_1 g_2, \ve_1' \ve_2' \cdot (\det g_1, \det g_2)_R)$ with $(\cdot, \cdot)_R$ being the Hilbert symbol. The covering of the parabolic subgroup is then defined by $\wt{\P}' = \wt{\M}' \U'$, where $\U'$ is the unipotent radical identified with its splitting image.

\subsection{Genuine and non-genuine representations}\label{geno}

Let $R$ denote either $F_v$ or $\A$. A function $f : \wt{\J}_m(R) \to \CC$ is called \textit{genuine} if $f(\ve' \cdot x) = \ve' f(x)$, and \textit{non-genuine} if $f(\ve' \cdot x) = f(x)$ for all $\ve' \in \mu_2$ and $x \in \wt{\J}_m(R)$. Similarly, a representation $(\sigma, V)$ of $\wt{\J}_m(R)$ is \textit{genuine} if $\sigma(\ve' \cdot x) = \ve' \sigma(x)$, and \textit{non-genuine} if $\sigma(\ve' \cdot x) = \sigma(x)$ for all $\ve' \in \mu_2$ and $x \in \wt{\J}_m(R)$.

By inflation via the projection $pr : \wt{\J}_m(R) \to \J_m(R)$, we regard any function (resp.~representation) of $\J_m(R)$ as a non-genuine function (resp.~representation) of $\wt{\J}_m(R)$. Under this convention, we shall work primarily with functions and representations of $\wt{\J}_m(R)$, treating those of $\J_m(R)$ as non-genuine objects.

For $m \ge 0$, we denote by $\wt{\II}$ the \textit{genuine trivial character} of $\wt{\J}_m(R)$, which is defined by 
\[
\wt{\II}((x, \ve')) = \ve' \cdot \II \quad \text{for } (x, \ve') \in \wt{\J}_m(R).
\]
Then $\wt{\II}$ is an element of $\Irr(\wt{\J}_m)$ and is genuine.

For $\sigma \in \Irr(\wt{\J}_m)$ and $\tau \in \Irr(\J_k)$, let $\wt{\ve} \in \{0, 1\}$. We define the indicator $\ve_{\sigma} \in \{0, 1\}$ and the representation $\tau^{\wt{\ve}} \in \Irr(\wt{\J}_k)$ as follows:
\begin{equation}\label{nota}
\ve_{\sigma} \coloneqq 
\begin{cases}
1, & \text{if } \sigma \text{ is genuine}, \\
0, & \text{if } \sigma \text{ is non-genuine},
\end{cases}
\qquad
\tau^{\tilde{\ve}} := \begin{cases} \tau & \text{if } \tilde{\ve}=0, \\ \tau \otimes \tilde{\mathbb{I}} & \text{if } \tilde{\ve}=1 \end{cases}
\end{equation}

\subsection{Almost tempered representations}\label{atr}

We define the notion of almost tempered representations for general linear groups, classical groups, and their metaplectic counterparts. Let $G_n$ denote $\GL_n$, $\G_n^{\ve}$, or $\J_n$. Consider a standard parabolic subgroup $P = MN$ of $G_n$ with Levi factor $M \simeq \GL_{k_1} \times \dots \times \GL_{k_l} \times G_{n_0}$, where $n = n_0 + \sum_{i=1}^l k_i$. In the case $G_n = \GL_n$, we have $M \simeq \GL_{k_1} \times \dots \times \GL_{k_l}$ with $n = \sum k_i$ and $n_0 = 0$, where $G_{n_0}$ is understood to be the trivial group.

A standard module of $G_n(F_v)$ is defined by the parabolic induction:
\[\tau_1 |\cdot|^{s_1} \times \dots \times \tau_l |\cdot|^{s_l} \rtimes \pi_0 \coloneqq \Ind_{P(F_v)}^{G_n(F_v)} \left( \tau_1 |\cdot|^{s_1} \otimes \dots \otimes \tau_l |\cdot|^{s_l} \otimes \pi_0 \right),
\]
where:
\begin{itemize}
    \item each $\tau_i$ is an irreducible tempered representation of $\GL_{k_i}(F_v)$;
    \item $\pi_0$ is an irreducible tempered representation of $G_{n_0}(F_v)$. If $G_n = \GL_n$, we take $l \ge 1$ and let $\pi_0$ be the trivial representation of the trivial group;
    \item the real exponents satisfy $s_1 > s_2 > \dots > s_l > 0$ if $G_n$ is a classical group. For $G_n = \GL_n$, the exponents $s_i$ are not necessarily positive.
\end{itemize}

For the metaplectic group $\wt{\J}_n(F_v)$, we consider genuine irreducible representations. Let $\gamma_{\psi_v}$ be the Weil factor associated with $\psi_v$, which is an eighth root of unity (cf. \cite[Appendix]{Rao93}). For $\tau \in \Irr(\GL_k)$, we define its metaplectic counterpart $\tau_{\psi_v} \in \Irr(\wt{\J}_k)$ by
\[
\tau_{\psi_v}(h', \ve') \coloneqq \ve' \cdot \gamma_{\psi_v}(\det h') \cdot \tau(h') \quad \text{for } (h', \ve') \in \wt{\J}_k(F_v).
\]
The standard modules for $\wt{\J}_n(F_v)$ are defined analogously by using $\tau_{i, \psi_v}$ in place of $\tau_i$. By convention, we shall identify representations of $\GL_{k_i}(F_v)$ with those of $\wt{\J}_{k_i}(F_v)$ and maintain the notation $\tau_1 |\cdot|^{s_1} \times \dots \times \tau_l |\cdot|^{s_l} \rtimes \pi_0$, with the understanding that the dependence on $\psi_v$ is implicit in the metaplectic case.

By the Langlands classification, every irreducible smooth representation of $G_n(F_v)$ (or $\wt{\J}_n(F_v)$) can be realized as the unique irreducible quotient of a standard module. We let $L(\tau_1 |\cdot|^{s_1}, \dots, \tau_l |\cdot|^{s_l}, \pi_0)$ denote the unique irreducible quotient of the standard module $\tau_1 |\cdot|^{s_1} \times \dots \times \tau_l |\cdot|^{s_l} \rtimes \pi_0$.
\begin{defn}
An irreducible representation $L(\tau_1 |\cdot|^{s_1}, \dots, \tau_l |\cdot|^{s_l}, \pi_0)$ of $H_n(F_v)$, where $H_n \in \{ \GL_n, \G_n^{\ve}, \J_n, \wt{\J}_n \}$, is said to be \textbf{almost tempered} if its exponents satisfy:
\begin{enumerate}
    \item $1/2 > |s_1| > |s_2| > \cdots > |s_l|$ \qquad if $H_n = \GL_n$;
    \item $1/2 > s_1 > s_2 > \cdots > s_l > 0$ \qquad if $H_n \in \{\G_n^{\ve}, \J_n\}$;
    \item $1/2 \ge s_1 > s_2 > \cdots > s_l > 0$ \qquad if $H_n = \wt{\J}_n$.
\end{enumerate}
\end{defn}

Note that if $l=0$, an irreducible tempered representation is almost tempered. For $s_1 < 1/2$, the standard module is itself irreducible, and thus an almost tempered representation coincides with its corresponding standard module.

\begin{rem}\label{Ram}
The condition $s_1 < 1/2$ is motivated by the trivial bounds toward the generalized Ramanujan conjecture. Specifically, for an irreducible generic unitary representation, the exponents of its local components are known to satisfy $s_1 < 1/2$ when $H_n \neq \wt{\J}_n$ (cf. \cite[Corollary~10.1]{CKPS04}). In the metaplectic case $H_n = \wt{\J}_n$, this bound is slightly modified to $s_1 \le 1/2$ (cf. \cite[Theorem~11.2]{GRS11}). These bounds provide the analytic justification for our focus on almost tempered representations.
\end{rem}

\subsection[L-functions]{$L$-functions}\label{ltn}

In this section, we recall the definitions of partial and completed $L$-functions associated with the groups $\G_n^{\ve}$, $\H_n^{\ve}$, $\J_m$, and $\wt{\J}_m$. Throughout this section, we assume that $F$ is a number field to discuss the global theory.

Let $G$ be a reductive group over $F$ and $St : \widehat{G}(\CC) \to \GL_N(\CC)$ be the standard representation of the complex dual group $\widehat{G}(\CC)$. Let $\B = \T \N$ be a Borel subgroup of $G$, where $\T$ is a maximal torus and $\T^0$ is the maximal $F$-split torus in $\T$. If $G$ is split, we have $\T^0 = \T$.

Let $\chi = \otimes_v \chi_v$ be an automorphic character of $\GL_1(\A)$ and $\pi = \otimes'_v \pi_v$ be an irreducible cuspidal automorphic representation of $G(\A)$. For a place $v$ where $G, \pi_v$, and $\chi_v$ are unramified, let $\varpi$ be a uniformizer of $F_v$ and $q$ be the cardinality of the residue field. 

If $G$ is a split group of type $\H_n^{\ve}$ or $\J_n$, $\pi_v$ is the unramified constituent of the principal series $\Ind_{\B(F_v)}^{G(F_v)} (\lambda)$ for an unramified character $\lambda = \lambda_1 \otimes \dots \otimes \lambda_n$ of $\T(F_v)$. To simplify the notation, we write $\lambda_j$ for $\lambda_j(\varpi)$. We define the Satake parameter $c_{\pi_v}$ as 
\[
c_{\pi_v} \coloneqq \begin{cases}
\mathrm{diag}(\lambda_1, \dots, \lambda_n, \lambda_n^{-1}, \dots, \lambda_1^{-1}) \in \SO_{2n}(\CC) & \text{if } G = \H_n^{\ve}, \\
\mathrm{diag}(\lambda_1, \dots, \lambda_n, 1, \lambda_n^{-1}, \dots, \lambda_1^{-1}) \in \SO_{2n+1}(\CC) & \text{if } G = \J_n.
\end{cases}
\]
The local $L$-factor is then given by $L(s, \pi_v \times \chi_v) \coloneqq \det(\mathrm{Id}_N - c_{\pi_v} \cdot \chi_v(\varpi) q^{-s})^{-1}$.

In the case where $G$ is the quasi-split but non-split group $\H_n^0$, there exists a conjugacy class of unramified characters $\lambda = \lambda_1 \otimes \dots \otimes \lambda_{n-1}$ of $\T^0(F_v)$ such that $\pi_v$ is an unramified constituent of $\Ind_{\B(F_v)}^{G(F_v)}(\lambda)$. We set
\begin{align*}
c_{\pi_v} &= \diag(\lambda_1(\varpi), \dots, \lambda_{n-1}(\varpi), \lambda_{n-1}^{-1}(\varpi), \dots, \lambda_1^{-1}(\varpi)) \in \Sp_{2n-2}(\CC), \\
c_{\pi_v}' &= \diag(\lambda_1(\varpi), \dots, \lambda_{n-1}(\varpi), -1, 1, \lambda_{n-1}^{-1}(\varpi), \dots, \lambda_1^{-1}(\varpi)) \in \GL_{2n}(\CC).
\end{align*}
The local $L$-factor is defined as $L(s, \pi_v \times \chi_v) \coloneqq \det(\mathrm{Id}_{2n} - c_{\pi_v}' \cdot \chi_v(\varpi) q^{-s})^{-1}$.

When $G$ is the non-connected group $\G_n^{\ve}$, we define the local $L$-factor $L(s, \pi_v \times \chi_v)$ by setting it equal to $L(s, \pi_{v,0} \times \chi_v)$, where $\pi_{v,0}$ is an irreducible sub-representation of the restriction $\pi_v |_{\H_n^{\ve}(F_v)}$. An argument similar to the proof of \cite[Lemma 2.8]{HKK25} confirms that this definition is well-defined and independent of the choice of $\pi_{v,0}$.

For the metaplectic group $\wt{\J}_n$, let $\pi' = \otimes'_v \pi_v'$ be an irreducible genuine cuspidal automorphic representation of $\wt{\J}_n(\A)$. At an unramified place $v$, there exists an unramified character $\lambda = \lambda_1 \otimes \dots \otimes \lambda_n$ of $\T(F_v)$ such that $\pi_v' \simeq \Ind_{\wt{\B}(F_v)}^{\wt{\J}_n(F_v)}(\lambda_{\psi_v})$, where $\lambda_{\psi_v}$ is the character of $\wt{\T}(F_v)$ defined as in Section~\ref{atr}. The Satake parameter is
\[
c_{\pi_v'} \coloneqq \diag(\lambda_1(\varpi), \dots, \lambda_n(\varpi), \lambda_n^{-1}(\varpi), \dots, \lambda_1^{-1}(\varpi)) \in \GL_{2n}(\CC),
\]
and the local $L$-factor is $L_{\psi_v}(s, \pi_v' \times \chi_v) \coloneqq \det(\mathrm{Id}_{2n} - c_{\pi_v'} \cdot \chi_v(\varpi) q^{-s})^{-1}$. Note that this factor depends on the choice of $\psi$ via the identification of the unramified principal series.

Let $\G$ be either the quasi-split classical group or the metaplectic group and let $\sigma$ be an irreducible cuspidal automorphic representation of $\G(\A)$. For a finite set of places $S$ including all archimedean places and those where $\sigma$ or $\chi$ is ramified, we define the partial $L$-function as
\[
L_{\psi}^S(s, \sigma \times \chi) \coloneqq \begin{cases}
\prod_{v \notin S} L(s, \sigma_v \times \chi_v) & \text{if } \G = \G_n^{\ve},\ \H_n^{\ve}, \\

\prod_{v \notin S} L_{\psi_v}(s, \sigma_v \times \chi_v) & \text{if } \G = \wt{\J}_n \text{ and } \sigma \text{ is genuine}.
\end{cases}
\]
The dependency on $\psi$ occurs only in the metaplectic genuine case. If $\chi$ is trivial, we simply write $L_{\psi}^S(s, \sigma)$. 

As we have observed, for $\G = \G_n^{\ve}$, $\H_n^{\ve}$, $\J_n$, and $\wt{\J}_n$, the theory of partial $L$-functions for $\G \times \GL_1$ is primarily based on the local $L$-factors for unramified representations. To define the completed $L$-function, this notion must be extended to all irreducible smooth representations. Piatetski--Shapiro and Rallis~\cite{PS-R87} introduced a general class of zeta integrals, generalizing the Godement--Jacquet zeta integrals via the so-called doubling method. Subsequently, Yamana~\cite{Yam14} developed a comprehensive theory of these local zeta integrals and defined the local $L$-factor as the greatest common divisor (GCD) of the local zeta integrals over the space of good sections.

On the other hand, there exists an alternative definition of local $L$-factors in terms of the $\gamma$-factors arising from local zeta integrals, due to Lapid--Rallis~\cite{LR05} for classical groups and Gan~\cite{Gan12} for metaplectic groups. These two definitions possess complementary strengths: while the former is well-suited for controlling the poles of local zeta integrals, the latter has the advantage of satisfying the expected compatibility with the local Langlands correspondence; namely, it agrees with the corresponding $L$-factor on the Galois side. 

Yamana~\cite{Yam14} proved that these two definitions indeed give rise to the same local $L$-factor. It is also worth noting that Cai, Friedberg, and Kaplan~\cite{CFK22} further generalized the construction of Lapid--Rallis by defining local $\gamma$-factors (and hence local $L$-factors) in the broader setting of $\G(F_v) \times \GL_r(F_v)$ for $r \ge 1$ for split classical groups. Since our focus in this work is on quasi-split classical groups and metaplectic groups, we follow Yamana’s definition for the local $L$-factor $L_{\psi}(s, \sigma_v \times \chi_v)$ for any irreducible smooth representation $\sigma_v \times \chi_v$ of $\G(F_v) \times \GL_1(F_v)$. (Here, the additive character $\psi$ is relevant only when $\G = \wt{\J}_n$ and $\sigma_v$ is genuine.)

Following this theory, the completed $L$-function is defined by the product over all places:
\[
L_{\psi}(s, \sigma \times \chi) \coloneqq \prod_v L_{\psi}(s, \sigma_v \times \chi_v).
\]
This completed $L$-function admits a meromorphic continuation to $\CC$ and satisfies a functional equation. By construction, it relates to the partial $L$-function via
\[
L_{\psi}(s, \sigma \times \chi) = L_{\psi}^S(s, \sigma \times \chi) \cdot \prod_{v \in S} L_{\psi}(s, \sigma_v \times \chi_v).
\]

\subsection{Weil representation and theta series}\label{heis}

In this section, we summarize the construction of the Weil representation and the associated theta series, which are fundamental to the global theta correspondence. For $k \ge 0$, let $(W_k, \langle \cdot, \cdot \rangle_{W_k})$ be a $2k$-dimensional symplectic space over $F$, and fix a polarization $W_k = Y_k \oplus Y_k^*$ by choosing maximal totally isotropic subspaces $Y_k$ and $Y_k^*$. 

Let $\mc{H}(W_k)$ denote the Heisenberg group of rank $2k+1$ associated with $W_k$. For an $F$-algebra $R$ (where $R = \A$ or $F_v$), the elements of $\mc{H}(W_k)(R)$ are realized as triples $(y, y'; t) \in Y_k(R) \oplus Y_k^*(R) \oplus R$ with the group law
\[
(y_1, y_1'; t_1) + (y_2, y_2'; t_2) \coloneqq (y_1 + y_2, y_1' + y_2'; t_1 + t_2 + \tfrac{1}{2} (\langle y_1, y_2' \rangle_{W_k} + \langle y_2, y_1' \rangle_{W_k})).
\]
The center of $\mc{H}(W_k)$ is isomorphic to $\mathbb{G}_a$, consisting of elements of the form $(0, 0; t)$. For $0 \le k \le m-1$, we identify $\mc{H}(W_k)$ as a subgroup of the unipotent radical $\U_{m,m-k}' \subset \J_m$ via the embedding
\begin{equation}\label{hs}
(x, y; t) \mapsto \begin{pmatrix} \mathrm{Id}_{m-k-1} & 0 & 0 & 0 & 0 & 0 \\ & 1 & x & \frac{y}{2} & t & 0 \\ & & 1 & 0 & y^* & 0 \\ & & 0 & 1 & x^* & 0 \\ & & & & 1 & 0 \\ & & & & & \mathrm{Id}_{m-k-1} \end{pmatrix} \in \U_{m,m-k}'.
\end{equation}
This map induces an isomorphism between $\mc{H}(W_k)$ and the quotient group $\U_{m,m-k-1}' \backslash \U_{m,m-k}'$.
The metaplectic group $\wt{\J}_k(F_v)$ acts on the Heisenberg group $\mc{H}(W_k)(F_v)$ via the natural action of $\J_k(F_v)$ on the symplectic space $W_k(F_v)$. The \textit{Jacobi group}, defined by the semi-direct product $\mc{H}(W_k)(F_v) \rtimes \wt{\J}_k(F_v)$, admits the Weil representation $\Omega_{\psi_v,W_k}$ realized on the Schwartz space $\mc{S}(Y_k(F_v))$. It is important to note that $\Omega_{\psi_v,W_k}$ is the unique representation that extends the projective Schrödinger representation $\overline{\Omega}_{\psi_v,W_k}$ of $\mc{H}(W_k)(F_v) \rtimes \J_k(F_v)$ to the metaplectic Jacobi group $\mc{H}(W_k)(F_v) \rtimes \wt{\J}_k(F_v)$ (cf. \cite[Corollary~4.19]{Li08}). We denote the restriction of $\Omega_{\psi_v,W_k}$ to $\wt{\J}_k(F_v)$ by $\omega_{\psi_v,W_k}$.

We remark that while $\Omega_{\psi_v,W_k}$ is irreducible, its restriction $\omega_{\psi_v,W_k}$ is reducible; namely, $\omega_{\psi_v, W_k} = \omega_{\psi_v, W_k}^{+} \oplus \omega_{\psi_v, W_k}^{-}$, where $\omega_{\psi_v, W_k}^{+}$ consists of even functions in $\mc{S}(Y_k(F_v))$ and $\omega_{\psi_v, W_k}^{-}$ consists of odd functions in $\mc{S}(Y_k(F_v))$.\bb

To facilitate the subsequent analysis, we recall the explicit formulas for this action (cf. \cite[\S 1.6]{GRS98}):
\begin{enumerate}
    \item For $y, y_0 \in Y_k(F_v)$ and $\phi \in \mc{S}(Y_k(F_v))$,
    \[
    (\Omega_{\psi_v,W_k}((y, 0, 0)) \cdot \phi)(y_0) =  \phi(y + y_0).
    \]
    \item For $y_0 \in Y_k(F_v)$, $y' \in Y_k^*(F_v)$ and $\phi\in \mc{S}(Y_k(F_v))$,
    \[
    (\Omega_{\psi_v,W_k}((0, y', 0)) \cdot \phi)(y_0) = \psi_v\left(  \langle y_0, y' \rangle_{W_k} \right) \cdot \phi(y_0).
    \]
    \item For $y_0 \in Y_k(F_v), \ t \in F_v$ and $\phi \in \mc{S}(Y_k(F_v))$,
    \[
    (\Omega_{\psi_v,W_k}((0, 0, t)) \cdot \phi)(y_0) = \psi_v(t) \cdot \phi(y_0).
    \] 
    \item For $z \in \GL(Y_k)(F_v), \ \ve' \in \mu_2$ and $\phi \in \mc{S}(Y_k(F_v))$,
    \[
    \left( \Omega_{\psi_v,W_k} \left( \begin{pmatrix} z & \\ & z^* \end{pmatrix}, \ve' \right) \cdot \phi \right)(y) = \ve' \cdot \gamma_{\psi_v}(\det(z)) \cdot \phi(y\cdot z).
    \]
    \item For $s \in \Hom(Y_k, Y_k^*), \ \ve' \in \mu_2$ and $\phi \in \mc{S}(Y_k(F_v))$,
    \[
    \left( \Omega_{\psi_v,W_k}\big((s, \ve')\big) \cdot \phi \right)(y) = \ve' \cdot \psi_v\left( \tfrac{1}{2} \langle y, sy \rangle \right) \cdot \phi(y).
    \]
\end{enumerate}

In the degenerate case $k = 0$, the Heisenberg group $\mc{H}(W_0)$ reduces to $\mathbb{G}_a$, and $\wt{\J}_0$ is the trivial group. Consequently, we have
\[
\mc{H}(W_0) \rtimes \wt{\J}_0 \simeq \mathbb{G}_a,
\]
and the Weil representation is simply given by $\Omega_{\psi_v,W_0}(t) = \psi_v(t)$ for $t \in \mc{H}(W_0) \simeq \mathbb{G}_a$. When $\mc{H}(W_0)$ is viewed as a subgroup of the unipotent radical $\U_{m,m}'$, the character $\psi$ is explicitly given by
\begin{equation}\label{psi}
\psi(u) = \psi(\langle u f_m^*, f_m^* \rangle_{W_m}).
\end{equation}

Let $\Omega_{\psi,W_k} \coloneqq \bigotimes_v' \Omega_{\psi_v,W_k}$ be the global Weil representation of $\mc{H}(W_k)(\A) \rtimes \wt{\J}_k(\A)$ on $\mc{S}(Y_k(\A))$. Then using $\Omega_{\psi,W_k}$, we can define the global theta function $\theta_{\psi,W_k}(\phi)$ as
\begin{equation}\label{gtheta}
\theta_{\psi,W_k}(\phi)(v \wt{h}') \coloneqq \sum_{y \in Y_k(F)} (\Omega_{\psi,W_k}(v \wt{h}') \cdot \phi)(y), \quad (v, \wt{h}') \in \mc{H}(W_k)(\A) \rtimes \wt{\J}_k(\A).
\end{equation}
Therefore, $\theta_{\psi,W_k}$ gives an $(\mc{H}(W_k)(\A) \rtimes \wt{\J}_k(\A))$-equivariant map
\[
\theta_{\psi,W_k}\colon \Omega_{\psi,W_k} \rightarrow [\text{functions on } \mc{H}(W_k)(\A) \rtimes \wt{\J}_k(\A)].
\]

Under the identification $\mc{H}(W_k) \simeq \U_{m,m-k-1}' \backslash \U_{m,m-k}'$, we can view the theta functions $\{ \theta_{\psi,W_k}(\phi) \}_{\phi \in \mc{S}(Y_k(\A))}$ as being defined on $\U_{m,m-k}'(\A) \rtimes \wt{\J}_k(\A)$. Consequently, we regard $\Omega_{\psi,W_k}$ as a representation of $\U_{m,m-k}'(\A) \rtimes \wt{\J}_k(\A)$ on the space $\{ \theta_{\psi,W_k}(\phi) \}_{\phi \in \mc{S}(Y_k(\A))}$, where $\U_{m,m-k-1}'(\A)$ acts trivially. Let $\omega_{\psi,W_k} = \bigotimes_v' \omega_{\psi_v,W_k}$ denote the restriction of $\Omega_{\psi,W_k}$ to $\wt{\J}_k(\A)$. By restricting $\theta_{\psi,W_k}$ to $\omega_{\psi,W_k}$, we also obtain an $\wt{\J}_k(\A)$-equivariant map
\[
\theta_{\psi,W_k}'\colon \omega_{\psi,W_k} \rightarrow [\text{functions on } \wt{\J}_k(\A)]
\] 
defined by
\begin{equation}\label{gtheta_prime}
\theta_{\psi,W_k}'(\phi)(\wt{h}') \coloneqq \sum_{y \in Y_k(F)} (\omega_{\psi,W_k}(\wt{h}') \cdot \phi)(y), \quad \wt{h}' \in \wt{\J}_k(\A).
\end{equation}
For a finite set $S$ of places of $F$, the partial Weil representation $\omega_{\psi,W_k}^S$ is defined by the restricted tensor product
\[
\omega_{\psi,W_k}^S \coloneqq \Big(\sideset{}{'}\bigotimes_{v\notin S}\omega_{\psi_v,W_k}^{+}\Big) \otimes \Big(\bigotimes_{v\in S}\omega_{\psi_v,W_k}^{-}\Big).
\]
Then 
\[\omega_{\psi,W_k}=\bigoplus_{S}\omega_{\psi,W_k}^S\]
where $S$ ranges over all finite subsets of places of $F$.

The basic facts about the maps $\theta_{\psi,W_k}$ and $\theta_{\psi,W_k}'$ are summarized in the following proposition:
\begin{prop}\label{p.t}
\begin{enumerate}
    \item For every $\phi \in \Omega_{\psi,W_k}$, $\theta_{\psi,W_k}(\phi)$ is left $\U_{m,m-k}'(F) \rtimes \wt{\J}_k(F)$-invariant and contained in $\mc{A}^2(\U_{m,m-k}' \rtimes \wt{\J}_k)$.
    \item $\ker(\theta_{\psi,W_k}') = \bigoplus_{|S| \text{ is odd}} \omega_{\psi,W_k}^S$ and $\operatorname{Im}(\theta_{\psi,W_k}') \simeq \bigoplus_{|S| \text{ is even}} \omega_{\psi,W_k}^S$.
    \item $\theta_{\psi,W_k}'(\omega_{\psi,W_k}^S)$ is contained in $\mc{A}^2(\wt{\J}_k)$ and is cuspidal if and only if $S$ is non-empty.
\end{enumerate}
\end{prop}

\subsection{The Bessel period}\label{spb}

In this section, we introduce the Bessel periods and their associated local models, which play a central role in the formulation of the GGP conjectures. 

For $2 \le k \le n+\ve$, we define a character $\mu_{k,\ve} = \bigotimes_{v} \mu_{k,\ve,v}$ of $\U_{n,k-1}^{\ve}(\A)$ by
\[
\mu_{k,\ve}(u) \coloneqq \psi\left( (ue_2, e_1^*)_{V_{n}^{\ve}} + \dots + (ue_{k-1}, e_{k-2}^*)_{V_{n}^{\ve}} + (ue, e_{k-1}^*)_{V_{n}^{\ve}} \right).
\]
In the case $k=1$, we set $\mu_{1,\ve} = \II$. In the extremal case $k = n+\ve$, we simply denote $\mu_{n+\ve,\ve}$ (resp. $\mu_{n+\ve,\ve,v}$) by $\mu_{\ve}$ (resp. $\mu_{\ve,v}$). Note that while the type $(d,c)$ is suppressed in our notation, the character $\mu_{k,0}$ fundamentally depends on the underlying type of the orthogonal space $V_n^0$.

For $1\le k \le n+\ve-1$, let $V_{n,k}^{\ve}$ be the subspace of $V_n^{\ve}$ defined by
\[
V_{n,k}^{\ve} \coloneqq \bigoplus_{i = k}^{n + \ve - 1} (F e_i \oplus F e_i^*) \oplus F e',
\]
and we set $V_{n,n+\ve}^{\ve} \coloneqq F e'$. Note that $V_{n,n+1}^{1}=\{0\}$ since $e' = 0$. For $1 \le k \le n+\ve-1$, we define the groups
\begin{align*}
\G_{n,k}^{\ve} &\coloneqq \{ g \in \GL(V_{n,k}^{\ve}) \mid (gv_1, gv_2)_{V_{n,k}^{\ve}} = (v_1, v_2)_{V_{n,k}^{\ve}} \text{ for all } v_1, v_2 \in V_{n,k}^{\ve} \}, \\
\H_{n,k}^{\ve} &\coloneqq \{ h \in \G_{n,k}^{\ve} \mid \det(h) = 1 \}.
\end{align*}
For $k=n+\ve$, both $\G_{n,n+\ve}^{\ve}$ and $\H_{n,n+\ve}^{\ve}$ are defined to be the trivial group.

We also define a distinguished element $\e_k \in \G_{n,k}^{\ve}$ as follows:
\begin{itemize}
    \item If $\ve = 0$, $\e_k \in \G_{n,k}^{0}$ acts on $V_{n,k}^{0}$ as multiplication by $-1$, i.e., $\e_k = -\mathrm{Id}$.
    \item If $\ve = 1$ and $1 \le k \le n+\ve-1$, $\e_k \in \G_{n,k}^{1}$ acts trivially on the orthogonal complement of $F e_n \oplus F e_n^*$ in $V_{n,k}^{1}$, and acts on this plane by $\e_k(e_n) = e_n^*$ and $\e_k(e_n^*) = e_n$.
    \item If $\ve = 1$ and $k = n+1$, we set $\e_k = \mathrm{Id}$.
\end{itemize}
It is clear that $\G_{n,k}^{\ve} = \H_{n,k}^{\ve} \rtimes (\mu_2 \cdot \e_k)$. For $\mathbf{t} \in \mu_2=\{1,-1\}$, define $\mathbf{t} \cdot \e_k \in \G_{n,k}^{\ve}$ as 
\[
\mathbf{t} \cdot \e_k =
\begin{cases} 
\mathrm{Id}, & \quad \text{if } \mathbf{t}=1, \\ 
\e_k, & \quad \text{if } \mathbf{t}=-1.
\end{cases}
\]
By fixing the partial basis $\{e_1, \dots, e_{k-1}, e, e_{k-1}^*, \dots, e_1^*\}$, we may regard $\G_{n,k}^{\ve}$ as a subgroup of $\G_n^{\ve}$. Since $\e_k$ normalizes $\U_{n,k-1}^{\ve}$, we define the extended unipotent group $\wt{\U}_{n,k-1}^{\ve} \coloneqq \U_{n,k-1}^{\ve} \rtimes (\mu_2 \cdot \e_k)$. When $k=n+\ve-1$, we denote $\e_{n+\ve-1}$ simply by $\e$.

\subsubsection{Bessel models}
Let $\mathrm{sgn}_v$ be the character of $\mu_2(F_v) \cdot \e_k$ defined by $\mathrm{sgn}_v(\mathbf{t} \cdot \e_k) \coloneqq \det(\mathbf{t} \cdot \e_k)$. Since $\e_k$ stabilizes $\mu_{k,\ve,v}$, there are exactly two extensions $\mu_{k,\ve,v}^{\pm} \colon \wt{\U}_{n,k-1}^{\ve}(F_v) \to \CC^{\times}$ of $\mu_{k,\ve,v}$, where $\mu_{k,\ve,v}^+ \coloneqq \mu_{k,\ve,v} \otimes \II$ and $\mu_{k,\ve,v}^- \coloneqq \mu_{k,\ve,v} \otimes \mathrm{sgn}_v$. 
We set $\D_{n,k}^{\ve} \coloneqq \U_{n,k-1}^{\ve} \rtimes \H_{n,k}^{\ve}$ and $\wt{\D}_{n,k}^{\ve} \coloneqq \wt{\U}_{n,k-1}^{\ve} \rtimes \H_{n,k}^{\ve}$. 

For $\wt{\pi} \in \Irr(\G_n^{\ve}(F_v))$, we say that $\wt{\pi}$ is of $\mu_{k,\ve,v}^{\pm}$-\textit{Bessel type} if
\[
\Hom_{\wt{\D}_{n,k}^{\ve}(F_v)}(\wt{\pi} \boxtimes \rho, \mu_{k,\ve,v}^{\pm}) \neq 0
\]
for some $\rho \in \Irr(\H_{n,k}^{\ve}(F_v))$. Here, the representations $\rho$ and $\mu_{k,\ve,v}^{\pm}$ are extended to $\wt{\D}_{n,k}^{\ve}(F_v)$ by letting them act trivially on $\wt{\U}_{n,k-1}^{\ve}(F_v)$ and $\H_{n,k}^{\ve}(F_v)$, respectively. 

In this case, given a nonzero functional $\wt{l} \in \Hom_{\wt{\U}_{n,k-1}^{\ve}(F_v)}(\wt{\pi}, \mu_{k,\ve, v}^{\pm})$, we define the $\mu_{k,\ve,v}^{\pm}$-\textit{Bessel model} $\mc{B}_{k,\psi_v}^{\ve,\pm}(\wt{\pi})$ as the space consisting of the functions
\[
\mc{B}_{k,\psi_v}^{\ve,\pm}(\wt{\varphi})(g) \coloneqq \wt{l}(\wt{\pi}(g)\wt{\varphi}) \quad \text{for } \wt{\varphi} \in \wt{\pi}, \ g \in \G_n^{\ve}(F_v).
\]
Analogous definitions of Bessel type and Bessel models apply to representations $\pi \in \Irr(\H_n^{\ve}(F_v))$.

\subsubsection{Whittaker--Bessel models}
When $k = n + \ve$, the group $\H_{n,n+\ve}^{\ve}$ is trivial, and the Bessel functional $\mc{B}_{n + \ve, \psi_v}^{\ve,\pm}$ reduces to a Whittaker functional. For $\wt{\varphi} \in \wt{\pi}$, we denote this evaluation by $W_{\psi_v}^{\ve,\pm}(\wt{\varphi}) \coloneqq \mc{B}_{n + \ve, \psi_v}^{\ve,\pm}(\wt{\varphi})$. If this functional is nonzero, we say that $\wt{\pi}$ is $\mu_{\ve,v}^{\pm}$-\textit{generic}. 

We define the associated Whittaker functions on $\G_n^{\ve}(F_v)$ and $\H_n^{\ve}(F_v)$ by 
\[
W_{\psi_v}^{\ve,\pm}(\wt{\varphi})(g) \coloneqq W_{\psi_v}^{\ve,\pm}(\wt{\pi}(g)\wt{\varphi}) \quad \text{and} \quad W_{\psi_v}^{\ve}(\varphi)(h) \coloneqq W_{\psi_v}^{\ve}(\pi(h)\varphi),
\]
respectively. The spaces of these Whittaker functions, which correspond to the Whittaker models of $\wt{\pi}$ and $\pi$, are denoted by $\mc{W}_{\psi_v}^{\ve,\pm}(\wt{\pi})$ and $\mc{W}_{\psi_v}^{\ve}(\pi)$, respectively.

\subsubsection{Bessel periods}\label{Bp}
We now define the global counterparts of these notions. Let $\mf{T}$ be the collection of all finite sets of places of $F$ with even cardinality. For $\TT \in \mf{T}$, we define the character $\mathrm{sgn}_{\TT} \colon (\mu_2(F) \backslash \mu_2(\A)) \cdot \e_k \to \CC^{\times}$ by $\mathrm{sgn}_{\TT} \coloneqq \prod_{v \in \TT} \mathrm{sgn}_v$. When $ \TT = \emptyset $, $\mathrm{sgn}_{\emptyset}$ is just the trivial character. Sometimes, we implicitly extend the character $\mathrm{sgn}_{\TT}$ to $[\G_n^{\ve}]$ by extending it trivially across $[\H_n^{\ve}]$.

Given an irreducible cuspidal automorphic representation $\wt{\pi}$ (resp. $\pi$) of $\G_n^{\ve}(\A)$ (resp. $\H_n^{\ve}(\A)$) and an automorphic representation $\rho$ of $\H_{n,k}^{\ve}(\A)$, we define the global period functional $\mc{B}_{k,\psi}^{\ve,\TT}$ on $\wt{\pi} \boxtimes \rho$ (resp. $\pi \boxtimes \rho$) by
\begin{multline*}
\mc{B}_{k,\psi}^{\ve,\TT}(\wt{\varphi}_1 \otimes \varphi_2) \coloneqq 
\int_{[\H_{n,k}^{\ve}]} \Bigg(\int_{[\U_{n,k-1}^{\ve}]} \mu_{k,\ve}^{-1}(u) \left( \int_{[\mu_2]} \wt{\varphi}_1((\mathbf{t} \cdot \e_k)uh) \mathrm{sgn}_{\TT}(\mathbf{t} \cdot \e_k) \, d\mathbf{t} \right) \, du \Bigg)\varphi_2(h) \, dh,
\end{multline*}
and respectively,
\[
\mc{B}_{k,\psi}^{\ve}(\varphi_1 \otimes \varphi_2) \coloneqq 
\int_{[\H_{n,k}^{\ve}]} \left(\int_{[\U_{n,k-1}^{\ve}]} \mu_{k,\ve}^{-1}(u) \varphi_1(uh) \, du \right) \varphi_2(h) \, dh,
\]
for $\wt{\varphi}_1 \in \wt{\pi}$ (resp. $\varphi_1 \in \pi$) and $\varphi_2 \in \rho$.

When $\rho=\II$, we regard $\mc{B}_{k,\psi}^{\ve,\TT}$ (resp. $\mc{B}_{k,\psi}^{\ve}$) as a functional on $\wt{\pi}$ (resp. $\pi$) given by
\[
\mc{B}_{k,\psi}^{\ve,\TT}(\wt{\varphi}) \coloneqq 
\int_{[\H_{n,k}^{\ve}]} \int_{[\U_{n,k-1}^{\ve}]} \mu_{k,\ve}^{-1}(u) \left( \int_{[\mu_2]} 
\wt{\varphi}((\mathbf{t} \cdot \e_k)uh) \mathrm{sgn}_{\TT}(\mathbf{t} \cdot \e_k) \, d\mathbf{t} \right) \, du \, dh,
\]
and respectively,
\[
\mc{B}_{k,\psi}^{\ve}(\varphi) \coloneqq 
\int_{[\H_{n,k}^{\ve}]} \int_{[\U_{n,k-1}^{\ve}]} \mu_{k,\ve}^{-1}(u) \varphi(uh) \, du \, dh,
\]
for $\wt{\varphi} \in \wt{\pi}$ (resp. $\varphi \in \pi$).

\subsubsection{Whittaker--Bessel periods} \label{muge}
In the extremal case $k = n + \ve$, the group $\H_{n,n+\ve}^{\ve}$ is trivial. Consequently, the period $\mc{B}_{n+\ve,\psi}^{\ve}(\varphi)$ coincides with the \textit{Whittaker--Bessel period} of $\varphi \in \mc{A}(\H_n^{\ve})$. For $\TT \in \mf{T}$, $\wt{\varphi} \in \mc{A}(\G_n^{\ve})$, and $\varphi \in \mc{A}(\H_n^{\ve})$, we introduce the following notation:
\[
W_{\psi}^{\ve,\TT}(\wt{\varphi}) \coloneqq \mc{B}_{n+\ve,\psi}^{\ve,\TT}(\wt{\varphi}), \quad W_{\psi}^{\ve}(\varphi) \coloneqq \mc{B}_{n+\ve,\psi}^{\ve}(\varphi).
\]

For an irreducible cuspidal automorphic representation $\wt{\pi}$ (resp. $\pi$) of $\G_n^{\ve}(\A)$ (resp. $\H_n^{\ve}(\A)$), we say that $\wt{\pi}$ (resp. $\pi$) is $\mu_{\ve}^{\TT}$-\textit{generic} (resp. $\mu_{\ve}$-\textit{generic}) if the functional $W_{\psi}^{\ve,\TT}$ (resp. $W_{\psi}^{\ve}$) does not vanish identically on its space. If there exists some $\TT \in \mf{T}$ such that $\wt{\pi}$ is globally $\mu_{\ve}^{\TT}$-generic, we simply say that $\wt{\pi}$ is $\mu_{\ve}$-\textit{generic}.

As in the local setting, for $\wt{\varphi} \in \wt{\pi}$ (resp. $\varphi \in \pi$), we evaluate these functionals at group elements by defining the corresponding Whittaker functions on $\G_n^{\ve}(\A)$ (resp. $\H_n^{\ve}(\A)$) as follows:
\[
W_{\psi}^{\ve,\TT}(\wt{\varphi})(g) \coloneqq W_{\psi}^{\ve,\TT}(\wt{\pi}(g)\wt{\varphi}) \quad \big(\text{resp. } W_{\psi}^{\ve}(\varphi)(h) \coloneqq W_{\psi}^{\ve}(\pi(h)\varphi) \big).
\]
The space of these Whittaker functions  is denoted by $\mc{W}_{\psi}^{\ve,\TT}(\wt{\pi})$ (resp. $\mc{W}_{\psi}^{\ve}(\pi)$). 

\begin{rem} \label{type}
It is important to observe that the definition of Whittaker--Bessel periods depends fundamentally on the choice of the generic character $\mu_{\ve}$ of $\U_{n+\ve-1}^{\ve}(\A)$. When $\ve=1$, there exists a unique $\T_{n}^{1}(\A)$-orbit of generic characters. In contrast, when $\ve=0$, the $\T_{n-1}^{0}(\A)$-orbits of generic characters of $\U_{n-1}^{0}(\A)$ are parameterized by $c N_{E/F}(E^{\times})/F^{\times 2}$, where $E = F(\sqrt{d})$ (see \cite[Section~2.2]{AG17}). Indeed, by \cite[Section~12]{GGP12}, every such $\T_{n-1}^{0}(\A)$-orbit of generic characters can be realized as $\mu_{0}$ by varying the type of the underlying orthogonal space $V_{n}^{0}$ from $(d,c)$ to $(d,c')$ for $c' \in c N_{E/F}(E^{\times})/F^{\times 2}$. By identifying the corresponding unipotent groups via their natural isomorphisms, it suffices to establish results for $\G_n^0$ (resp., $\H_n^0$) across all possible types $(d,c) \in (F^{\times}/F^{\times 2})^2$ with their corresponding characters $\mu_{0}$ to cover all Whittaker data. This principle applies equally to the study of local Whittaker--Bessel models.
\end{rem}

\subsection{The Fourier--Jacobi period}\label{spf}

In this section, we recall the construction of the Fourier--Jacobi periods and their corresponding local models for the symplectic and metaplectic groups. For $2 \le k \le m$ and any place $v$ of $F$, we define a character $\mu_{k,v}'$ of $\U_{m,k}'(F_v)$ by
\[
\mu_{k,v}'(u') \coloneqq \psi\left( \langle u' f_2, f_1^* \rangle_{W_m} + \cdots + \langle u' f_k, f_{k-1}^* \rangle_{W_m} \right), \quad u' \in \U_{m,k}'(F_v).
\]
For $k=0, 1$, we define $\mu_{k,v}'$ to be the trivial character $\II$. The global character $\mu_k' = \otimes_v \mu_{k,v}'$ of $\U_{m,k}'(\A)$ is defined as the restricted tensor product of these local characters.

For $0 \le k \le m-1$, let $W_{m,k}'$ be the subspace of $W_m$ defined by
\[
W_{m,k}' \coloneqq \bigoplus_{i = k+1}^{m} (F f_i \oplus F f_i^*).
\]
We set $\J_{m,k}' \coloneqq \Sp(W_{m,k}')$ to be the associated symplectic group. By fixing the basis $\{f_1, \dots, f_k, f_1^*, \dots, f_k^*\}$, we regard $\J_{m,k}'$ as a subgroup of $\J_m$. For $k=m$, we define $W_{m,m}' = \{0\}$ and let $\J_{m,m}'$ be the trivial group. 

For $0 \le k \le m$, we define the subgroup $\D_{m,k}'$ of $\wt{\J}_m$ by
\[
\D_{m,k}' \coloneqq \U_{m,k}' \rtimes \wt{\J}_{m,k}'.
\]
Using the isomorphism $\mc{H}(W_{m,k}') \simeq \U_{m,k-1}' \backslash \U_{m,k}'$ as in \eqref{hs}, the Weil representation $\Omega_{\psi_v, W_{m,k}'}$ can be viewed as a representation of $\D_{m,k}'(F_v)$. We also extend $\mu_k'$ to a character of $\D_{m,k}'$ by letting it act trivially on $\wt{\J}_{m,k}'$.

\subsubsection{Fourier--Jacobi models}
For $\lambda \in F_v^{\times}$, let $\nu_{\psi_v,W_{m,k}'}^{\lambda} \coloneqq \mu_{k,v}' \otimes \Omega_{\psi_{\lambda,v},W_{m,k}'}$ be a representation of $\D_{m,k}'(F_v)$. When $\lambda = 1$, we suppress the superscript and simply write $\nu_{\psi_v,W_{m,k}'}$. Note that $\nu_{\psi,W_{m,m}'}^{\lambda} = \mu_{m,v}' \otimes \psi_{\lambda,v}$ (cf.~\eqref{psi}).

For $\sigma \in \Irr(\wt{\J}_m)$, we say that $\sigma$ is of $(\mu_{k,v}', \lambda)$-\textit{Fourier--Jacobi type} if
\[
\Hom_{\D_{m,k}'(F_v)}(\sigma \boxtimes \rho^{\ve_{\sigma}}, \nu_{\psi_v,W_{m,k}'}^{\lambda}) \neq 0
\]
for some $\rho \in \Irr(\wt{\J}_{m,k}')$. In this case, let $l' \in \Hom_{\U_{m,k}'(F_v)}(\sigma, \nu_{\psi_v,W_{m,k}'}^{\lambda})$ be a nonzero element. We define the $(\mu_{k,v}', \lambda)$-\textit{Fourier--Jacobi model} $\mc{FJ}_{k,\psi_v}^{\lambda}(\sigma)$ as the space of functions
\[
\mc{FJ}_{k,\psi_v}^{\lambda}(\varphi)(\wt{h'}) \coloneqq l'(\wt{h'} \cdot \varphi) \quad \text{for } \varphi \in \sigma, \ \wt{h'} \in \wt{\J}_m(F_v).
\]

\subsubsection{Whittaker--Fourier--Jacobi models}
When $k = m$, we have $\D_{m,m}' = \U_{m,m}'$. For $\lambda \in F_v^{\times}$, the functional $\mc{FJ}_{m,\psi_v}^{\lambda}$ reduces to a Whittaker functional. For $\varphi \in \sigma$, we denote this evaluation by $W_{\psi_v}^{\lambda}(\varphi) \coloneqq \mc{FJ}_{m,\psi_v}^{\lambda}(\varphi)$. If this functional is non-trivial, we say that $\sigma$ is $(\psi_v, \lambda)$-\textit{generic}.

We define the associated Whittaker function on $\wt{\J}_m(F_v)$ by 
\[
W_{\psi_v}^{\lambda}(\varphi)(\wt{h'}) \coloneqq W_{\psi_v}^{\lambda}(\sigma(\wt{h'})\varphi).
\]
The space of these Whittaker functions, which corresponds to the Whittaker--Fourier--Jacobi model of $\sigma$, is denoted by $\mc{W}_{\psi_v}^{\lambda}(\sigma)$.

\subsubsection{Fourier--Jacobi periods}
We now define the global analogues of these notions. For $0 \le k \le m-1$ and $\lambda \in F^{\times}$, let $\nu_{\psi,W_{m,k}'}^{\lambda} \coloneqq \otimes_v \nu_{\psi_v,W_{m,k}'}^{\lambda}$. Given an irreducible cuspidal automorphic representation $\sigma$ of $\wt{\J}_m(\A)$ and an automorphic representation $\rho$ of $\wt{\J}_{m,k}'(\A)$, we define the global period functional $\mc{FJ}_{k,\psi}^{\lambda}$ on $\sigma \boxtimes \rho^{\ve_{\sigma}} \boxtimes \nu_{\psi^{-1},m,k}^{\lambda}$ by
\begin{equation*} 
\begin{split}
&\mc{FJ}_{k,\psi}^{\lambda}(\varphi \otimes \varphi' \otimes \phi) \\
&\coloneqq \int_{[\J_{m,k}']} \int_{[\U_{m,k}']} \varphi(u'\wt{h'}) \varphi'(\wt{h'}) (\mu_k')^{-1}(u') \theta_{\psi_{\lambda}^{-1},W_{m,k}'}(\phi)(u'\wt{h'}) \, du' \, dh'
\end{split}
\end{equation*}
for $\varphi \in \sigma$, $\varphi' \in \rho^{\ve_{\sigma}}$, and $\phi \in \nu_{\psi^{-1},m,k}^{\lambda}$. Here, $\wt{h'} \in \wt{\J}_{m,k}'(\A)$ is a lifting of $h'$, and the integrand is independent of the choice of this lift. 

In the case where $\rho = \II$, we regard $\mc{FJ}_{k,\psi}^{\lambda}$ as a functional on $\sigma \boxtimes \nu_{\psi^{-1},W_{m,k}'}^{\lambda}$ given by

\begin{equation*}
\begin{split}
&\mc{FJ}_{k,\psi}^{\lambda}(\varphi \otimes \phi) \\&\coloneqq \int_{[\J_{m,k}']} \int_{[\U_{m,k}']} \varphi(u' \wt{h}') \II^{\ve_{\sigma}}(\wt{h}')(\mu_k')^{-1}(u') \theta_{\psi_{\lambda}^{-1},W_{m,k}'}(\phi)(u' \wt{h}') \, du' \, dh'.
\end{split}
\end{equation*}


\subsubsection{Whittaker--Fourier--Jacobi periods}
When $k = m$, the group $\J_{m,m}'$ is trivial. For an irreducible  automorphic representation $\sigma$ of $\wt{\J}_m(\A)$, the functional $\mc{FJ}_{m,\psi}^{\lambda}$ on $\sigma$ is defined by
\[
\mc{FJ}_{m,\psi}^{\lambda}(\varphi) \coloneqq \int_{[\U_{m}']} \varphi(u') (\mu_{m}' \cdot \psi_{\lambda})^{-1}(u') \, du'.
\]
This coincides with the \textit{Whittaker--Fourier--Jacobi period} of $\varphi$, often denoted by $W_{\psi}^{\lambda}(\varphi)$. If the functional $W_{\psi}^{\lambda}$ does not vanish identically on $\sigma$, we say that $\sigma$ is \textit{globally $(\psi, \lambda)$-generic}. 

As in the local setting, for $\varphi \in \sigma$, we evaluate this functional at group elements by defining the corresponding Whittaker function on $\wt{\J}_m(\A)$ as follows:
\[
W_{\psi}^{\lambda}(\varphi)(\wt{h'}) \coloneqq W_{\psi}^{\lambda}(\sigma(\wt{h'})\varphi), \quad \wt{h'} \in \wt{\J}_m(\A).
\]
The space of these Whittaker functions is denoted by $\mc{W}_{\psi}^{\lambda}(\sigma)$.

\section{Vanishing of the local theta liftings}\label{sec:3}

In this section, we compute the twisted Jacquet modules of the local Weil representations. These computations are essential for proving the vanishing of local theta liftings within a certain range. For the remainder of this section, we assume that all objects are defined over a non-archimedean local field $F_v$. For brevity, we suppress the subscript $v$ from the notation and identify an algebraic group $G$ with its group of $F_v$-rational points $G(F_v)$.

\subsection{Mixed models of local Weil representations}\label{weil}

In this subsection, let $(V, (\cdot, \cdot)_V)$ and $(W, \langle \cdot, \cdot \rangle_{W})$ denote vector spaces over $F$ of dimensions $n$ and $m$, respectively. We assume that the pair $(V, W)$ consists of one orthogonal space and one symplectic space. We consider $V \otimes W$ as a symplectic space over $F$ with the symplectic form $(\cdot, \cdot)$ given by
\[
(v_1 \otimes w_1, v_2 \otimes w_2) \coloneqq (v_1, v_2)_V \cdot \langle w_1, w_2 \rangle_W.
\]

Let $\G$ and $\J$ denote the isometry groups of $V$ and $W$, respectively. If $V$ (resp. $W$) is a symplectic space, we let $\wt{\G}$ (resp. $\wt{\J}$) denote the metaplectic double cover of $\G$ (resp. $\J$). To ensure a unified treatment, even when $V$ (resp. $W$) is an orthogonal space, we formally write $\wt{\G}$ (resp. $\wt{\J}$) for $\G$ (resp. $\J$). In such cases, an element in $\wt{\J}$ is represented as $(h, \ve') \in \J \times \{\pm 1\}$, although the underlying isometry is simply $h \in \J$.

There is a natural embedding of $\G \times \J$ into $\Sp(V \otimes W)$. Let $\Omega_{\psi,V,W}$ be the Weil representation associated with $\mc{H}(V \otimes W) \rtimes  \wt{\Sp}(V \otimes W)$. As shown in \cite[Sec.~8.1]{Yam14}, if $V$ (resp. $W$) is an odd-dimensional orthogonal space, the metaplectic group $\wt{\Sp}(V \otimes W)$ splits over $\G \times \wt{\J}$ (resp. $\wt{\G} \times \J$). In all other cases, $\wt{\Sp}(V \otimes W)$ splits over $\G \times \J$ (cf. \cite{Ku94}, \cite{MVW87}). For a unified formulation, we define the pair $(\G', \J')$ as follows:
\[
(\G', \J') \coloneqq \begin{cases} 
(\G, \wt{\J}) & \text{if $V$ is an odd-dimensional orthogonal space,} \\ 
(\wt{\G}, \J) & \text{if $W$ is an odd-dimensional orthogonal space,} \\ 
(\G, \J) & \text{otherwise.}
\end{cases}
\]

Let $X_k$ (resp. $Y_m$) be a $k$-dimensional (resp. $m$-dimensional) isotropic subspace of $V$ (resp. $W$), and let $X_k^*$ (resp. $Y_m^*$) be its dual. Set $V_k \coloneqq X_k \oplus X_k^*$ and let $V^{\perp}$ be the orthogonal complement of $V_k$ in $V$. Let $\G^0$ be the isometry group of $V^{\perp}$, and define $(\G^0)'$ to be $\G^0$ or $\wt{\G^0}$ consistently with the definition of $\G'$.

We have the decompositions $V = X_k \oplus V^{\perp} \oplus X_k^*$ and $W = Y_m \oplus Y_m^*$. By choosing a polarization $\z \oplus \z^*$ of $V^{\perp} \otimes W$, and using the polarization $(X_k^* \otimes W) \oplus \z^*$ of $V \otimes W$, we realize the action of $\omega_{\psi,V,W}$ of $\G' \times \J'$ on the Bruhat--Schwartz space $\mc{S}(X_k^* \otimes W) \otimes \mc{S}(\z^*)$ via the mixed Schrödinger model.

To describe this action, let $\P = \M \U$ be the parabolic subgroup of $\G$ stabilizing $X_k$, with Levi subgroup $\M$ and unipotent radical $\U$. For $\alpha \in \Hom(X_k^*, X_k)$, $\beta \in \Hom(V^{\perp}, X_k)$, and $\gamma \in \GL(X_k)$, we define \[\alpha^* \in \Hom(X_k^*, X_k), \ \beta^* \in \Hom(X_k^*, V^{\perp}), \text{ and } \gamma^* \in \GL(X_k^*)\] such that
\[
\langle \alpha x_1, x_2 \rangle = \langle x_1, \alpha^* x_2 \rangle, \quad \langle \beta v, x \rangle = \langle v, \beta^* x \rangle, \quad \langle \gamma x_1, x_2 \rangle = \langle x_1, \gamma^* x_2 \rangle
\]
for $x, x_i \in X_k^*, v \in V^{\perp}$. Let $\N \coloneqq \{ \alpha \in \Hom(X_k^*, X_k) \mid \alpha^* = -\alpha \}$. Then $\M \simeq \GL(X_k) \times \G^0$ and $\U \simeq \Hom(V^{\perp}, X_k) \times \N$. Let $\Z$ be the maximal unipotent subgroup of $\M$, which we regard as a subgroup of $\GL(X_k)$. The action of $(\Z \V \times (\G^0)') \times \J'$ on $\mc{S}(X_k^* \otimes W) \otimes \mc{S}(\z^*)$ via $\omega_{\psi,V,W} \simeq \omega_{\psi,V_k,W} \otimes \omega_{\psi,V^{\perp},W}$ is given as follows (cf. \cite[\S 7.4]{GI16}): for $\phi_1 \otimes \phi_2 \in \mc{S}(X_k^* \otimes W) \otimes \mc{S}(\z^*)$ and $(\w, \y) \in (X_k^* \otimes W) \oplus \z^*$,
\begin{enumerate}
    \item $\omega_{\psi,V,W}(z, 1)(\phi_1 \otimes \phi_2)(\w, \y) = \phi_1(z^* \w) \phi_2(\y)$ for $z \in \Z \subset \GL(X_k)$;
    \item $\omega_{\psi,V,W}(g, 1)(\phi_1 \otimes \phi_2)(\w, \y) = \phi_1(\w) (\omega_{\psi,V^{\perp},W}(g, 1) \phi_2)(\y)$ for $(g, \ve') \in (\G^0)'$;
    \item $\omega_{\psi,V,W}(t, 1)(\phi_1 \otimes \phi_2)(\w, \y) = \psi((\y, \y_{t,1}) + (\y_{t,2}, \y_{t,1})) (\phi_1 \otimes \phi_2)(\w, \y + \y_{t,2})$ for $t \in \Hom(V^{\perp}, X_k)$, where $(\y_{t,1}, \y_{t,2}) \in \z \oplus \z^*$ satisfies $t^* \w = \y_{t,1} + \y_{t,2}$;
    \item $\omega_{\psi,V,W}(s, 1)(\phi_1 \otimes \phi_2)(\w, \y) = \psi(\frac{1}{2} (s \w, \w)) (\phi_1 \otimes \phi_2)(\w, \y)$ for $s \in \N \subset \Hom(X_k^*, X_k)$;
    \item $\omega_{\psi,V,W}(1, (h', \ve'))(\phi_1 \otimes \phi_2)(\w, \y)$ 
\\ $= \phi_1(\w \cdot (h')^{-1}) (\omega_{\psi,V^{\perp},W}(1, (h', \ve')) \phi_2)(\y)$ for $(h', \ve') \in \J'.$
\end{enumerate}

Choose a basis $\{e_1,\cdots,e_{k}\}$ of $X_{k}$ and $\{e_{1}^*,\cdots,e_{k}^*\}$ of $X_{k}^*$ such that $(e_i,e_j^*)_V=\delta_{ij}$. Using the ordered basis $\{e_{1}^*,\cdots,e_{k}^*\}$ of $X_{k}^*$, we identify $X_{k}^* \otimes W$ with $W^{k}$. Furthermore, using the basis $\{e_1,\cdots,e_{k}\}$ of $X_{k}$, $\{e^{1},\cdots,e^{n-2k}\}$ of $V^{\perp}$, $\{e_{k}^*,\cdots,e_{1}^*\}$ of $X_{k}^*$, we can identify $\Z$, $\Hom(V^{\perp},X_{k})$ and $\N$  as $Z_{k}$, $M_{k,n-2k}$ and $S_{k}$, respectively. (For the definition of $Z_{k}$ and $S_{k}$, see subsection~\ref{subsec:2-1}). 

With the above identifications, we can describe the action of $\big(Z_{k}\cdot  S_{k}\cdot M_{k,n-2k}\ \times (\G^0)' \big) \times \J'$ on $\mc{S}(W^{k})\otimes \mc{S}(\z^*) $ is expressed by:
\begin{flalign}
\label{a5} \omega_{\psi,V,W}(z, 1)(\phi_1 \otimes \phi_2)(\w; \y) &= (\phi_1 \otimes \phi_2)(\w \cdot z; \y),  &z \in Z_k, \\
\label{a7} \omega_{\psi,V,W}(s, 1)(\phi_1 \otimes \phi_2)(\w; \y) &= \psi\left(\tfrac{1}{2} \mathrm{tr}(Gr(\mathbf{w}) s \varpi_k)\right) (\phi_1 \otimes \phi_2)(\w; \y),  &s \in S_k,
\end{flalign}
where $Gr(\mathbf{w}) = (\langle w_i, w_j \rangle_W)$.

If $\{e^1,\cdots,e^{n-2k}\}$ is an orthogonal basis of $V^{\perp}$, set $\z=V^{\perp}\otimes Y_m$ and $\z^*=V^{\perp}\otimes Y_m^*$. Then for $\w\in Y_m^{k}\ss W^{k}=X_k^*\otimes W$,  $\y=(y_1,\cdots,y_{n-2k})\in (Y_m^*)^{n-2k}=V^{\perp} \otimes Y_m^*$ and $t \in M_{k,n-2k}$, \begin{equation}\label{a9}\omega_{\psi,V,W}(t,1)(\phi_1\otimes \phi_2 )(\w;{\bf y})=\psi\big( \sum_{j=1}^{n-2k}\sum_{i=1}^{k}t_{i,j}\cdot \la y_j,w_i \ra_{W}\big)\cdot(\phi_1\otimes \phi_{2})(\mathbf{w},{\bf y}).\end{equation}

When $n=2l$ and $k=l-1$ and $V^{\perp}$ is a 2-dimensional split space such that $V^{\perp}=\la e_l \ra \oplus \la e_l^* \ra$, we may take $\z=\la e_l^* \ra\otimes W$ and $\z^*=\la e_l \ra\otimes W$. Then for $\w=(w_1,\cdots,w_{l-1}) \in {W}^{l-1}\simeq X_{l-1}^* \otimes W$ and $\y=w \in W\simeq \z^*$ and $t \in M_{l-1,2}$,
\begin{align}\label{a11}
    &\omega_{\psi,V,W}(t,1)(\phi_1\otimes \phi_2 )(\w;{\bf y}) \\ \nonumber
    &\quad = \psi\Big(\sum_{i=1}^{l-1} t_{i,2} \la w,w_i \ra + \sum_{i,j=1}^{l-1} t_{i,1} t_{j,2} \la w_i,w_j\ra\Big) \phi_1(\w) \phi_{2}\Big(\big(\sum_{i=1}^{l-1} t_{i,1} w_i+w\big)\otimes e_l\Big).
\end{align}

\begin{rem}\label{1dim}
If $V$ is a $1$-dimensional orthogonal space and $e \in V$ is a nonzero element, let $(e, e)_V = 2\lambda$. We denote $W_e \coloneqq V \otimes W$ and identify it with $W$ by endowing $W_e$ with the symplectic form $\langle \cdot, \cdot \rangle_{W_e}$ defined by $\lambda \langle \cdot, \cdot \rangle_W$. Under this identification, the decomposition $(e \otimes Y) \oplus (e \otimes Y^*)$ constitutes a polarization of $W_e$. 

In this setting, we ignore the action of $\G'$ and regard $\Omega_{\psi,V,W}$ as the module $\Omega_{\psi,W_e}$ for $\mc{H}(W_e) \rtimes \wt{\J'}$. By identifying both $\mc{H}(W)$ and $\mc{H}(W_e)$ with the set $W \times F$, the map $(w, t) \mapsto (w, \lambda t)$ provides an isomorphism $\mc{H}(W) \simeq \mc{H}(W_e)$ that intertwines the actions of $\wt{\J'}$. 

Pulling back the projective Schrödinger representation $\overline{\Omega}_{\psi,W_e}$ of $\mc{H}(W_e) \rtimes \J'$ via this map yields the projective representation $\overline{\Omega}_{\psi_{\lambda},W}$ of $\mc{H}(W) \rtimes \J'$. Since $\Omega_{\psi_{\lambda},W}$ is the unique representation extending $\overline{\Omega}_{\psi_{\lambda},W}$ to $\mc{H}(W) \rtimes \wt{\J'}$ (cf. \cite[Corollary~4.19]{Li08}), we conclude that as representations of $\mc{H}(W) \rtimes \wt{\J'}$,
\[
\Omega_{\psi,W_e} \simeq \Omega_{\psi_{\lambda},W}.
\]
\end{rem}

\subsection{Computation of the twisted Jacquet module of the local Weil representation}

In this subsection, we let $V \coloneqq V_n^{\ve}$ and $W \coloneqq W_k$. We investigate the twisted Jacquet modules of the local Weil representation $\omega_{\psi,V,W}$ with respect to certain unipotent subgroups of $\G_n^{\ve}$.

Let $(\omega_{\psi,V,W})_{\mu_{t,\ve}^{\pm}}$ denote the twisted Jacquet module of $\omega_{\psi,V,W}$ with respect to $\wt{\U}_{n,t-1}^{\ve}$ and $\mu_{t,\ve}^{\pm}$. Specifically, this is the quotient space $\omega_{\psi,V,W} / \mc{V}$, where $\mc{V}$ is the subspace spanned by vectors of the form 
\[
\omega_{\psi,V,W}(u) \phi - \mu_{t,\ve}^{\pm}(u) \phi, \quad u \in \wt{\U}_{n,t-1}^{\ve}, \ \phi \in \omega_{\psi,V,W}.
\]
Similarly, we define the twisted Jacquet module $(\omega_{\psi,V,W})_{\mu_{t,\ve}}$ with respect to $\U_{n,t-1}^{\ve}$ and $\mu_{t,\ve}$. Utilizing the explicit action of the Weil representation described in Section \ref{weil}, we establish the following vanishing theorem.

\begin{thm}[cf. {\cite[Prop.~9.4]{GS12}}, {\cite[Prop.~4.1]{MS00}}]\label{tjw} 
For any $2 \le t \le n+\ve$, if $k < t-1$, then 
\[
(\omega_{\psi,V,W})_{\mu_{t,\ve}^{\pm}} = (\omega_{\psi,V,W})_{\mu_{t,\ve}} = 0.
\]
\end{thm}
\begin{proof}
For the case $\ve=1$ and $t=n+1$, a proof was outlined in \cite[Prop.~2.1]{JS03}; however, certain technical details concerning specific orbits were omitted in their discussion (see Remark \ref{miss} below). To address this, we provide a refined proof that incorporates a subtle yet essential adjustment to ensure completeness. Furthermore, our approach accommodates the case $\ve=0$ and $t=n$ where $\H_n^0 = \SO_{2n}$ is quasi-split but non-split, thereby extending the result of \cite[Prop.~4.1]{MS00}, which treats the split case.

It suffices to prove $(\omega_{\psi, V, W})_{\mu_{t,\ve}} = 0$, as this directly implies the vanishing of the extended Jacquet module $(\omega_{\psi, V, W})_{\mu_{t,\ve}^{\pm}}$. Let $V^{\perp}$ be the orthogonal complement of $X_{t-1} \oplus X_{t-1}^*$ in $V$. For notational simplicity, we shall write $Y$ for $Y_m$.

We employ the polarization $(X_{t-1}^* \otimes W) \oplus (V^{\perp} \otimes Y^*)$ of $V \otimes W$ and utilize the explicit action of the Weil representation associated with this polarization. As in Section \ref{weil}, by choosing a basis $\{e_1^*, \dots, e_{t-1}^*\}$ for $X_{t-1}^*$ and a basis $\{e_{t}, \dots, e_{n-1}, e, e', e_{n-1}^*, \dots, e_{t}^*\}$ for $V^{\perp}$, we identify the underlying space with $W^{t-1} \oplus (Y^*)^{2(n-t)+(2+\ve)}$. Let $Gr(\mathbf{w}) = (\langle w_i, w_j \rangle_W)$ denote the Gram matrix. We define a closed subset $\mathbf{W} \subset W^{t-1}$ by
\[
\mathbf{W} \coloneqq \left\{ \mathbf{w} = (w_1, \dots, w_{t-1}) \in W^{t-1} \mid Gr(\mathbf{w}) = 0 \right\},
\]
and set $\overline{\mathbf{W}} \coloneqq \mathbf{W} \times (Y^*)^{2(n-t)+(2+\ve)}$. Let $J_{t}$ denote the twisted Jacquet functor with respect to $\U_{n,t-1}^{\ve}$ and $\mu_{t,\ve}$. We first claim that 
\[
(\omega_{\psi,V,W})_{\mu_{t,\ve}} \simeq J_{t}(\mc{S}(\overline{\mathbf{W}})).
\]

Note that $\overline{\mathbf{W}}$ is a closed subset of $W^{t-1} \times (Y^*)^{2(n-t)+(2+\ve)}$. To simplify the notation, let $\mathcal{V} = W^{t-1} \times (Y^*)^{2(n-t)+(2+\ve)}$. According to the excision principle of Bernstein and Zelevinsky \cite{BZ76}, we have the following exact sequence of Schwartz spaces:
\[
\xymatrix{0 \ar[r] & \mc{S}(\mathcal{V} \setminus \overline{\mathbf{W}}) \ar[r]^-{\overline{i}} & \mc{S}(\mathcal{V}) \ar[r]^-{\overline{\text{res}}} & \mc{S}(\overline{\mathbf{W}}) \ar[r] & 0},
\]
where $\overline{i}$ is the map induced by the open inclusion and $\overline{\text{res}}$ is the restriction map. Since the functor $J_{t}$ is exact, it induces an exact sequence
\[
\xymatrix{0 \ar[r] & J_{t}(\mc{S}(\mathcal{V} \setminus \overline{\mathbf{W}})) \ar[r]^-{\overline{\text{i}}} & J_{t}(\mc{S}(\mathcal{V})) \ar[r]^-{\overline{\text{res}}} & J_{t}(\mc{S}(\overline{\mathbf{W}})) \ar[r] & 0}.
\]

By the definition of $\overline{\mathbf{W}}$ and the action formula \eqref{a7}, $J_{t}\big(\mc{S}(\mathcal{V}\bs \overline{\mathbf{W}})\big)=0$.
 Thus, we obtain the desired isomorphism
\[
J_{t}(\mc{S}(\mathcal{V})) \simeq J_{t}(\mc{S}(\overline{\mathbf{W}})).
\]
This establishes our first claim. We now consider the action of $Z_{t-1} \times \wt{\J}_k$ on $\mathbf{W} \subset W^{t-1}$ inherited from the Weil representation $\omega_{\psi,V,W}$, which is given by
\[
(w_1, \dots, w_{t-1}) \cdot (z, (h', \ve')) = (w_1 \cdot (h')^{-1}, \dots, w_{t-1} \cdot (h')^{-1}) z.
\]
By exploiting the $Z_{t-1}$-action on $\mathbf{W}$, we can choose representatives for the $(Z_{t-1} \times \wt{\J}_k)$-orbits of $\mathbf{W}$ in the following form:
\[
(0, \dots, 0, w_{t_1}, 0, \dots, 0, w_{t_{j-1}}, 0, \dots, 0, w_{t_j}, 0, \dots, 0) \in \mathbf{W} \subset W^{t-1},
\]
for some $1 \le j \le k < t-1$, where $\{w_{t_1}, \dots, w_{t_j}\}$ is a linearly independent set in $W$. Furthermore, by the Witt extension theorem and the $\wt{\J}_k$-action on $\mathbf{W}$, we may choose even more restrictive representatives of the form:
\[
(0, \dots, 0, f_1, 0, \dots, 0, f_{j-1}, 0, \dots, 0, f_j, 0, \dots, 0) \in \mathbf{W} \subset W^{t-1},
\]
where $\{f_1, \dots, f_j\}$ is the standard partial basis of $W$ as defined in Section \ref{subsec:2-1}.

Let $\{ {\bf w}_i \}_{1 \le i \le N}$ denote these finite representatives in $\mathbf{W}$, and let $\{ C_i \}_{1 \le i \le N}$ be the corresponding $(Z_{t-1} \times \wt{\J}_k)$-orbits defined by $C_i \coloneqq {\bf w}_i \cdot (Z_{t-1} \times \wt{\J}_k)$. By reordering the indices if necessary, we may assume that the orbits are arranged in non-decreasing order of dimension, i.e., $\dim(C_i) \le \dim(C_{i+1})$ for $1 \le i \le N-1$.

For each $1 \le j \le N$, note that $C_j$ is a closed subset of the union $\bigcup_{i \ge j} C_i$. Consequently, the excision principle \cite{BZ76} provides the following exact sequence of Schwartz spaces:
\begin{align*} 
\xymatrix{0 \ar[r] & \mc{S}\big(\bigcup_{i \ge j+1} C_i \big) \ar[r] & \mc{S}\big(\bigcup_{i \ge j} C_i \big) \ar[r] & \mc{S}(C_j) \ar[r] & 0.}
\end{align*} 
To simplify the notation, we set $\mc{Y} \coloneqq (Y^*)^{2(n-t)+(2+\ve)}$ and for any subset $X \subset \mathbf{W}$, let $\mathcal{F}(X) \coloneqq J_{t}(\mc{S}(X)\otimes \mc{S}(\mc{Y}))$. Since the tensor product and the twisted Jacquet functor $J_{t}$ are right exact, we obtain the following right exact sequence:
\begin{equation}\label{ex}
\xymatrix{
\mathcal{F}(\bigcup_{i \ge j+1} C_i ) \ar[r] & \mathcal{F}(\bigcup_{i \ge j} C_i) \ar[r] & \mathcal{F}(C_j) \ar[r] & 0.
}
\end{equation}
 We claim that $\mathcal{F}(C_j) = 0$ for all $1 \le j \le N$.

Suppose, toward a contradiction, that $\mathcal{F}(C_{i_0})  \neq 0$ for some $1 \le i_0 \le N$. Let 
\[
{\bf{w}}_{i_0} = (0, \dots, 0, f_1, 0, \dots, 0, f_{s-1}, 0, \dots, 0, f_s, 0, \dots, 0), \quad s \le k < t-1,
\] 
be a representative of the orbit $C_{i_0}$, and let $R_{i_0}$ denote the stabilizer of ${\bf{w}}_{i_0}$ in $Z_{t-1} \times \wt{\J}_k$. We consider the map 
\[ 
\Phi_{{\bf{w}}_{i_0}} : \mc{S}(C_{i_0}) \longrightarrow \ind_{R_{i_0}}^{Z_{t-1} \times \wt{\J}_k} \II, \quad \phi \longmapsto [ (g, \wt{h'}) \mapsto (\omega_{\psi,V,W}(g, \wt{h'}) \phi)({\bf{w}}_{i_0}) ].
\]
It can be readily verified that $\Phi_{{\bf{w}}_{i_0}}$ is an isomorphism of $(Z_{t-1} \times \wt{\J}_k)$-modules. Consequently, we have
\[
\mc{S}(C_{i_0}) \otimes \mc{S}(\mc{Y}) \simeq \left( \ind_{R_{i_0}}^{Z_{t-1} \times \wt{\J}_k} \II \right) \otimes \mc{S}(\mc{Y}).
\]

First, consider the case where ${\bf{w}}_{i_0} \neq (f_1, \dots, f_{t-2}, 0)$. In this case, since $s < t-1$, there exists at least one simple root subgroup $L$ of $Z_{t-1}$ such that the character $\mu_{t,\ve}$ is non-trivial on $L$, while $L \times \{1\}$ is contained in the stabilizer $R_{i_0}$. This implies that $\mathcal{F}(C_{i_0}) = 0$, which contradicts our initial assumption.

Next, we address the critical case where ${\bf{w}}_{i_0} = (f_1, \dots, f_{t-2}, 0)$. To handle this orbit, we consider the subgroup $\overline{M}_{t}^{\ve} \subset M_{t-1, 2(n-t)+(2+\ve)}$ defined by
\[
\overline{M}_{t}^{\ve} \coloneqq \{ \tau \in M_{t-1, 2(n-t)+(2+\ve)} \mid \tau_{p,q} = 0 \text{ for } p \neq t-1 \}.
\]
For any $\phi \in \mc{S}(C_{i_0}) \otimes \mc{S}(\mc{Y})$ and $\tau \in \overline{M}_{t}^{\ve}$, let $\mathbf{y} = (y_1, \dots, y_{2(n-t)+(2+\ve)}) \in \mc{Y}$. By the action formula \eqref{a9}, we have
\begin{equation*}
\begin{aligned}
    &\left(\omega_{\psi,V,W}(\tau, 1) \phi\right)({\bf{w}}_{i_0}, \mathbf{y}) \\
    &= \psi\left( \sum_{j=1}^{2(n-t)+(2+\ve)}\sum_{i=1}^{t-2} \tau_{i,j} \langle f_i, y_j \rangle_{W_k} \right) \phi({\bf{w}}_{i_0}, \mathbf{y}) = \phi({\bf{w}}_{i_0}, \mathbf{y}),
\end{aligned}
\end{equation*}
where the last equality holds because $\tau_{i,j} = 0$ for all $1 \le i \le t-2$ by the definition of $\overline{M}_{t}^{\ve}$. However, since the character $\mu_{t,\ve}$ is non-trivial on $\overline{M}_{t}^{\ve}$, we obtain a contradiction.

Consequently, $\mathcal{F}(C_{j}) = 0$ for all $1 \le j \le N$. By applying the exact sequence \eqref{ex} repeatedly through the filtration of orbits, we conclude that
\[
\mathcal{F}(\mathbf{W})=J_{t}(\mc{S}(\mathbf{W}) \otimes \mc{S}(\mc{Y})) = 0.
\]
Given the identification $\mc{S}(\overline{\mathbf{W}}) \simeq \mc{S}(\mathbf{W}) \otimes \mc{S}(\mc{Y})$, it follows that $J_{t}(\mc{S}(\overline{\mathbf{W}})) = 0$, as desired.
\end{proof}
\begin{rem}\label{miss}
In \cite[Prop.~2.1]{JS03}, the argument overlooks the critical case where the representative ${\bf{w}}_{i_0}$ of an orbit $C_{i_0}$ takes the form $(f_1, \dots, f_{n-1}, 0)$. As demonstrated in our proof, to treat this case rigorously, one must invoke the action of $\Hom(V^{\perp}, X_{n-1})$ within the Weil representation $\omega_{\psi,V,W}$ as formulated in \eqref{a9}. This refinement ensures the completeness of the vanishing theorem by accounting for all possible orbits in the decomposition of $\mathbf{W}$.
\end{rem}

For $\wt{\pi} \in \Irr(\G_n^{\ve})$, the maximal $\wt{\pi}$-isotypic quotient of $\omega_{\psi,V,W}$ is of the form
\[
\wt{\pi} \boxtimes \Theta_{\psi,V,W}(\wt{\pi})
\]
for some smooth representation $\Theta_{\psi,V,W}(\wt{\pi})$ of $\wt{\J}_k$ of finite length, which we call the \textit{big theta lift} of $\wt{\pi}$. The maximal semisimple quotient of $\Theta_{\psi,V,W}(\wt{\pi})$ is denoted by $\theta_{\psi,V,W}(\wt{\pi})$ and is called the \textit{small theta lift} of $\wt{\pi}$. Note that $\theta_{\psi,V,W}(\wt{\pi})$ is a non-genuine (resp. genuine) representation of $\wt{\J}_k$ if $\ve=0$ (resp. $\ve=1$).

Similarly, for $\sigma \in \Irr(\wt{\J}_k)$, we define the big theta lift $\Theta_{\psi,W,V}(\sigma)$ as a smooth representation of $\G_n^{\ve}$ (or $\H_n^{\ve}$) of finite length such that 
\[
\Theta_{\psi,W,V}(\sigma) \boxtimes \sigma
\]
is the maximal $\sigma$-isotypic quotient of $\omega_{\psi,V,W}$. Its maximal semisimple quotient is denoted by $\theta_{\psi,W,V}(\sigma)$. Here, if $\sigma$ is a non-genuine (resp. genuine) representation of $\wt{\J}_k$, then the corresponding representation of $\G_n^{\ve}$ is for $\ve=0$ (resp. $\ve=1$). According to the Howe duality conjecture, which is now a theorem (\cite{GT1, GT2, Wa90}), both $\theta_{\psi,V,W}(\wt{\pi})$ and $\theta_{\psi,W,V}(\sigma)$ are irreducible whenever they are nonzero.

The following proposition demonstrates that the existence of a Bessel model imposes a vanishing condition on the theta lifts. In the case where $\H_n^{\ve}$ is split and $t=n+\ve$, this result was established in \cite[Prop.~3.3]{GRS97} and \cite[Cor.~4.1]{MS00} for $\ve=0$, and in \cite[Prop.~2.1]{JS03} and \cite[Cor.~9.5]{GS12} for $\ve=1$.

\begin{prop} \label{c1}
For $2 \le t \le n+\ve$, let $\wt{\pi} \in \Irr(\G_n^{\ve})$ (resp. $\pi \in \Irr(\H_n^{\ve})$) be of $\mu_{t,\ve}^{\pm}$ (resp. $\mu_{t,\ve}$)-Bessel type. If $k < t-1$, then the big theta lift $\Theta_{\psi,V,W}(\wt{\pi})$ (resp. $\Theta_{\psi,V,W}(\pi)$) vanishes.
\end{prop}

\begin{proof} 
Suppose that $\Theta_{\psi,V,W}(\wt{\pi})$ is nonzero. By the definition of the big theta lift, we have a non-trivial $(\G_n^{\ve} \times \wt{\J}_k)$-equivariant homomorphism
\[
\Hom_{\G_n^{\ve} \times \wt{\J}_k}(\omega_{\psi,V,W}, \wt{\pi} \boxtimes \Theta_{\psi,V,W}(\wt{\pi})) \neq 0.
\]
Since $\wt{\pi}$ is of $\mu_{t,\ve}^{\pm}$-Bessel type, there exists a nonzero linear functional on $\wt{\pi}$ that is $(\wt{\U}_{n,t-1}^{\ve}, \mu_{t,\ve}^{\pm})$-equivariant, which implies $\wt{\pi} \hookrightarrow \Ind_{\wt{\U}_{n,t-1}^{\ve}}^{\G_n^{\ve}}(\mu_{t,\ve}^{\pm})$. Consequently, we obtain
\[
\Hom_{\G_n^{\ve} \times \wt{\J}_k}(\omega_{\psi,V,W}, \Ind_{\wt{\U}_{n,t-1}^{\ve}}^{\G_n^{\ve}}(\mu_{t,\ve}^{\pm}) \boxtimes \Theta_{\psi,V,W}(\wt{\pi})) \neq 0.
\]
By Frobenius reciprocity, this leads to
\[
\Hom_{\wt{\J}_k}((\omega_{\psi,V,W})_{\mu_{t,\ve}^{\pm}}, \Theta_{\psi,V,W}(\wt{\pi})) \neq 0,
\]
where $(\omega_{\psi,V,W})_{\mu_{t,\ve}^{\pm}}$ denotes the twisted Jacquet module. However, this contradicts Theorem \ref{tjw}, which states that the Jacquet module vanishes for $k < t-1$. The proof for $\pi$ of $\mu_{t,\ve}$-Bessel type follows analogously and is thus omitted.
\end{proof}
Using arguments analogous to those presented above, one can establish the following proposition, which addresses the vanishing of theta lifts for representations of Fourier--Jacobi type. For the case where $V$ is split and $t=k$, this result was proved in \cite[Prop.~2.4]{GRS97} and \cite[Cor.~2.2]{JS03}. Since the proof follows the same logic as that of Proposition \ref{c1}, we omit the details. 

\begin{prop}\label{con}
For $2 \le t \le k$, let $\sigma \in \Irr(\wt{\J}_k)$ be a representation of $(\mu_{t,v}', \lambda_{\ve_{\sigma}})$-Fourier--Jacobi type. Then, for $n < t - \ve_{\sigma}$, the big theta lift $\Theta_{\psi^{-1},W,V}(\sigma)$ vanishes.
\end{prop}


\section{The global theta lifts from $\G_n^{\ve}$ to $\wt{\J}_{k}$}
\label{sec:4}

Before delving into the technical details, we briefly outline the overarching structure of Sections 4 through 7. Because the arguments in these sections are quite substantial, the following table serves as a compact guide to help the reader keep the overall proof architecture in mind. Sections 4 and 5 treat the Bessel case, while Sections 6 and 7 treat the Fourier--Jacobi case. Within each pair, the first section relates the relevant period to a global theta lift, and the second section connects the same period to the analytic behavior of an $L$-function.\bb

\begin{center}
\renewcommand{\arraystretch}{1.2}
\begin{tabular}{c|c|c}
\hline
\textbf{Period type} & \textbf{Theta-lift relation} & \textbf{$L$-function relation} \\
\hline
Bessel & Section 4 & Section 5 \\
Fourier--Jacobi & Section 6 & Section 7 \\
\hline
\end{tabular}
\end{center}
\bb
In this section, we investigate the Whittaker periods of the global theta lifts from $\G_n^{\ve}(\A)$ to $\wt{\J}_k(\A)$. Throughout the remainder of this paper, when considering $(\G_n^{\ve}, \wt{\J}_k)$ as a dual reductive pair in the case $\ve = 0$, we canonically identify $\wt{\J}_k$ with $\J_k$. Under this identification, an element $\wt{h'} = (h', \ve') \in \wt{\J}_k$ is understood simply as $h' \in \J_k$.

We consider the global Weil representation $\omega_{\psi,V_n^{\ve},W_k} \coloneqq \bigotimes_v' \omega_{\psi_v,V_{n,v}^{\ve},W_{k,v}}$ of $\G_n^{\ve}(\A) \times \wt{\J}_k(\A)$. It is realized in the Bruhat--Schwartz space $\mc{S}((V_n^{\ve} \otimes Y_k^*)(\A)) = \bigotimes_v' \mc{S}(V_{n,v}^{\ve} \otimes Y_{k,v}^*)$. We define a symplectic form $( \cdot , \cdot )$ on $V_n^{\ve} \otimes W_k$ by
\[
(v_1 \otimes w_1, v_2 \otimes w_2) \coloneqq \la v_1, v_2 \ra_{V_n^{\ve}} \cdot \la w_1, w_2 \ra_{W_k}.
\]
For simplicity of notation, we shall write $V \coloneqq V_{n}^{\ve}$, $W \coloneqq W_k$, and $Y \coloneqq Y_k$.

Let $\P_k' = \M_k' \N_k'$ be the parabolic subgroup of $\J_k$ stabilizing the maximal isotropic subspace $Y$, with Levi component $\M_k'$ and unipotent radical $\N_k'$. We have the natural identifications
\[
\M_k' \simeq \GL(Y) \quad \text{and} \quad \N_k' \simeq \{ \alpha \in \Hom(Y^*, Y) \mid \alpha^* = -\alpha \},
\]
where $\alpha^* \in \Hom(Y^*, Y)$ is the adjoint element satisfying
\[
\la \alpha y_1, y_2 \ra = \la y_1, \alpha^* y_2 \ra \quad \text{for all } y_1, y_2 \in Y^*.
\]
Similarly, for $m \in \GL(Y)$, we denote by $m^* \in \GL(Y^*)$ the adjoint element satisfying
\[
\la m x, y \ra = \la x, m^* y \ra \quad \text{for all } x \in Y, \ y \in Y^*.
\]

Let $\Z_k'$ be the standard maximal unipotent subgroup of $\M_k'$, regarded as a subgroup of $\GL(Y)$ via the isomorphism $\M_k' \simeq \GL(Y)$. From the action of the local Weil representations, we obtain the global action of $\G_n^{\ve}(\A) \times (\Z_k'(\A) \N_k'(\A))$ on $\omega_{\psi,V,W}$ as follows. For $\phi \in \mc{S}((V \otimes Y^*)(\A))$ and $\tb{x} \in (V \otimes Y^*)(\A)$:
\begin{itemize}
    \item $\omega_{\psi,V,W}(g, 1)\phi(\tb{x}) = \phi(g^{-1} \cdot \tb{x})$ for $g \in \G_n^{\ve}(\A)$,
    \item $\omega_{\psi,V,W}(1, z)\phi(\tb{x}) = \phi(z^* \cdot \tb{x})$ for $z \in \Z_k'(\A) \subset \U_k'(\A)$,
    \item $\omega_{\psi,V,W}(1, n)\phi(\tb{x}) = \psi\left(\frac{1}{2}(n \cdot \tb{x}, \tb{x})\right)\phi(\tb{x})$ for $n \in \N_k'(\A) \subset \U_k'(\A)$.
\end{itemize}

There is a $(\G_n^{\ve}(\A) \times \wt{\J}_k(\A))$-equivariant map $\theta_{\psi,V,W} : \mc{S}((V \otimes Y^*)(\A)) \to \mc{A}(\G_n^{\ve} \times \wt{\J}_k)$ given by the theta series
\[
\theta_{\psi,V,W}(\phi; g, \wt{h'}) \coloneqq \sum_{\tb{x} \in (V \otimes Y^*)(F)} \left( \omega_{\psi,V,W}(g, \wt{h'}) \phi \right)(\tb{x}).
\]

We normalize the Haar measure $dg$ on $\G_{n}^{\ve}(\A)$ such that
\[
\int_{[\G_{n}^{\ve}]} f(g) \, dg = \int_{[\mu_2]} \int_{[\H_{n}^{\ve}]} f((\t \cdot \e) \cdot h) \, dh \, d\t
\]
for any smooth function $f$ on $\G_{n}^{\ve}(\A)$, whenever the integral on the right-hand side is absolutely convergent. Here, $[G]$ denotes the quotient space $G(F) \bs G(\A)$.

For $\TT \in \mf{T}$, $f \in \mc{A}(\G_n^{\ve})$, and $\phi \in \mc{S}((V \otimes Y^*)(\A))$, we define the global theta lift to $\wt{\J}_k(\A)$ by
\begin{align*}
\theta_{\psi,V,W}^{\TT}(\phi, f)(\wt{h'}) 
&\coloneqq \int_{[\G_n^{\ve}]} \theta_{\psi,V,W}(\phi; g, \wt{h'}) f(g) \mathrm{sgn}_{\TT}(g) \, dg \\
&= \int_{[\H_n^{\ve}]} \int_{[\mu_2]} \theta_{\psi,V,W}(\phi; (\t \cdot \e) \cdot h, \wt{h'})  f((\t \cdot \e) \cdot h) \mathrm{sgn}_{\TT}(\t \cdot \e) \, d\t \, dh.
\end{align*}
Analogously, for $f \in \mc{A}(\H_n^{\ve})$, we set
\[
\theta_{\psi,V,W}(\phi, f)(\wt{h'}) \coloneqq \int_{[\H_n^{\ve}]} \theta_{\psi,V,W}(\phi; h, \wt{h'}) f(h) \, dh.
\]

For an irreducible cuspidal automorphic representation $\wt{\pi}$ of $\G_n^{\ve}(\A)$ (resp. $\pi$ of $\H_n^{\ve}(\A)$), we define the space of global theta lifts
\[
\Theta_{\psi,V,W}^{\TT}(\wt{\pi}) \coloneqq \{ \theta_{\psi,V,W}^{\TT}(\phi, f) \mid \phi \in \omega_{\psi,V,W}, \ f \in \wt{\pi} \}
\]
\[
\big(\text{resp. } \Theta_{\psi,V,W}(\pi) \coloneqq \{ \theta_{\psi,V,W}(\phi, f) \mid \phi \in \omega_{\psi,V,W}, \ f \in \pi \} \big).
\]
When $\TT = \emptyset$, we simply denote $\Theta_{\psi,V,W}^{\emptyset}(\wt{\pi})$ by $\Theta_{\psi,V,W}(\wt{\pi})$. We note that $\Theta_{\psi,V,W}^{\TT}(\wt{\pi}) = \Theta_{\psi,V,W}(\wt{\pi} \otimes \mathrm{sgn}_{\TT})$.

The space $\Theta_{\psi,V,W}^{\TT}(\wt{\pi})$ (resp. $\Theta_{\psi,V,W}(\pi)$) forms an automorphic representation of $\wt{\J}_k(\A)$, which is genuine if $\ve=1$ and non-genuine if $\ve=0$. If it is square-integrable, it is irreducible (cf. \cite[Prop.~3.1]{Gan23}). Furthermore, if $k_0$ is the first occurrence index in the theta tower such that $\Theta_{\psi,V,W_{k_0}}^{\TT}(\wt{\pi}) \neq 0$ (resp. $\Theta_{\psi,V,W_{k_0}}(\pi) \neq 0$), then the lift is cuspidal by Rallis' tower property, and hence irreducible.

The following theorem demonstrates that the non-vanishing of the Bessel period $\mc{B}_{k, \psi}^{\ve,\TT}$ (resp. $\mc{B}_{k, \psi}^{\ve}$) is equivalent to the non-vanishing and genericity of $\Theta_{\psi, V, W_k}^{\TT}(\wt{\pi})$ (resp. $\Theta_{\psi, V, W_k}(\pi)$). For the case $\ve=1$, this was established in \cite[Prop.~1]{Fu95} specifically for $k = n$. For the case $\ve=0$, when $\H_n^0$ is split of type $(1,1)$, the statement appears in \cite[Prop.~3.2 and 3.5]{GRS97}.

\begin{thm}\label{a1} Let $\wt{\pi}$ (resp. $\pi$) be an irreducible cuspidal representation of $\G_n^{\ve}(\A)$ (resp. $\H_n^{\ve}(\A)$) and $\TT \in \mf{T}$. For $k>n+\ve$, if $\Theta_{\psi,V,W_{k}}^{\TT}(\wt{\pi})$ (resp. $\Theta_{\psi,V,W_{k}}(\pi)$) is nonzero, it is non-generic with respect to  $(\psi,\lambda)$ for any $\lambda \in F^{\times}$. For $k=n+\ve -1\text{ or } n+\ve$ and $k\ge 1$, if $\Theta_{\psi,V,W_{k}}^{\TT}(\wt{\pi})$ (resp. $\Theta_{\psi,V,W_{k}}(\pi)$) is nonzero and $(\psi,\lambda)$-generic for some $\lambda \in F^{\times}$, then $\lambda\equiv\lambda_{\ve} \pmod{F_{\ve}^{\times}}$.
Furthermore, $\Theta_{\psi,V,W_{k}}^{\TT}(\wt{\pi})$ (resp. $\Theta_{\psi,V,W_{k}}(\pi)$) is nonzero and $(\psi,\lambda_{\ve})$-generic if and only if $\mc{B}_{k,\psi}^{\ve,\TT} \ne 0$ on $\wt{\pi}$ (resp. $\mc{B}_{k,\psi}^{\ve
} \ne 0$ on $\pi$).
\end{thm}
\begin{proof}
Since the proofs for the $\G_n^{\ve}$ and $\H_n^{\ve}$ cases are largely identical, we present the proof only for the $\G_n^{\ve}$ case. We employ the following ordered bases when representing elements of $\G_n^{\ve}(\A)$ as matrices:
\[
\begin{cases}
\{e_1, \dots, e_{n-1}, e, e', e_{n-1}^*, \dots, e_1^*\} & \text{(case $\ve=0$)} \\
\{e_1, \dots, e_{n}, e, e_{n}^*, \dots, e_1^*\} & \text{(case $\ve=1$).}
\end{cases}
\]
Similarly, elements of $\J_{k}(\A)$ are written as matrices with respect to the basis $\{f_1, \dots, f_{k}, f_{k}^*, \dots, f_{1}^*\}$.

For $z \in Z_{k}$, we set
\[
v'(z) \coloneqq \begin{pmatrix} z & 0 \\ 0 & z^* \end{pmatrix} \in \Z_k' \subset \wt{\J}_{k},
\]
and for $s' \in S_{k}'$, we set
\[
n'(s') \coloneqq \begin{pmatrix} \mathrm{Id}_{k} & s' \\ 0 & \mathrm{Id}_{k} \end{pmatrix} \in \U_k' \subset \wt{\J}_{k}.
\]
(For the precise definitions of $Z_k$ and $S_{k}'$, we refer the reader to Section \ref{subsec:2-1}).

Then $n'(S_{k}')v'(Z_{k})$ constitutes the standard maximal unipotent subgroup $\U_{k}'(\A)$ of $\wt{\J}_{k}(\A)$. We choose the Haar measures $du'$, $ds'$, and $dz$ on the respective unipotent subgroups to be consistent with the normalization established in Section \ref{subsec:2-1}, such that $du' = ds' \, dz$.

Using the basis $\{f_1^*, \dots, f_{k}^*\}$ of $Y^*$, we identify elements of $V \otimes Y^*$ with $k$-tuples $(x_1, \dots, x_k) \in V^k$. The actions of $Z_{k}(\A)$ and $S_{k}'(\A)$ on the Weil representation $\omega_{\psi,V,W_{k}}$ can be explicitly described via $v'(Z_{k})$ and $n'(S_{k}')$ as follows:
\begin{align}
    \label{w2} 
    (\omega_{\psi,V,W_k}(\mathbf{1}, v'(z)) \phi)(x_1, \dots, x_{k}) 
    &= \phi((x_1, \dots, x_{k}) \cdot z) 
    &&\text{for } z \in Z_{k}(\A), \\[10pt]
    \label{w1} 
    (\omega_{\psi,V,W_{k}}(\mathbf{1}, n'(s')) \phi)(x_1, \dots, x_{k}) 
    &= \psi\left(\frac{1}{2}\mathrm{tr}(Gr(\mathbf{x}) \cdot s' \cdot w_{k})\right) \phi(x_1, \dots, x_{k}) 
    &&\text{for } s' \in S_{k}'(\A).
\end{align}
where $\mathbf{x} = (x_1, \dots, x_{k}) \in V(\A)^{k}$ and $Gr(\mathbf{x}) = (\la x_i, x_j \ra_{V})$ is the Gram matrix.

For $f \in \Theta_{\psi,V,W_k}^{\TT}(\wt{\pi})$, we compute its $(\psi, \lambda)$-Whittaker period:
\[
W_{\psi}^{\lambda}(f) = \int_{[\U_{k}']} (\mu_k' \cdot \psi_{\lambda})^{-1}(u') f(u') \, du'.
\]
This integral can be rewritten as
\[
W_{\psi}^{\lambda}(f) = \int_{[Z_{k}]} (\mu_k')^{-1}(v'(z)) \int_{[S_{k}']}  \psi_{\lambda}^{-1}(n'(s')) f\big(n'(s')v'(z)\big) \, ds' \, dz.
\]
We first evaluate the partial Fourier coefficient $W_{\lambda}'(f)$ defined by
\[
W_{\lambda}'(f) \coloneqq \int_{[S_{k}']}  \psi_{\lambda}^{-1}(n'(s')) f(n'(s')) \, ds'.
\]

Recall that $f$ is given by
\begin{align*}
f(\wt{h'}) &= \int_{[\G_n^{\ve}]} \theta_{\psi,V,W_k}(\phi; g, \wt{h'}) \wt{\varphi}(g) \mathrm{sgn}_{\TT}(g) \, dg \\
&= \int_{[\H_n^{\ve}]} \int_{[\mu_2]} \theta_{\psi,V,W_k}(\phi; (\t \cdot \e) h, \wt{h'}) \wt{\varphi}((\t \cdot \e) h) \mathrm{sgn}_{\TT}(\t \cdot \e) \, d\t \, dh
\end{align*}
for some $\wt{h'} \in \wt{\J}_k(\A)$, $\phi \in \mc{S}((V \otimes Y_k^*)(\A))$, and $\wt{\varphi} \in \wt{\pi}$. 

Substituting this into $W_{\lambda}'(f)$ and unfolding the theta series, we obtain
\begin{align*}
W_{\lambda}'(f) &= \int_{[S_{k}']} \psi_{\lambda}^{-1}(n'(s')) \left( \int_{[\G_n^{\ve}]} \theta_{\psi,V,W_k}(\phi; g, n'(s')) \wt{\varphi}(g) \mathrm{sgn}_{\TT}(g) \, dg \right) ds' \\
&= \int_{[S_{k}']}  \psi_{\lambda}^{-1}(n'(s')) \bigg( \int_{[\G_{n}^{\ve}]} \sum_{\mathbf{x} \in V(F)^k} \big(\omega_{\psi,V,W_k}(g, n'(s'))\phi\big)(\mathbf{x})  \wt{\varphi}(g) \mathrm{sgn}_{\TT}(g) \, dg \bigg) ds' \\
&= \int_{[\G_{n}^{\ve}]} \left( \sum_{(x_1, \dots, x_{k}) \in \VV} \big(\omega_{\psi,V,W_k}(g, \1)\phi\big)(x_1, \dots, x_{k}) \right) \wt{\varphi}(g) \mathrm{sgn}_{\TT}(g) \, dg.
\end{align*}
where 
\[
\VV \coloneqq \left\{ (x_1, \dots, x_{k}) \in V(F)^{k} \mid Gr(x_1, \dots, x_{k}) = 
\begin{pmatrix} 
0 & \cdots & 0 \\ 
\vdots & \ddots & \vdots \\ 
0 & \cdots & 2\lambda 
\end{pmatrix} 
\right\}
\]
and the last equality follows from \eqref{w1}, since the integral over $S_k'$ vanishes unless $(x_1, \dots, x_{k}) \in \VV$.

The representation $\omega_{\psi,V,W_k}$ induces an action of $\G_n^{\ve}(F) \times Z_{k}(F)$ on $\VV \subset V(F)^{k}$ given by
\[
(x_1, \dots, x_{k}) \cdot (g, z) = (x_1 \cdot g^{-1}, \dots, x_{k} \cdot g^{-1}) z.
\]
Choose an arbitrary element $(y_1, \dots, y_k) \in \VV$. If $k > n+\ve$, we claim that the set $\{y_1, \dots, y_{k-1}\} \subset V(F)$ is linearly dependent. Suppose, for the sake of contradiction, that it is linearly independent. Since the dimension of a maximal isotropic subspace of $V$ is bounded by $n$, this implies $k = n+1$ and $\ve = 0$. We may then choose vectors $\{y_1^*, \dots, y_{k-1}^*\}$ in $V$ such that $\la y_i, y_j^* \ra_V = \delta_{ij}$. Because $\{y_1, \dots, y_{k-1}, y_1^*, \dots, y_{k-1}^*\}$ forms a basis for $V$, $y_k$ can be written as a linear combination of these basis elements. However, since $y_k$ is orthogonal to all $y_i$ and $y_j^*$ for $i,j < k$, this forces $y_k = 0$, which contradicts the condition $\la y_k, y_k \ra_V = 2\lambda \neq 0$. Thus, our claim holds.

Assume now that $\{y_1, \dots, y_{k-1}\} \subset V(F)$ is linearly dependent, and let $\mathbf{y} = (y_1, \dots, y_{k-1})$. There exist an integer $1 \le i \le k-1$ (noting $i \neq k$ since $\la y_k, y_k \ra_V = 2\lambda \neq 0$) and an element $z_0 \in Z_k(F)$ such that
\[
(y_{1}, \dots, y_k) \cdot z_0 = (y_{1}, \dots, y_{i-1}, 0, y_{i+1}, \dots, y_{k}).
\]
Let $\mathbf{y}_0 = (y_{1}, \dots, 0, \dots, y_k) \in \VV$. We define the integral
\begin{equation*}
\begin{split}
&W_{\phi,\wt{\varphi}}^{\psi,\lambda,\TT}(\mathbf{y}, \wt{h'}) \coloneqq \int_{[Z_{k}]} \mu_k'(v'(z))^{-1} \int_{[\G_{n}^{\ve}]} \big(\omega_{\psi,V,W_k}(g, v'(z)\wt{h'})\phi\big)(\mathbf{y}) \wt{\varphi}(g) \mathrm{sgn}_{\TT}(g) \, dg \, dz.
\end{split}
\end{equation*}
Since $z_0 \in Z_k(F)$, we have $W_{\phi,\wt{\varphi}}^{\psi,\lambda,\TT}(\mathbf{y}, \wt{h'}) = \mu_k'(v'(z_0))^{-1}  \cdot W_{\phi,\wt{\varphi}}^{\psi,\lambda,\TT}(\mathbf{y}_0, \wt{h'})$.

Observe that there is a simple root subgroup $L$ of $Z_k$ that stabilizes $\mathbf{y}_0$, and $\mu_k'$ is non-trivial on $v'(L(\A))$. Choosing an element $l_0 \in L(\A)$ such that $\mu_k'(v'(l_0)) \neq 1$, the invariance of $\mathbf{y}_0$ under $L$ yields
\[
W_{\phi,\wt{\varphi}}^{\psi,\lambda,\TT}(\mathbf{y}_0, l_0) = W_{\phi,\wt{\varphi}}^{\psi,\lambda,\TT}(\mathbf{y}_0, \1).
\]
On the other hand, applying the change of variables $z \mapsto z l_0^{-1}$ gives
\[
W_{\phi,\wt{\varphi}}^{\psi,\lambda,\TT}(\mathbf{y}_0, l_0) = \mu_k'(v'(l_0)) W_{\phi,\wt{\varphi}}^{\psi,\lambda,\TT}(\mathbf{y}_0, \1).
\]
Consequently, $W_{\phi,\wt{\varphi}}^{\psi,\lambda,\TT}(\mathbf{y}, \1) = \mu_k'(v'(z_0))^{-1}  \cdot W_{\phi,\wt{\varphi}}^{\psi,\lambda,\TT}(\mathbf{y}_0, \1) = 0$. Since $W_{\psi}^{\lambda}(f)$ is the sum of $W_{\phi,\wt{\varphi}}^{\psi,\lambda,\TT}(\mathbf{y}, \1)$ over all $\mathbf{y} \in \VV$, it follows that $W_{\psi}^{\lambda}(f) = 0$ when $k > n+\ve$. This establishes the first statement of the theorem.

Now we consider the cases $k = n-1+\ve$ or $k = n+\ve$. Define $V_{\lambda} \coloneqq \{v_0 \in V^{\ve} \mid \la v_0, v_0 \ra_V = 2\lambda\}$. From the above argument, if $\Theta_{\psi,V,W_{k}}^{\TT}(\wt{\pi}) \neq 0$, we must have $\VV \neq \emptyset$, which implies $V_{\lambda} \neq \emptyset$. Choosing any element $v_0 \in V_{\lambda}$, we can write $v_0 = a e + b e'$ for some $a, b \in F$. Then
\[
\la v_0, v_0 \ra_V = 
\begin{cases} 
2c(a^2 - db^2) & \text{if } \ve=0, \\ 
2da^2 & \text{if } \ve=1. 
\end{cases}
\]
Consequently, $\lambda \equiv \lambda_{\ve} \pmod{F_{\ve}^{\times}}$. This proves the second statement of the theorem.

Returning to the integration over $\VV$, if an element $(x_1, \dots, x_k) \in \VV$ contributes a non-trivial summand to $W_{\lambda_{\ve}}'(f)$, the subset $\{x_1, \dots, x_k\} \subset V(F)$ must be linearly independent. By the Witt extension theorem, the $\G_n^{\ve}(F)$-action on the linearly independent subsets in $\VV$ is transitive, yielding a single orbit. We denote the representative of this orbit by $(e_1, \dots, e_{k-1}, e)$.

Define the stabilizer
\[
\R_{k}^{\ve} \coloneqq \{ h \in \G_{n}^{\ve} \mid h e_1 = e_1, \dots, h e_{k-1} = e_{k-1}, \, h e = e \}.
\]
The explicit matrix forms of $\R_{k}^{\ve}$ are given as follows:
\begin{flalign*}
&\text{(Case: $\ve=0$)} & \R_{k}^{\ve} &= 
\begin{cases}
\begin{pmatrix} \mathrm{Id}_{n-2} & * & 0 & * & * & * \\ & * & 0 & * & * & * \\ & 0 & 1 & 0 & 0 & 0 \\ & * & 0 & * & * & * \\ & * & 0 & * & * & * \\ & & & & & \mathrm{Id}_{n-2} \end{pmatrix} & \text{if } k=n-1, \\ 
\begin{pmatrix} \mathrm{Id}_{n-1} & 0 & * & * \\ & 1 & 0 & 0 \\ & & * & * \\ & & & \mathrm{Id}_{n-1} \end{pmatrix} & \text{if } k=n.
\end{cases} &
\end{flalign*}
\begin{flalign*}
&\text{(Case: $\ve=1$)} & \R_{k}^{\ve} &= 
\begin{cases}
\begin{pmatrix} \mathrm{Id}_{n-1} & * & 0 & * & * \\ & * & 0 & * & * \\ & 0 & 1 & 0 & 0 \\ & * & 0 & * & * \\ & & & & \mathrm{Id}_{n-1} \end{pmatrix} & \text{if } k=n, \\ 
\begin{pmatrix} \mathrm{Id}_{n} & 0 & * \\ & 1 & 0 \\ & &  \mathrm{Id}_{n} \end{pmatrix} & \text{if } k=n+1.
\end{cases} &
\end{flalign*}

Therefore, we can rewrite the Whittaker period as:
\begin{align*}
W_{\psi}^{\lambda_{\ve}}(f) &= \int_{[Z_{k}]} \mu_{k}'^{-1}(v'(z)) \int_{[\G_{n}^{\ve}]} \left( \sum_{\mathbf{x} \in \VV} \big(\omega_{\psi,V,W_k}(g, v'(z))\phi\big)(\mathbf{x}) \right) \wt{\varphi}(g) \mathrm{sgn}_{\TT}(g) \, dg \, dz \\
&= \int_{[Z_{k}]} \mu_{k}'^{-1}(v'(z)) \int_{[\G_{n}^{\ve}]} \int_{\R_{k}^{\ve}(F) \bs \G_n^{\ve}(F)} \wt{\varphi}(g) \mathrm{sgn}_{\TT}(g) \\
&\quad \quad \times \big(\omega_{\psi,V,W_k}(g_1 g, v'(z))\phi\big)(e_1, \dots, e_{k-1}, e) \, dg_1 \, dg \, dz \\
&= \int_{[Z_{k}]} \mu_{k}'^{-1}(v'(z)) \int_{\R_{k}^{\ve}(F) \bs \G_{n}^{\ve}(\A)} \wt{\varphi}(g) \mathrm{sgn}_{\TT}(g) \big(\omega_{\psi,V,W_k}(g, v'(z))\phi\big)(e_1, \dots, e_{k-1}, e) \, dg \, dz.
\end{align*}

Every element of $Z_k$ can be factored as $\begin{pmatrix} z & \\ & 1 \end{pmatrix} \begin{pmatrix} \mathrm{Id}_{k-1} & a \\ & 1 \end{pmatrix}$ for $z \in Z_{k-1}$ and $a \in \M_{k-1 \times 1}$. We define a map $v : Z_k \to \G_n^{\ve}$ by
\[
v\left( \begin{pmatrix} z & \\ & 1 \end{pmatrix} \begin{pmatrix} \mathrm{Id}_{k-1} & a \\ & 1 \end{pmatrix} \right) \coloneqq v\left( \begin{pmatrix} z & \\ & 1 \end{pmatrix} \right) v\left( \begin{pmatrix} \mathrm{Id}_{k-1} & a \\ & 1 \end{pmatrix} \right),
\]
where
\[
v\left( \begin{pmatrix} z & \\ & 1 \end{pmatrix} \right) = 
\begin{cases} 
\begin{pmatrix} z & & \\ & \mathrm{Id}_{4-\ve} & \\ & & z^* \end{pmatrix} \in \G_n^{\ve} & \text{if } k=n+\ve-1, \\ 
\begin{pmatrix} z & & \\ & \mathrm{Id}_{2-\ve} & \\ & & z^* \end{pmatrix} \in \G_n^{\ve} & \text{if } k=n+\ve,
\end{cases}
\]
and \\

\noindent (Case: $\ve=0$)
{\small
\setlength{\arraycolsep}{3pt} 
\begin{equation*}
v\left( \begin{pmatrix} \mathrm{Id}_{k-1} & a \\ & 1 \end{pmatrix} \right) = 
\begin{cases} 
\begin{pmatrix} \mathrm{Id}_{n-2} & & a & & & \\ & 1 & & & & \\ & & 1 & & & a' \\ & & & 1 & & \\ & & & & 1 & \\ & & & & & \mathrm{Id}_{n-2} \end{pmatrix} \in \G_n^{\ve} & \text{if } k=n-1, \\[25pt] 
\begin{pmatrix} \mathrm{Id}_{n-1} & a & & \\ & 1 & & a' \\ & & 1 & \\ & & & \mathrm{Id}_{n-1} \end{pmatrix} \in \G_n^{\ve} & \text{if } k=n.
\end{cases}
\end{equation*}
}

\bigskip

\noindent (Case: $\ve=1$)
{\small
\setlength{\arraycolsep}{3pt}
\begin{equation*}
v\left( \begin{pmatrix} \mathrm{Id}_{k-1} & a \\ & 1 \end{pmatrix} \right) = 
\begin{cases} 
\begin{pmatrix} \mathrm{Id}_{n-1} & & a & & \\ & 1 & & & \\ & & 1 & & a' \\ & & & 1 & \\ & & & & \mathrm{Id}_{n-1} \end{pmatrix} \in \G_n^{\ve} & \text{if } k=n, \\[20pt] 
\begin{pmatrix} \mathrm{Id}_{n} & a & \\ & 1 & a' \\ & & \mathrm{Id}_{n} \end{pmatrix} \in \G_n^{\ve} & \text{if } k=n+1.
\end{cases}
\end{equation*}
}

Note that $\mu_{k}'(v'(z)) = \mu_{\ve}(v(z))$ for all $z \in Z_k$. From \eqref{w2}, it is straightforward to verify that for $t = \begin{pmatrix} z & \\ & 1 \end{pmatrix}$ or $t = \begin{pmatrix} \mathrm{Id}_{k-1} & a \\ & 1 \end{pmatrix} \in Z_k$, we have
\[
\big(\omega_{\psi,V,W_k}(g, v'(t))\phi\big)(e_1, \dots, e_{k-1}, e) = \big(\omega_{\psi,V,W_k}(v(t)^{-1}g, \1)\phi\big)(e_1, \dots, e_{k-1}, e).
\]

For $\wt{\varphi} \in \mc{A}(\G_n^{\ve})$ and $\TT \in \mf{T}$, we define the period integral over the stabilizer:
\[
\wt{\varphi}^{\R_{k}^{\ve},\TT}(g) \coloneqq \int_{[\R_{k}^{\ve}]} \wt{\varphi}(rg) \mathrm{sgn}_{\TT}(rg) \, dr, \quad g \in \G_n^{\ve}(\A).
\]

Since $\R_{k}^{\ve}$ is the stabilizer of $(e_1, \dots, e_{k-1}, e)$ in $\G_n^{\ve}$, we can rewrite the Whittaker period as follows:
\begin{align*}
W_{\psi}^{\lambda_{\ve}}(f) &= \int_{[Z_{k}]} \int_{\R_{k}^{\ve}(F) \backslash \G_{n}^{\ve}(\mathbb{A})} \mu_{k}'^{-1}(v'(z)) \widetilde{\varphi}(g) \mathrm{sgn}_{\TT}(g) \big(\omega_{\psi,V,W_k}(g, v'(z))\phi\big)(e_1, \dots, e_{k-1}, e) \, dg \, dz \\
&=  \int_{[Z_{k}]} \int_{\R_{k}^{\ve}(\mathbb{A}) \backslash \G_n^{\ve}(\mathbb{A})} \mu_{k}'^{-1}(v'(z)) \widetilde{\varphi}^{\R_{k}^{\ve},\TT}(g)  \big(\omega_{\psi,V,W_k}(g, v'(z))\phi\big)(e_1, \dots, e_{k-1}, e) \, dg \, dz \\
&= \int_{[Z_{k}]} \int_{\R_{k}^{\ve}(\mathbb{A}) \backslash \G_n^{\ve}(\mathbb{A})} \mu_{k}'^{-1}(v'(z)) \widetilde{\varphi}^{\R_{k}^{\ve},\TT}(g)  \big(\omega_{\psi,V,W_k}(v(z)^{-1}g, \mathbf{1})\phi\big)(e_1, \dots, e_{k-1}, e) \, dg \, dz \\
&= \int_{[Z_{k}]} \int_{\R_{k}^{\ve}(\mathbb{A}) \backslash \G_n^{\ve}(\mathbb{A})} \mu_{\ve}^{-1}(v(z)) \widetilde{\varphi}^{\R_{k}^{\ve},\TT}(v(z)g) \big(\omega_{\psi,V,W_k}(g, \mathbf{1})\phi\big)(e_1, \dots, e_{k-1}, e) \, dg \, dz \\
&= \int_{\R_{k}^{\ve}(\mathbb{A}) \backslash \G_n^{\ve}(\mathbb{A})} \big(\omega_{\psi,V,W_k}(g, \mathbf{1})\phi\big)(e_1, \dots, e_{k-1}, e)  \left( \int_{[Z_{k}]} \mu_{\ve}^{-1}(v(z)) \widetilde{\varphi}^{\R_{k}^{\ve},\TT}(v(z)g) \, dz \right) dg \\
&= \int_{\R_{k}^{\ve}(\mathbb{A}) \backslash \G_n^{\ve}(\mathbb{A})} \big(\omega_{\psi,V,W_k}(g, \mathbf{1})\phi\big)(e_1, \dots, e_{k-1}, e) \cdot \mathcal{B}_{k,\psi}^{\ve,\TT}(\mathfrak{R}(g)\widetilde{\varphi}) \, dg,
\end{align*}
where $\mf{R}(g)\wt{\varphi}$ denotes the right translation of $\wt{\varphi}$ by $g$.

Consequently, the non-vanishing of the Whittaker period $W_{\psi}^{\lambda_{\ve}}(f) \neq 0$ implies that the Bessel period $\mc{B}_{k,\psi}^{\ve,\TT}$ is non-trivial on $\wt{\pi}$. Conversely, if $\mc{B}_{k,\psi}^{\ve,\TT}$ is nonzero on $\wt{\pi}$ for some $\TT \in \mf{T}$, there exists an automorphic form $\wt{\varphi} \in \wt{\pi}$ such that $\mc{B}_{k,\psi}^{\ve,\TT}(\wt{\varphi}) \neq 0$. Conversely, if $\mc{B}_{k,\psi}^{\ve,\mathbb{T}}$ is nonzero on $\wt{\pi}$ for some $\mathbb{T}\in\mathfrak{T}$, there exists an automorphic $\wt{\varphi}'\in\wt{\pi}$ such that $\mc{B}_{k,\psi}^{\ve,\mathbb{T}}(\wt{\varphi}')\ne0$. Let $\mathbf{e}_0 = (e_1, \dots, e_{k-1}, e) \in (V \otimes Y_k^*)(F)$. Choose a finite set $S$ containing all archimedean places such that for $v \notin S$, $\mathbf{e}_{0,v} \in (V \otimes Y_k^*)(\mc{O}_v)$. 

For each $v \notin S$, we simply set the local Schwartz function $\phi_v$ to be the characteristic function of the standard $\mc{O}_v$-lattice $(V \otimes Y_k^*)(\mc{O}_v)$. Since $\mathbf{e}_{0,v}$ lies in this lattice, the evaluation of $\phi_v$ at this point is exactly $1$, ensuring that the infinite product is well-defined and non-vanishing.

For the places $v \in S$, we choose $\phi_v$ to be highly concentrated around $\mathbf{e}_{0,v}$. Specifically, for non-archimedean $v \in S$, we take it to be the normalized characteristic function of an arbitrarily small compact open neighborhood of $\mathbf{e}_{0,v}$. For archimedean $v \in S$, we choose it to be a non-negative smooth bump function tightly supported within an infinitesimal neighborhood of $\mathbf{e}_{0,v}$.

We then define the global Schwartz function as $\phi = \bigotimes_v \phi_v \in \mc{S}((V \otimes Y_k^*)(\A))$. By shrinking the supports of the local functions strictly at the places $v \in S$, the sequence of global Schwartz functions forms an approximate identity converging to the Dirac delta distribution supported at $\mathbf{e}_0$. Recall that $R_k^{\ve}$ is precisely the stabilizer of $\mathbf{e}_0$. Thus, in the limit, integrating over the quotient $R_k^{\ve}(\A) \bs G_n^{\ve}(\A)$ against this Dirac delta distribution localizes the integration variable exactly to the identity coset. 

This forces the entire global integral to collapse to a nonzero multiple of the inner period $\mc{B}_{k,\psi}^{\ve,\mathbb{T}}(\wt{\varphi}') \ne 0$. Therefore, we can guarantee the existence of an appropriate $\phi \in \omega_{\psi,V,W_k}$ such that $W_{\psi}^{\lambda_{\ve}}(f) \ne 0$ for $f = \theta_{\psi,V,W_k}^{\mathbb{T}}(\phi, \wt{\varphi}')$. This completes the proof.
\end{proof}
From Theorem \ref{a1}, we deduce the following two corollaries.

\begin{cor}\label{a2} 
Let $n \ge 2-\ve$, and let $\wt{\pi}$ (resp. $\pi$) be an irreducible cuspidal representation of $\G_n^{\ve}(\A)$ (resp. $\H_n^{\ve}(\A)$). Assume that $\wt{\pi}_v$ (resp. $\pi_v$) is $\mu_{\ve,v}^{\pm}$-generic (resp. $\mu_{\ve,v}$-generic) for some non-archimedean place $v$. Then for any $\TT \in \mf{T}$, the Bessel period $\mc{B}_{n-1+\ve,\psi}^{\ve,\TT}$ is non-trivial on $\wt{\pi}$ (resp. $\pi$) if and only if the global theta lift $\Theta_{\psi,V,W_{n-1+\ve}}^{\TT}(\wt{\pi})$ (resp. $\Theta_{\psi,V,W_{n-1+\ve}}(\pi)$) is nonzero, cuspidal, and $(\psi,\lambda_{\ve})$-generic.
\end{cor}

\begin{proof}
Since the proofs for the $\G_n^{\ve}$ and $\H_n^{\ve}$ cases are largely identical, we present the argument only for the $\G_n^{\ve}$ case. The $(\Leftarrow)$ direction is immediate from Theorem \ref{a1}, so it suffices to establish the $(\Rightarrow)$ direction. 

Suppose that the Bessel period $\mc{B}_{n-1+\ve,\psi}^{\ve,\TT}$ is non-trivial on $\wt{\pi}$. By Theorem \ref{a1}, the theta lift $\Theta_{\psi,V,W_{n-1+\ve}}^{\TT}(\wt{\pi})$ is nonzero and $(\psi,\lambda_{\ve})$-generic. It remains to show that it is cuspidal. 

Suppose, for the sake of contradiction, that $\Theta_{\psi,V,W_{n-1+\ve}}^{\TT}(\wt{\pi})$ is not cuspidal. By Rallis' tower property, there exists an index $k_0 < n-1+\ve$ such that the lift $\Theta_{\psi,V,W_{k_0}}^{\TT}(\wt{\pi}) = \Theta_{\psi,V,W_{k_0}}(\wt{\pi} \otimes \mathrm{sgn}_{\TT})$ is nonzero and cuspidal. Let $\sigma$ be an irreducible summand of $\Theta_{\psi,V,W_{k_0}}(\wt{\pi} \otimes \mathrm{sgn}_{\TT})$. By our hypothesis, the local component $(\wt{\pi} \otimes \mathrm{sgn}_{\TT})_v$ is either $\mu_{\ve,v}^{+}$-generic or $\mu_{\ve,v}^{-}$-generic, and $\sigma_v$ is the local theta lift of $(\wt{\pi} \otimes \mathrm{sgn}_{\TT})_v$ to $\wt{\J}_{k_0}(F_v)$. However, since $k_0 < n-1+\ve$, this contradicts Proposition \ref{c1}.
\end{proof}

\begin{cor}\label{a3} 
Let $\wt{\pi}$ (resp. $\pi$) be an irreducible cuspidal $\mu_{\ve}^{\TT}$-generic automorphic representation of $\G_n^{\ve}(\A)$ (resp. $\H_n^{\ve}(\A)$) for some $\TT \in \mf{T}$. Then $\Theta_{\psi,V,W_{n+\ve}}^{\TT}(\wt{\pi})$ (resp. $\Theta_{\psi,V,W_{n+\ve}}(\pi)$) is nonzero and $(\psi,\lambda_{\ve})$-generic. Furthermore, if $\Theta_{\psi,V,W_{n-1+\ve}}^{\TT}(\wt{\pi})$ (resp. $\Theta_{\psi,V,W_{n-1+\ve}}(\pi)$) vanishes, then $\Theta_{\psi,V,W_{n+\ve}}^{\TT}(\wt{\pi})$ (resp. $\Theta_{\psi,V,W_{n+\ve}}(\pi)$) is cuspidal.
\end{cor}

\begin{proof}
The first statement is an immediate consequence of Theorem \ref{a1}. The second statement follows directly from Rallis' tower property and Proposition \ref{c1}.
\end{proof}


\bb

\section{The relationship between $L(s,\pi)$ and the special Bessel period $\mc{B}_{n-1+\ve,\psi}^{\ve}(\pi)$} 
\label{sec:5}

In this section, we prove the following two main theorems.

\begin{thm}\label{t1} 
Let $n \ge 2-\ve$ and let $\wt{\pi}$ be an irreducible $\mu_{\ve}$-generic cuspidal representation of $\G_n^{\ve}(\A)$. Let $S$ be a finite set of places of $F$ containing all archimedean places, such that for $v \notin S$, both $\wt{\pi}_v$ and $\psi_{v}$ are unramified. Then the following conditions are equivalent:
\begin{enumerate}
    \item[$(i)$] \hspace{1mm}
    $\begin{cases} 
    L(s,\wt{\pi}) \text{ has a pole at } s=1 & \text{if } \ve=0, \\ 
    L(s,\wt{\pi}) \text{ is holomorphic and nonzero at } s=1/2 & \text{if } \ve=1. 
    \end{cases}$
    \item[$(ii)$] \hspace{1mm}
    $\begin{cases} 
    L^S(s,\wt{\pi}) \text{ has a pole at } s=1 & \text{if } \ve=0, \\ 
    L^S(s,\wt{\pi}) \text{ is holomorphic and nonzero at } s=1/2 & \text{if } \ve=1. 
    \end{cases}$
    \item[$(iii)$] $\mc{B}_{n-1+\ve,\psi}^{\ve}$ is non-trivial on $\wt{\pi} \otimes \mathrm{sgn}_{\TT}$ for some $\TT \in \mf{T}$.
    \item[$(iv)$] $\wt{\pi} \otimes \mathrm{sgn}_{\TT}$ has a nonzero and $(\psi,\lambda_{\ve})$-generic global theta lift to $\wt{\J}_{n-1+\ve}(\A)$ for some $\TT \in \mf{T}$.
    \item[$(v)$] $\wt{\pi} \otimes \mathrm{sgn}_{\TT}$ has a nonzero global theta lift to $\wt{\J}_{n-1+\ve}(\A)$ for some $\TT \in \mf{T}$.
\end{enumerate}
\end{thm}

\begin{thm}\label{t2}
Let $n \ge 2-\ve$ and let $\pi$ be an irreducible $\mu_{\ve}$-generic cuspidal representation of $\H_n^{\ve}(\A)$. Let $S$ be a finite set of places of $F$ containing all archimedean places, such that for $v \notin S$, both $\pi_v$ and $\psi_v$ are unramified. Then the following conditions are equivalent:
\begin{enumerate}
    \item[$(i)$] \hspace{1mm}
    $\begin{cases} 
    L(s,\pi) \text{ has a pole at } s=1 & \text{if } \ve=0, \\ 
    L(s,\pi) \text{ is holomorphic and nonzero at } s=1/2 & \text{if } \ve=1. 
    \end{cases}$
    \item[$(ii)$] \hspace{1mm}
    $\begin{cases} 
    L^S(s,\pi) \text{ has a pole at } s=1 & \text{if } \ve=0, \\ 
    L^S(s,\pi) \text{ is holomorphic and nonzero at } s=1/2 & \text{if } \ve=1. 
    \end{cases}$
    \item[$(iii)$] $\mc{B}_{n-1+\ve,\psi}^{\ve}$ is non-trivial on $\pi$. 
    \item[$(iv)$] $\pi$ has a nonzero and $(\psi,\lambda_{\ve})$-generic global theta lift to $\wt{\J}_{n-1+\ve}(\A)$.
    \item[$(v)$] $\pi$ has a nonzero global theta lift to $\wt{\J}_{n-1+\ve}(\A)$.
\end{enumerate}
\end{thm}

\begin{rem}\label{str}
When $\ve=0$ and $\H_n^{0}$ is split of type $(1,1)$, the equivalence of conditions $(ii)$, $(iii)$, and $(iv)$ in Theorem \ref{t2} was established in \cite[Theorem~3.4]{GRS97}. For strongly generic representations $\pi$, the equivalence of conditions $(ii)$--$(v)$ is discussed following \cite[Proposition~3.5]{GRS97} for the $\ve=0$ case, and in \cite[Main Theorem]{Fu95} for the $\ve=1$ case. Consequently, Theorem \ref{t2} extends these results to the broader setting of \textit{generic} representations of \textit{quasi-split} $\H_n^{\ve}$ of arbitrary type $(d,c) \in (F^{\times}/ F^{\times 2})^2$.
\end{rem}

\begin{rem}\label{kap}
To prove Theorem \ref{t1}, we utilize the Rankin--Selberg theory for $\SO_{2n} \times \GL_1$ and $\SO_{2n+1} \times \GL_1$. For the case $\ve = 1$ (the odd orthogonal case), most of the required results are available in the literature. However, for $\ve = 0$ (the even orthogonal case), the foundational results established in \cite{Kap15} naturally focus on specific types of $\H_n^0$, namely:
$$
(d,c) = 
\begin{cases} 
(d, -1/8) & \text{if } d \in F^{\times 2}, \\ 
(d, 1/2d) & \text{if } d \notin F^{\times 2}. 
\end{cases}
$$
Moreover, these specific cases were addressed independently. Recently, Adrian, Henniart, Kaplan, and Oi \cite[Section~5.2]{AHKO25} demonstrated that, for split $\SO_{2n}$, the local gamma factor for $\SO_{2n} \times \GL_1$ originally defined in \cite{Kap15} can be extended to an arbitrary Whittaker datum—rather than being restricted to a specific one—while preserving the relevant properties of the gamma functions. Consequently, in the split setting, one might expect that the Rankin--Selberg integral methods established in \cite{Kap15} can be adapted to arbitrary Whittaker data in a relatively straightforward manner. However, when $\SO_{2n}$ is quasi-split but non-split, such an extension presents distinct structural challenges and does not follow in a straightforward manner from existing techniques.

Given our objective to encompass an arbitrary type $(d, c)$ for $\H_n^0$ (cf. Remark \ref{type}), we take Kaplan's robust framework as our starting point to generalize the Rankin--Selberg integral theory for $\SO_{2n} \times \GL_1$ to all types of $\SO_{2n}$ in a uniform manner. While we build upon the original methods in \cite{Kap15}, achieving this uniform generalization requires developing refined technical modifications tailored to the arbitrary type setting. In addition to providing self-contained expositions for several key steps that were only briefly outlined in \cite{Kap15}, we establish new propositions required for our broader context. Most notably, we provide a complete and rigorous proof for Proposition \ref{p3}, including a full treatment of the archimedean cases. This result is essential for our purposes, and its formal verification extends the literature significantly, even for the specific types of $\H_n^0$ previously studied. Although we omit certain standard technical details that follow \textit{mutatis mutandis} from \cite{Kap15}, our exposition ensures a rigorous, self-standing, and comprehensive foundation for the general setting.
\end{rem}

We first define the global zeta integral $\I$ which represents the partial $L$-function $L^S(s,\pi)$ for an irreducible cuspidal $\mu_{\ve}$-generic automorphic representation of $\H_n^{\ve}(\A)$. For convenience, we will henceforth represent the group $\G_n^{\ve}$ and its subgroups using the following ordered bases of $V_n^{\ve}$:
\[
\begin{cases} 
\{e_1, \dots, e_{n-1}, e, e', e_{n-1}^*, \dots, e_{1}^*\} & \text{if } \ve=0, \\ 
\{e_1, \dots, e_{n-1}, e_n, e, e_n^*, e_{n-1}^*, \dots, e_{1}^*\} & \text{if } \ve=1.
\end{cases}
\]

We define the subgroup
\[
\G^{\ve} \coloneqq \G_{n,n-1+\ve}^{\ve} = 
\begin{cases}
\G_{n,n-1}^{0} \simeq \O(3) & \text{if } \ve=0, \\ 
\G_{n,n}^{1} \simeq \O(2) & \text{if } \ve=1,
\end{cases}
\]
and specify its embedding into $\G_n^{\ve}$ as follows:
\begin{align*}
&\begin{pmatrix} a & b & c \\ d & e & f \\ g & h & k \end{pmatrix} \longmapsto 
\begin{pmatrix} 
\mathrm{Id}_{n-2} & & \\ 
& \begin{pmatrix} a & & b & c \\ & 1 & & \\ d & & e & f \\ g & & h & k \end{pmatrix} & \\ 
& & \mathrm{Id}_{n-2} 
\end{pmatrix} 
& \text{if } \ve=0, \\
&\begin{pmatrix} a & b \\ c & d \end{pmatrix} \longmapsto 
\begin{pmatrix} 
\mathrm{Id}_{n-1} & & \\ 
& \begin{pmatrix} a & & b \\ & 1 & \\ c & & d \end{pmatrix} & \\ 
& & \mathrm{Id}_{n-1} 
\end{pmatrix} 
& \text{if } \ve=1. 
\end{align*}Through these embeddings, we naturally regard $\G^{\ve}$ as a subgroup of $\G_n^{\ve}$. Accordingly, its special orthogonal subgroup $\H^{\ve}$ is naturally regarded as a subgroup of $\H_n^{\ve}$.

Let $\B^{\ve} = \T^{\ve}\N^{\ve}$ be the standard Borel subgroup of $\H^{\ve}$, where $\T^{\ve}$ is the split maximal torus and $\N^{\ve}$ is the unipotent radical. (Note that when $\ve=1$, we simply have $\H^{1} = \B^{1} = \T^{1}$ and $\N^{1}$ is the trivial.) Since $\T^{\ve} \simeq \mathbb{G}_m$, we write $t = t(a) = \mathrm{diag}(a, 1,  a^{-1}) \in \T^{0}(F)$ (resp. $t = t(a) = \mathrm{diag}(a,a^{-1}) \in \T^{1}(F)$) for $a \in F^{\times}$ in the case $\ve=0$ (resp. $\ve=1$).

Let $|\cdot|^{s}$ be the character of $\B^{\ve}$ defined by
\[
|\cdot|^{s} : 
\begin{cases}  
\begin{pmatrix} a & x & y \\ & 1 & x' \\ & & a^{-1} \end{pmatrix} \longmapsto |a|^{s} & \text{if } \ve=0, \\ 
\begin{pmatrix} a & \\ & a^{-1} \end{pmatrix} \longmapsto |a|^{s} & \text{if } \ve=1,
\end{cases}
\]
and let $I(s) = \Ind_{\B^{\ve}}^{\H^{\ve}}(|\cdot|^{s-1/2})$ denote the normalized induced representation of $\H^{\ve}$. For each place $v$ of $F$, let $\K^{\ve,v} \subset \K_{\ve,v}$ be the standard maximal compact subgroup of $\H^{\ve}(F_v)$, chosen such that $\K^{\ve,v}$ is special if $v$ is non-archimedean. We set $\K^{\ve} \coloneqq \prod_v \K^{\ve,v} \subset \H^{\ve}(\A)$.

The Eisenstein series attached to a smooth holomorphic section $f_s \in I(s)(\A) = \Ind_{\B^{\ve}(\A)}^{\H^{\ve}(\A)}(|\cdot|^{s-1/2})$ is defined by
\[
E(f_s, g) \coloneqq \sum_{\gamma \in \B^{\ve}(F) \bs \H^{\ve}(F)} f_s(\gamma g), \quad \text{for } g \in \H^{\ve}(\A).
\]

For a cusp form $\varphi \in \mc{A}_{cusp}(\H_n^{\ve})$, write
\[
\varphi^{\psi}(\varphi)(h) \coloneqq \int_{[\U_{n,n-1+\ve}^{\ve}]} \varphi(uh) \mu_{n-1+\ve,\ve}^{-1}(u) \, du, \quad h \in \H_n^{\ve}(\A).
\]

Let $\pi$ be an irreducible $\mu_{\ve}$-generic cuspidal automorphic representation of $\H_n^{\ve}(\A)$. The global zeta integral $\I(\cdot, f_s)$ that will provide the $L$-function of $\pi$ is defined as follows:  
\begin{equation}\label{bsi}
\I(\varphi, f_s) \coloneqq \int_{\H^{\ve}(F) \bs \H^{\ve}(\A)} \varphi^{\psi}(\varphi)(h)\cdot  E(f_s, h) \, dh, \quad \varphi \in \pi.
\end{equation}
The rapid decay of the cusp form $\varphi$ and the moderate growth of the Eisenstein series ensure that the integral $\I(\varphi, f_s)$
is absolutely convergent for all $s \in \CC$, except at the poles of $E(f_s, \cdot)$.

Define the Weyl element $w_{\ve} \in \mathrm{M}_{(2n+\ve) \times (2n+\ve)}(F)$ by
\[
w_{\ve} = \begin{pmatrix}
 & 1 & & & \\
\mathrm{Id}_{n-2+\ve} & & & & \\
 & & \mathrm{Id}_{2-\ve} & & \\
 & & & & \mathrm{Id}_{n-2+\ve} \\
 & & & 1 & 
\end{pmatrix},
\]
and for $x \in \mathbb{G}_a^{n-2+\ve}$, define
\[
\U_{\text{op}}^{\ve}(x) = \begin{pmatrix} 
1 & & & & \\ 
{}^t x & \mathrm{Id}_{n-2+\ve} & & & \\ 
& & \mathrm{Id}_{2-\ve} & & \\ 
& & & \mathrm{Id}_{n-2+\ve} & \\ 
& & & x' & 1 
\end{pmatrix} \in \H_n^{\ve}.
\]
Let $\U_{\text{op}}^{\ve} \coloneqq \{ \U_{\text{op}}^{\ve}(x) \mid x \in \mathbb{G}_a^{n-2+\ve} \}$.  For $W \in \mc{W}_{\psi}^{\ve}(\pi)$, where $\mc{W}_{\psi}^{\ve}(\pi)$ is the Whittaker model of $\pi$, we define the corresponding zeta integral $Z(W, f_s)$ by
\begin{equation}\label{zeta}
Z(W, f_s) \coloneqq \int_{\N^{\ve}(\A) \bs \H^{\ve}(\A)} \int_{\U_{\text{op}}^{\ve}(\A)} W(r w_{\ve} h) f_s(h) \, dr \, dh.
\end{equation}

In Propositions \ref{p1}--\ref{p3} below, we investigate several key properties of the zeta integral $Z(W, f_s)$. The first proposition relates the global integral $\I(\varphi, f_s)$ to the zeta integral $Z(W, f_s)$.
\begin{prop}[cf. {\cite[Prop.~3.3]{Kap15}}]\label{p1} 
Let $n \ge 2-\ve$ and let $\pi$ be an irreducible $\mu_{\ve}$-generic cuspidal representation of $\H_n^{\ve}(\A)$. For $\varphi \in \pi$, $s \in \CC$ with $\mathrm{Re}(s) \gg 0$, and a smooth holomorphic section $f_s \in I(s)(\A)$, we have the identity
\[
\I(\varphi, f_s) = Z(W_{\psi}^{\ve}(\varphi), f_s).
\]
\end{prop}

For the case $\ve=1$, this result has been established in \cite{No75, Gin90} (see also \cite[Section~5]{Fu95}). Consequently, we focus on proving the case $\ve=0$.

\begin{proof} 
Expanding the Eisenstein series, we obtain
\begin{align*}
\I(\varphi, f_s) &= \int_{[\H^{0}]} \mc{B}_{n-1,\psi}^{0}(\varphi)(h) \sum_{\gamma \in \B^{0}(F)\bs \H^{0}(F)} f_s(\gamma h) \, dh \\ 
&= \int_{[\H^{0}]} \sum_{\gamma \in \B^{0}(F)\bs \H^{0}(F)} \mc{B}_{n-1,\psi}^{0}(\varphi)(h) f_s(\gamma h) \, dh \\ 
&= \int_{[\H^{0}]} \sum_{\gamma \in \B^{0}(F)\bs \H^{0}(F)} \mc{B}_{n-1,\psi}^{0}(\varphi)(\gamma h) f_s(\gamma h) \, dh.
\end{align*}
Here, the third equality follows from the fact that $\varphi$ is left $\G_n^{0}(F)$-invariant and $\mu_{n-1,0}(\gamma u \gamma^{-1}) = \mu_{n-1,0}(u)$ for all $\gamma \in \H^{0}(F)$ and $u \in \U_{n-1}^0(\A)$. Collapsing the integral over the quotient, we get
\begin{align}\label{in}
\nonumber \I(\varphi, f_s) &= \int_{\B^{0}(F)\bs \H^{0}(\A)} \mc{B}_{n-1,\psi}^{0}(\varphi)(h) f_s(h) \, dh \\
\nonumber &= \int_{\T^0(F)\N^0(\A)\bs \H^0(\A)} \int_{[\N^0]} \mc{B}_{n-1,\psi}^{0}(\varphi)(nh) f_s(nh) \, dn \, dh \\
&= \int_{\T^0(F)\N^0(\A)\bs \H^0(\A)} \left( \int_{[\N^0]} \mc{B}_{n-1,\psi}^{0}(\varphi)(nh) \, dn \right) f_s(h) \, dh.
\end{align}

We now analyze the inner integral. Recall that $Z_k$ denotes the standard maximal unipotent subgroup of $\GL_k$. Define the subgroup
\[
\U' \coloneqq \left\{ u = \begin{pmatrix} z & x & a \\ & \mathrm{Id}_2 & x' \\ & & z^* \end{pmatrix} \in \H_n^0 \;\middle|\; z \in Z_{n-1}, \ x \in \mathrm{M}_{(n-1) \times 2}, \ x_{n-1,1}=0 \right\},
\]
and assign it the character $\mu_{\U'}$ defined by
\[
\mu_{\U'}(u) \coloneqq \psi\left( \sum_{i=1}^{n-3} z_{i,i+1} + x_{n-2,1} \right).
\]
We also set
\[
\U_0 \coloneqq \left\{ u = \begin{pmatrix} \mathrm{Id}_{n-1} & & {}^t b & a \\ & 1 & & \\ & & 1 & b^* \\ & & & \mathrm{Id}_{n-1} \end{pmatrix} \in \H_n^{0} \;\middle|\; a \in \mathrm{M}_{(n-1) \times (n-1)}, \ b \in \mathbb{G}_a^{n-1} \right\}.
\]
For $h \in \H_n^0(\A)$, let $I'(\varphi, h)$ denote the inner integral from \eqref{in}:
\[
I'(\varphi, h) \coloneqq \int_{[\N^0]} \mc{B}_{n-1,\psi}^{0}(\varphi)(nh) \, dn.
\]
By decomposing $\N^0$ according to $\U'$, we can write
\begingroup
\setlength{\arraycolsep}{3pt} 
\begin{align*}
I'(\varphi, h) &= \int_{[\U']} \varphi(uh) \mu_{\U'}(u)^{-1} \, du \\
&= \int_{[Z_{n-2}]} \int_{[\mathbb{G}_a^{n-2}]} \int_{[\mathbb{G}_a^{n-2}]} \int_{[\U_0]} 
\varphi\Bigg( u 
\begin{pmatrix}
z & {}^ty & {}^tx & * & * & * \\
 & 1 & 0 & * & * & * \\
 & & 1 & 0 & 0 & x' \\
 & & & 1 & * & * \\
 & & & & 1 & y' \\
 & & & & & z^*
\end{pmatrix} h \Bigg) \\
&\qquad \times \psi\left(\sum_{i=1}^{n-3} z_{i,i+1} + x_{n-2}\right)^{-1} \, du \, dx \, dy \, dz.
\end{align*}

\begin{align*}
I'(\varphi, h) &= \int_{[Z_{n-2}]} \int_{[\mathbb{G}_a^{n-2}]} \int_{[\mathbb{G}_a^{n-2}]} \int_{[\U_0]} 
\varphi\Bigg( w_0 u w_0^{-1} \cdot w_0 
\begin{pmatrix}
z & {}^ty & {}^tx & * & * & * \\
 & 1 & 0 & * & * & * \\
 & & 1 & 0 & 0 & x' \\
 & & & 1 & * & * \\
 & & & & 1 & y' \\
 & & & & & z^*
\end{pmatrix} w_0^{-1} \cdot w_0 h \Bigg) \\
&\quad \times \psi\left(\sum_{i=1}^{n-3} z_{i,i+1} + x_{n-2}\right)^{-1} \, du \, dx \, dy \, dz \\
&= \int_{[Z_{n-2}]} \int_{[\mathbb{G}_a^{n-2}]} \int_{[\mathbb{G}_a^{n-2}]} \int_{[\U_0]} 
\varphi\Bigg( u 
\begin{pmatrix}
1 & 0 & 0 & * & * & * \\
{}^ty & z & {}^tx & * & * & * \\
 & & 1 & 0 & x' & 0 \\
 & & & 1 & * & * \\
 & & & & z^* & 0 \\
 & & & & y' & 1
\end{pmatrix} w_0 h \Bigg) \\
&\quad \times \psi\left(\sum_{i=1}^{n-3} z_{i,i+1} + x_{n-2}\right)^{-1} \, du \, dx \, dy \, dz,
\end{align*}
\endgroup
where the last equality follows by applying the change of variable $u \mapsto w_0^{-1} u w_0$.

Let $\P_n$ be the mirabolic subgroup of $\GL_n$, given by
\[
\P_n = \left\{ \begin{pmatrix} A & B \\ 0 & 1 \end{pmatrix} \in \GL_n \;\middle|\; A \in \GL_{n-1}, \ B \in  \mathbb{G}_a^{n-1} \right\}.
\]
We embed $\P_n$ into $\H_n^{0}$ via the map
\[
\begin{pmatrix} A & B \\ 0 & 1 \end{pmatrix} \longmapsto 
\begin{pmatrix} 
A & B & & \\ 
& 1 & & B' \\ 
& & 1 & \\ 
& & & A^* \end{pmatrix} \in \H_n^{0}.
\]
For any $h \in \H_n^0(\A)$, define the function $\wt{\varphi}_h$ on $\P_n(\A)$ by
\[
\wt{\varphi}_h(p) \coloneqq \int_{[\U_0]} \varphi(u p h) \, du.
\]
Because $\varphi$ is cuspidal, $\wt{\varphi}_h$ inherits this property and defines a cuspidal function on $\P_n(F) \bs \P_n(\A)$. Observe that
\[
\int_{[Z_n]} \wt{\varphi}_h(z) \mu_{n-1,0}(z)^{-1} \, dz = W_{\psi}^{0}(\varphi)(h).
\]
Applying the Fourier expansion of $\wt{\varphi}_h$ evaluated at the identity matrix $\mathrm{Id}_n$ (cf. \cite[(1.0.1)]{Ge-P87}), we have
\begin{equation}\label{Fo}
\wt{\varphi}_h(\mathrm{Id}_n) = \int_{[\U_0]} \varphi(uh) \, du = \sum_{\gamma \in Z_{n-1}(F) \bs \GL_{n-1}(F)} W_{\psi}^{0}(\varphi)\left( \begin{pmatrix} \gamma & & \\ & \mathrm{Id}_2 & \\ & & \gamma^* \end{pmatrix} h \right).
\end{equation}

Substituting \eqref{Fo} back into $I'(\varphi, h)$, we obtain
\begingroup
\setlength{\arraycolsep}{2.5pt} 
\begin{align*}
&I'(\varphi, h)\\ &= \sum_{\gamma \in Z_{n-1}(F) \bs \GL_{n-1}(F)}  \int_{[\mathbb{G}_a^{n-2}]} \int_{[Z_{n-2}]} \int_{[\mathbb{G}_a^{n-2}]} W_{\psi}^{0}(\varphi) \Bigg( \begin{pmatrix} \gamma & & \\ & \mathrm{Id}_2 & \\ & & \gamma^* \end{pmatrix} \\
&\quad \times \begin{pmatrix} 1 & 0 & 0 & * & * & * \\ {}^ty & z & {}^tx & * & * & * \\ & & 1 & 0 & x' & 0 \\ & & & 1 & * & * \\ & & & & z^* & 0 \\ & & & & y' & 1 \end{pmatrix} w_0 h \Bigg)  \psi\left(\sum_{i=1}^{n-3} z_{i,i+1} + x_{n-2}\right)^{-1} \, dx \, dz \, dy.
\end{align*}
\endgroup

Note that the quotient spaces $[\mathbb{G}_a^{n-2}]$ and $[Z_{n-2}]$ are compact, and the integral of a character over a compact group vanishes unless the character is trivial. To ensure the integral does not vanish over $x$, the last row of $\gamma$ must be $(0, \dots, 0, 1)$. Furthermore, since the summation is taken modulo $Z_{n-1}(F)$, we may assume without loss of generality that the last column of $\gamma$ is ${}^t(0, \dots, 0, 1)$. Proceeding inductively, we find that only matrices of the form
\[
\begin{pmatrix} c & 0 \\ a & \mathrm{Id}_{n-2} \end{pmatrix} \in Z_{n-1}(F) \bs \GL_{n-1}(F), \quad \text{for } c \in F^{\times}, \ a \in F^{n-2},
\]
contribute non-trivial summands to $I'(\varphi, h)$. 

Therefore, by unfolding the sum, we can rewrite $I'(\varphi, h)$ as
\begin{align*}
&I'(\varphi, h) \\ &= \sum_{c \in F^{\times}} \sum_{a \in F^{n-2}} \int_{[\mathbb{G}_a^{n-2}]} W_{\psi}^{0}(\varphi)\left( \begin{pmatrix} c & & \\ & \mathrm{Id}_{2n-2} & \\ & & c^{-1} \end{pmatrix} \U_{\text{op}}^{0}(a) \U_{\text{op}}^{0}(y) w_0 h \right) \, dy \\
&= \sum_{c \in F^{\times}} \int_{\mathbb{G}_a^{n-2}(\A)} W_{\psi}^{0}(\varphi)\left( \begin{pmatrix} c & & \\ & \mathrm{Id}_{2n-2} & \\ & & c^{-1} \end{pmatrix} \U_{\text{op}}^{0}(y) w_0 h \right) \, dy \\
&= \sum_{c \in F^{\times}} \int_{\mathbb{G}_a^{n-2}(\A)} W_{\psi}^{0}(\varphi)\left( \U_{\text{op}}^{0}(y) \begin{pmatrix} c & & \\ & \mathrm{Id}_{2n-2} & \\ & & c^{-1} \end{pmatrix} w_0 h \right) \, dy \\
&= \sum_{c \in F^{\times}} \int_{\mathbb{G}_a^{n-2}(\A)} W_{\psi}^{0}(\varphi)\left( \U_{\text{op}}^{0}(y) w_0 \begin{pmatrix} \mathrm{Id}_{n-2} & & & & \\ & c & & & \\ & & \mathrm{Id}_2 & & \\ & & & c^{-1} & \\ & & & & \mathrm{Id}_{n-2} \end{pmatrix} h \right) \, dy \\
&= \sum_{t \in \T^{0}(F)} \int_{\mathbb{G}_a^{n-2}(\A)} W_{\psi}^{0}(\varphi)\big( \U_{\text{op}}^{0}(y) w_0 t h \big) \, dy.
\end{align*}

Substituting this back into \eqref{in}, we finally conclude that
\[
\I(\varphi, f_s) = \int_{\N^{0}(\A) \bs \H^{0}(\A)} \int_{\mathbb{G}_a^{n-2}(\A)} W_{\psi}^{0}(\varphi)\big( \U_{\text{op}}^{0}(y) w_0 h \big) f_s(h) \, dy \, dh = Z(W_{\psi}^0(\varphi), f_s),
\]
as desired.
\end{proof}
Motivated by the definition of the global zeta integral $Z(\cdot, f_s)$, for each place $v$ of $F$, we define the local zeta integral $Z_v(\cdot, f_{s,v})$. Here, $f_{s,v}$ is a local section in the induced representation space $I(s)(F_v) = \Ind_{\B^{\ve}(F_v)}^{\H^{\ve}(F_v)}(|\cdot|^{s-1/2})$, and the integral is defined by
\[
Z_v(W_v, f_{s,v}) \coloneqq \int_{\N^{\ve}(F_v)\bs \H^{\ve}(F_v)} \int_{\U_{\text{op}}^{\ve}(F_v)} W_{v}(r w_{\ve} h) f_{s,v}(h) \, dr \, dh, \quad W_v \in \mc{W}_{\psi_v}^{\ve}(\pi_v).
\]

The following proposition provides the evaluation of the local zeta integral $Z_v$ for unramified data. 

\begin{prop}[cf. {\cite[Section~3.2]{Kap15}}]\label{p2} 
Let $n \ge 2-\ve$. Let $F_v$ be a non-archimedean local field with a residue field of order $q$. Suppose that $\pi \in \Irr(\H_n^{\ve}(F_v))$ is an unramified, $\mu_{\ve,v}$-generic representation. Let $W^0 \in \mc{W}_{\psi_v}^{\ve}(\pi)$ and $f_{s}^0 \in I(s)(F_v)$ be the normalized spherical vectors such that $W^0(\1)=1$ and $f_{s}^{0}(\1)=1$. Furthermore, assume that $W^0$ and $f_s^0$ are invariant under the right action of $\K_{\ve,v}$ and $\K^{\ve,v}$, respectively. Then, for $\mathrm{Re}(s) \gg 0$, we have
\[
Z_v(W^0, f_{s}^0) = 
\begin{cases}
L(s, \pi) \cdot \zeta_{v}(2s)^{-1} & \text{if } \ve=0, \\ 
L(s, \pi) & \text{if } \ve=1,
\end{cases}
\]
where $\zeta_{v}(s) = (1-q^{-s})^{-1}$.
\end{prop}

\begin{proof}
For the case $\ve=1$, this result was established in \cite[Section~12]{Sou93}. In the case $\ve=0$, the identity was proved in \cite[Section~3.2]{Kap15} for specific types of $\H_n^0$; however, the arguments therein naturally extend to an arbitrary type $(d,c)$. The key observation is that the Casselman--Shalika formula for Whittaker functions remains applicable to quasi-split classical groups with an arbitrary Whittaker datum. In view of its importance, we provide the details of this extension below.

Since we are working in a local setting, we suppress the subscript $v$ from our notation. For an algebraic group $G$, let $\delta_{G}$ denote its modulus character. Using the Iwasawa decomposition, we obtain
\begin{align*}
Z(W^0, f_s^0) &= \int_{\N^0(F) \bs \H^0(F)} \int_{\U_{\text{op}}^0(F)} W^0(r w_0 h) f_s^0(h) \, dr \, dh \\
&= \int_{\T^0(F)} \int_{\U_{\text{op}}^0(F)} W^0(r w_0 t) f_s^0(t) \delta_{\B^0}^{-1}(t) \, dr \, dt.
\end{align*}
Since $\T^{0} \simeq \mathbb{G}_m$, for each $a \in F^\times$, let $t(a) = \mathrm{diag}(a, 1, a^{-1}) \in \T^0(F)$. We define its image under conjugation by $w_0$ as $t'(a) \coloneqq w_0 t(a) w_0^{-1}$, which corresponds to the element $\mathrm{diag}(a, 1, \dots, 1, a^{-1}) \in \H_n^0(F)$. Here, the central identity block is of size $(2n-2) \times (2n-2)$, consistent with the embedding of $\T^0$ into $\H_n^0$.

Since $w_0 \in \K_{0}$ and $t'(a)$ normalizes $\U_{\text{op}}^0$, the change of variables $r \mapsto t'(a) r t'(a)^{-1}$ yields $dr \mapsto |a|^{2-n} dr$. Thus,
\[
Z(W^0, f_s^0) = \int_{F^{\times}} \left( \int_{\U_{\text{op}}^0(F)} W^0(t'(a) r) \, dr \right) f_s^0(t(a)) \delta_{\B^0}^{-1}(t(a)) |a|^{2-n} \, dt.
\]

To analyze the inner integral, we define the following elements in $\G_n^0(F)$. For $a \in F^\times$ and $x = (x_1, \dots, x_{n-2}) \in F^{n-2}$, let $m(a,x)$ denote the matrix
\[
m(a,x) = \begin{pmatrix} 
a & & & & \\ 
{}^t x & \mathrm{Id}_{n-2} & & & \\ 
& & \mathrm{Id}_2 & & \\ 
& & & \mathrm{Id}_{n-2} & \\ 
& & & x' & a^{-1} 
\end{pmatrix},
\]
where $x'$ is uniquely determined by the symmetric bilinear form defining the orthogonal group. We also define a specific unipotent element $u(c) \in \N^0(F)$ for $c \in F$:
\[
u(c) = \begin{pmatrix} 
1 & & c & & & \\ 
& \mathrm{Id}_{n-2} & & & & \\ 
& & 1 & & & c \\ 
& & & 1 & & \\ 
& & & & \mathrm{Id}_{n-2} & \\ 
& & & & & 1 
\end{pmatrix}.
\]
By the right $\K_0$-invariance of $W^0$, for any $c \in \mc{O}$, we have $W^0(m(a,x)) = W^0(m(a,x)u(c))$. A direct matrix multiplication yields
\[
m(a,x) u(c) = \begin{pmatrix} 
1 & & ac & & & \\ 
& \mathrm{Id}_{n-2} & {}^t x c & & & \\ 
& & 1 & &  &ac \\ 
& & & 1 & & \\ 
& & & & \mathrm{Id}_{n-2} & \\ 
& & & & & 1 
\end{pmatrix} m(a,x).
\]
Because $W^0$ is a Whittaker function with respect to the character $\psi$, this gives
\[
W^0(m(a,x)u(c)) = \psi(x_{n-2}c) \cdot W^0(m(a,x)).
\]
Therefore, $W^0(m(a,x))$ must vanish unless $x_{n-2} \in \mc{O}$. When $x_{n-2} \in \mc{O}$, the right $\K_0$-invariance implies that 
\[
W^0(m(a,x)) = W^0(m(a,x'')),
\]
where $x'' = (x_1, \dots, x_{n-3}, 0)$. By repeating this argument inductively with appropriate unipotent elements in $\N^0(\mc{O})$, we deduce that $W^0(m(a,x)) = 0$ unless $x \in \mc{O}^{n-2}$, in which case
\[
W^0(m(a,x)) = W^0(m(a,0)).
\]

Next, to restrict the support of $a$, consider the unipotent element $n(v)$ for $v = (v_1, \dots, v_{n-2}) \in F^{n-2}$:
\[
n(v) = \begin{pmatrix} 
1 & v & & & \\ 
& \mathrm{Id}_{n-2} & & & \\ 
& & \mathrm{Id}_2 & & \\ 
& & & \mathrm{Id}_{n-2} & {}^t v' \\ 
& & & & 1 
\end{pmatrix}.
\]
For any $v \in \mc{O}^{n-2}$, the right $\K_0$-invariance yields \[W^0(m(a,0)) = W^0(m(a,0)n(v)).\] We observe that
\[
m(a,0) n(v) = \begin{pmatrix} 
1 & av & & & \\ 
& \mathrm{Id}_{n-2} & & & \\ 
& & \mathrm{Id}_2 & & \\ 
& & & \mathrm{Id}_{n-2} & {}^t v' a \\ 
& & & & 1 
\end{pmatrix} m(a,0).
\]
Applying the Whittaker transformation property, we obtain \[W^0(m(a,0)n(v)) = \psi(av_1) \cdot W^0(m(a,0)).\] Since this identity must hold for all $v_1 \in \mc{O}$, we conclude that $W^0(m(a,0)) = 0$ unless $a \in \mc{O} \cap F^\times$. Consequently, the integral simplifies to
\[
Z(W^0, f_s^0) = \int_{\mc{O} \cap F^\times} W^0(t'(a)) \cdot |t(a)|^{s-1/2} \cdot \delta_{\B^0}^{-1/2}(t(a)) \cdot |a|^{2-n} \, da.
\]

Let the dual group $H_n'$ and the parameter $\ve'$ be defined as follows:
\[
H_n' = 
\begin{cases} 
\SO_{2n}(\CC) & \text{if } \H_n^0 \text{ is split}, \\ 
\Sp_{2n-2}(\CC) & \text{if } \H_n^0 \text{ is non-split}, 
\end{cases} 
\quad 
\ve' = 
\begin{cases} 
0 & \text{if } \H_n^0 \text{ is split}, \\ 
1 & \text{if } \H_n^0 \text{ is non-split}. 
\end{cases}
\]
For a dominant weight $J = (j_1, \dots, j_{n-\ve'}) \in \mathbb{Z}^{n-\ve'}$ with $j_1 \ge \dots \ge j_{n-\ve'}$, let $\rho_J$ denote the irreducible representation of $H_n'$ with highest weight $J$. For $a \in \mc{O} \cap F^\times$, let $|a| = q^{-\eta(a)}$ and set $J(a) = (\eta(a), 0, \dots, 0)$. Applying the Casselman--Shalika formula \cite[Theorem~5.4]{CS80}, we obtain
\[
W^0(t'(a)) = \delta_{\B_{n-1}^0}^{1/2}(t'(a)) \mathrm{Tr}(\rho_{J(a)}(c_\pi)).
\]
Since $\delta_{\B_{n-1}^0}(t'(a)) \delta_{\B^0}^{-1}(t(a)) = |a|^{2n-3}$, the integral becomes
\[
\int_{\mc{O} \cap F^\times} \mathrm{Tr}(\rho_{J(a)}(c_\pi)) |a|^s \, dt = \sum_{k=0}^\infty \mathrm{Tr}(\rho_{(k,0,\dots,0)}(c_\pi)) q^{-ks}.
\]
According to \cite[(3.22)]{Kap15}, this sum evaluates exactly to $L(s, \pi) \zeta(2s)^{-1}$, which completes the proof.
\end{proof}

Proposition \ref{p2} implies that the local zeta integral $Z_v$ converges absolutely for $\mathrm{Re}(s) \gg 0$ and admits a meromorphic continuation to $\CC$ for unramified data. It is a fundamental fact that these analytic properties hold generally, including in the ramified case.

\begin{prop}[cf. {\cite{Kap15, Sou93, Sou95}}]\label{p4}
Let $\pi_v \in \Irr(\H_n^{\ve}(F_v))$ be an irreducible $\mu_{\ve,v}$-generic representation. For any Whittaker function $W_v \in \mc{W}_{\psi_v}^{\ve}(\pi_v)$ and any smooth section $f_{s,v} \in I(s)(F_v)$, the integral $Z_v(W_v, f_{s,v})$ is absolutely convergent for $\mathrm{Re}(s) \gg 0$ and admits a meromorphic continuation to $\CC$.
\end{prop}

\begin{proof}
For the case $\ve=1$, we refer the reader to \cite{Sou93, Sou95}. For the case $\ve=0$, the arguments presented in \cite[Section~5.6 and Section~8]{Kap15} naturally extend to an arbitrary type $(d,c)$. Consequently, we omit the detailed proof here, as it follows \textit{mutatis mutandis} from the cited works.
\end{proof}
By Proposition \ref{p1} and Proposition \ref{p4}, for a pure tensor $\varphi = \bigotimes_v \varphi_v \in \pi$ and a holomorphic decomposable section $f_s = \bigotimes_v f_{s,v} \in I(s)(\A)$, we obtain the Euler product factorization
\begin{equation}\label{dec}
\I(\varphi, f_s) = \prod_v Z_v(W_{\psi_v}^{\ve}(\varphi_v), f_{s,v}) \quad \text{for } \mathrm{Re}(s) \gg 0.
\end{equation}

We also require the non-vanishing property of the local zeta integrals at all places. A local section $f_{s,v} \in I(s)(F_v)$ is called \textit{standard} if it is $\K_v^{\ve}$-finite and its restriction to $\K_v^{\ve}$ is independent of $s$.
\begin{prop}[cf. {\cite[Lemma~1.6]{GRS97}}]\label{p3}
Let $v$ be a place of $F$. Suppose that $\pi_v \in \Irr(\H_n^{\ve}(F_v))$ is an irreducible $\mu_{\ve,v}$-generic representation. Then for any $s_0 \in \CC$, there exist a Whittaker function $W_v \in \mc{W}_{\psi_v}^{\ve}(\pi_v)$ and a standard holomorphic section $f_{s,v} \in I(s)(F_v)$ such that $Z_v(W_v, f_{s,v})$ is holomorphic and nonzero at $s=s_0$.
\end{prop}

\begin{proof}
For non-archimedean places $v$, the result follows from \cite[Proposition~6.1]{Sou93} in the case $\ve=1$. For $\ve=0$, the result is obtained by extending \cite[Proposition~5.11]{Kap15} to an arbitrary type of $\H_n^0(F_v)$.

In contrast, for archimedean places $v$, the corresponding results do not appear to be available in the existing literature. Consequently, we provide a dedicated proof for the archimedean case here. Our underlying strategy adapts the arguments of \cite[Lemma~1.6]{GRS97}---which were originally established for the symplectic group $\Sp_{2n}$---to our quasi-split orthogonal group setting. As the arguments for both $\ve=0$ and $\ve=1$ are essentially identical, we restrict our presentation to the case $\ve=0$ for simplicity. 

For some $f_s \in I(s)(F_v)$ and $W \in \mc{W}_{\psi_v}^{0}(\pi_v)$, assume temporarily that $Z_v(W, f_s)$ is absolutely convergent at $s=s_0$. By the Iwasawa decomposition, we have 
\[
Z_v(W, f_{s_0}) = \int_{\K_v^{0}} \int_{F_v^{\times}} \int_{\U_{\text{op}}^{0}(F_v)} W(r w_0 t(a) k) f_{s_0}(t(a) k) |a|^{\alpha} \, dr \, da \, dk
\]
for some $\alpha \in \RR$. 

First, we choose a standard section $f_s \in I(s)(F_v)$. By the transformation property of the induced representation, $f_{s_0}(t(a) k)$ factors into a character of $a$ multiplied by $f_{s_0}(k)$. We can absorb this character into the existing power of $|a|$, replacing $\alpha$ with an appropriate $\beta \in \RR$. Therefore, such a section is completely determined by its restriction to $\K_v^{0} \cap \B^{0}(F_v) \bs \K_v^{0}$. By choosing the restriction $f_{s_0}|_{\K_v^{0}}$ to be a sequence of non-negative functions in $C^\infty(\K_v^0)$ forming an approximate identity that converges, in the sense of distributions, to the Dirac measure at the identity $\1 \in \K_v^{0}$, the integral over $\K_v^{0}$ converges to the value of the integrand evaluated at $k=\1$:
\[
\int_{F_v^{\times}} \int_{\U_{\text{op}}^0(F_v)} W(r w_0 t(a)) |a|^{s_0+\beta} \, dr \, da.
\]
Therefore, it suffices to find a Whittaker function $W \in \mc{W}_{\psi_v}^{0}(\pi_v)$ that makes this limiting double integral strictly nonzero. We may assume that $W(g) = W'(g w_0^{-1})$ for some $W' \in \mc{W}_{\psi_v}^{0}(\pi_v)$, so our goal is reduced to finding $W'$ such that
\[
\int_{F_v^{\times}} \int_{\U_{\text{op}}^0(F_v)} W'(r w_0 t(a) w_0^{-1}) |a|^{s_0+\beta} \, dr \, da \neq 0.
\]

For any $m \in \GL_{n-1}(F_v)$, we define its natural embedding into $\H_n^0(F_v)$ by $\widehat{m} \coloneqq \mathrm{diag}(m, \mathrm{Id}_2, m^*)$, where $m^* \in \GL_{n-1}(F_v)$ is uniquely determined by the symmetric bilinear form. For $1 \le i \le n-1$, let $e_i(x) \in \GL_{n-1}(F_v)$ denote the elementary matrix with $x$ in the $(1,i)$-entry and $1$ on the diagonal. We denote its embedding in $\H_n^0(F_v)$ simply by $\widehat{e_i(x)}$.

For each $x \in F_v$, let $u(x) \in \U_{n-1}^0(F_v)$ be the specific unipotent element defined as
\[
u(x) = \begin{pmatrix} 1 & & x & & & \\ & \mathrm{Id}_{n-2} & & & & \\ & & 1 & &  & x\\ & & & 1 & & \\ & & & & \mathrm{Id}_{n-2} & \\ & & & & & 1 \end{pmatrix},
\]
where the central block is of size $2 \times 2$. This element acts as the exact geometric analogue of the embedded elementary matrix $\widehat{e_n(x)}$ from the split case. 

By the Dixmier--Malliavin theorem \cite{DM78}, any element in $\mc{W}_{\psi_v}^{0}(\pi_v)$ can be expressed as a finite linear combination of functions of the form
\[
\int_{F_v} W_1(g \cdot u(x)) \phi_1(x) \, dx,
\]
where $W_1 \in \mc{W}_{\psi_v}^{0}(\pi_v)$ and $\phi_1 \in \mc{S}(F_v)$. Therefore, we may assume that $W'$ is of this form.

Let $t'(a) = w_0 t(a) w_0^{-1}$. Then, by a direct matrix computation and utilizing the transformation property of the Whittaker function, we obtain
\begin{align*}
&\int_{F_v^{\times}} \int_{\U_{\text{op}}^0(F_v)} W'(r t'(a)) |a|^{s_0+\beta} \, dr \, da \\ 
&= \int_{F_v^{\times}} \int_{F_v^{n-2}} \int_{F_v} W_1(\U_{\text{op}}^0({}^t y) t'(a) u(x)) \phi_1(x) |a|^{s_0+\beta} \, dx \, dy \, da \\
&= \int_{F_v^{\times}} \int_{F_v^{n-2}} \int_{F_v} W_1(\U_{\text{op}}^0({}^t y) t'(a)) \psi(a y_{n-2} x) \phi_1(x) |a|^{s_0+\beta} \, dx \, dy \, da \\
&= \int_{F_v^{\times}} \int_{F_v^{n-2}} W_1(\U_{\text{op}}^0({}^t y) t'(a)) \widehat{\phi}_1(a y_{n-2}) |a|^{s_0+\beta} \, dy \, da,
\end{align*}
where $\widehat{\phi}_1$ denotes the Fourier transform of $\phi_1$ with respect to $\psi$.

For each $1 \le j \le n-2$, define the subgroup
\[
\U_{\text{op}}^j(F_v) \coloneqq \left\{ \U_{\text{op}}^0({}^t y) \;\middle|\; y = (y_1, \dots, y_{n-2}) \in F_v^{n-2}, \ y_j = \dots = y_{n-2} = 0 \right\}.
\]

To decouple the variable $a$ from the support of $\widehat{\phi}_1$, we apply the change of variables $u = a y_{n-2}$. Then $dy_{n-2} = |a|^{-1} du$, and the above integral becomes
\[
\int_{F_v^{\times}} \int_{F_v^{n-3}} \int_{F_v} W_1(\U_{\text{op}}^0({}^t y', a^{-1} u) t'(a)) \widehat{\phi}_1(u) |a|^{s_0+\beta-1} \, du \, dy' \, da,
\]
where $y' = (y_1, \dots, y_{n-3})$. By choosing a sequence of Schwartz functions $\phi_1 \in \mc{S}(F_v)$ such that their Fourier transforms $\widehat{\phi}_1$ form an approximate identity converging, in the sense of distributions, to the Dirac measure at $0 \in F_v$, the integral converges to
\begin{equation*}
\begin{split}
&\int_{F_v^{\times}} \int_{F_v^{n-3}} W_1(\U_{\text{op}}^0({}^t y', 0) t'(a)) |a|^{s_0+\beta-1} \, dy' \, da \\
&\quad = \int_{F_v^{\times}} \int_{\U_{\text{op}}^{n-2}(F_v)} W_1(r t'(a)) |a|^{s_0+\beta-1} \, dr \, da.
\end{split}
\end{equation*}
Because the limit of this sequence of integrals is strictly nonzero, the integral itself must evaluate to a nonzero value for some specific choice of $\phi_1 \in \mc{S}(F_v)$ sufficiently far along the sequence. Note that the power of $|a|$ has shifted from $\beta$ to $\beta-1$, but since $\beta \in \RR$ is arbitrary, this shift is immaterial. Thus, it suffices to find $W_1 \in \mc{W}_{\psi_v}^{0}(\pi_v)$ such that 
\[
\int_{F_v^{\times}} \int_{\U_{\text{op}}^{n-2}(F_v)} W_1(r t'(a)) |a|^{s_0+\beta'} \, dr \, da \neq 0
\]
for $\beta' = \beta - 1$.

By repeating this process inductively and replacing the integration domain $\U_{\text{op}}^{n-2}(F_v)$ with $\U_{\text{op}}^{i}(F_v)$ for $i = n-3, \dots, 1$, we eventually reduce the problem to finding $W_2 \in \mc{W}_{\psi_v}^{0}(\pi_v)$ such that 
\[
\int_{F_v^{\times}} W_2(t'(a)) |a|^{s_0+\beta} \, da \neq 0.
\]

Next, we replace $W_2$ by a smoothed function
\[
W_2(g) = \int_{F_v} W_3(g \cdot \widehat{e_2(y)}) \phi_2(y) \, dy
\]
for some $W_3 \in \mc{W}_{\psi_v}^{0}(\pi_v)$ and $\phi_2 \in \mc{S}(F_v)$. Substituting this into the integral yields
\begin{align*}
\int_{F_v^{\times}} W_2(t'(a)) |a|^{s_0+\beta} \, da &= \int_{F_v^{\times}} W_3(t'(a)) |a|^{s_0+\beta} \left( \int_{F_v} \psi(ay) \phi_2(y) \, dy \right) da \\
&= \int_{F_v^{\times}} W_3(t'(a)) |a|^{s_0+\beta} \widehat{\phi}_2(a) \, da.
\end{align*}
By choosing a sequence of Schwartz functions $\phi_2 \in \mc{S}(F_v)$ such that $\widehat{\phi}_2$ forms an approximate identity converging to the Dirac measure at $1 \in F_v^{\times}$, the integral converges to the value of the integrand evaluated at $a=1$:
\[
W_3(t'(1)) |1|^{s_0+\beta} = W_3(\1).
\]
Therefore, by selecting $W_3 \in \mc{W}_{\psi_v}^{0}(\pi_v)$ such that $W_3(\1) \neq 0$, the limiting value is strictly nonzero. This guarantees the existence of an actual test function $\phi_2 \in \mc{S}(F_v)$ for which 
\[
\int_{F_v^{\times}} W_2(t'(a)) |a|^{s_0+\beta} \, da \neq 0,
\]
successfully satisfying the non-vanishing condition.

Finally, it remains to verify the absolute convergence of $Z_v(W, f_s)$ at $s=s_0$ for our specific choices. Applying the same limit arguments to the absolute values of the functions constructed above, the absolute integral converges to the evaluation of the positive functions at the identity:
\[
\int_{\K_v^{0}} \int_{F_v^{\times}} \int_{\U_{\text{op}}^0(F_v)} |W(r w_0 t(a) k)| \cdot |f_{s_0}(t(a) k)| \cdot |a|^{\alpha} \, dr \, da \, dk \longrightarrow |W_3|(\1) < \infty.
\]
This confirms absolute convergence and completes the proof.
\end{proof}

Proposition~\ref{p1}, Proposition~\ref{p2} and Proposition~\ref{p3} are concerned with $\H_n^{\ve}$ instead of $\G_n^{\ve}$. Therefore, to prove Theorem~\ref{t1}, we need a lemma which makes it possible to transfer the problem from $\G_n^{\ve}(\A)$ to $\H_n^{\ve}(\A)$.

Define two maps 
\noindent $\Res, \ \mf{J} \colon \mc{A}(\G_n^{\ve}) \to \mc{A}(\H_n^{\ve})$ as
\begin{align*}
&\Res(\wt{\varphi}) = \wt{\varphi} \big|_{\H_n^{\ve}(\A)}, \quad \wt{\varphi} \in \mc{A}(\G_n^{\ve}),\\
&\mf{J}(\wt{\varphi})(h) \coloneqq \int_{\mu_2(F) \bs \mu_2(\A)} \wt{\varphi}(h \cdot (\mathbf{t} \cdot \e)) \, d\mathbf{t}, \quad h \in \H_n^{\ve}(\A), \ \wt{\varphi} \in \mc{A}(\G_n^{\ve}).
\end{align*}

\begin{lem}\label{lg}
For some $\TT \in \mf{T}$, if $\wt{\pi}$ is an irreducible $\mu_{\ve}^{\TT}$-generic cuspidal automorphic representation of $\G_n^{\ve}(\A)$, there is an irreducible $\mu_{\ve}$-generic cuspidal automorphic representation $\pi$ of $\H_n^{\ve}(\A)$ which belongs to both $\Res(\wt{\pi})$ and $\mf{J}(\wt{\pi}\otimes \mathrm{sgn}_{\TT})$.
\end{lem}

\begin{proof}
Choose an arbitrary $\wt{\varphi} \in \wt{\pi}\otimes \mathrm{sgn}_{\TT}$. Since $\wt{\varphi}$ is $\K^{\ve}$-finite, there is a finite subset $S$ of places including all archimedean places of $F$ such that $\wt{\varphi}$ is right $\big(\prod_{v \notin S} \mu_2(F_v)\big)\cdot \e$-invariant. Since $\big(\prod_{v \in S} \mu_2(F_v)\big)\cdot \e$ is a finite set, we see that $\mf{J}(\wt{\varphi})(h) = c \cdot \sum_{\mathbf{t}_i \in \prod_{v \in S} \mu_2(F_v)} \wt{\varphi}(h \cdot (\mathbf{t}_i \cdot \e))$ for some nonzero constant $c$. Therefore, $\mf{J}(\wt{\varphi})$ belongs to $\Res(\wt{\pi}\otimes \mathrm{sgn}_{\TT})$. Moreover, since $\wt{\pi}$ is $\mu_{\ve}^{\TT}$-generic, $W_{\psi}^{\ve,\TT}(\wt{\varphi}) = W_{\psi}^{\ve}(\mf{J}(\wt{\varphi}\otimes \mathrm{sgn}_{\TT})) \neq 0$ for some $\wt{\varphi} \in \wt{\pi}$. Therefore, there is an irreducible $\mu_{\ve}$-generic cuspidal representation $\pi_0$ of $\H_n^{\ve}(\A)$ in $\mf{J}(\wt{\pi}\otimes \mathrm{sgn}_{\TT})$ which belongs to $\Res(\wt{\pi})$ because $\Res(\wt{\pi}\otimes \mathrm{sgn}_{\TT}) = \Res(\wt{\pi})$.
\end{proof}

Now we prove Theorem~\ref{t1}.

\begin{proof}[Proof of Theorem~\ref{t1}]
Since $\wt{\pi}$ is $\mu_{\ve}$-generic, there exists some $\TT \in \mf{T}$ such that $\wt{\pi}$ is $\mu_{\ve}^{\TT}$-generic. 

We first prove the implication (i) $\Rightarrow$ (ii). Fix an arbitrary place $v$ of $F$. We can uniquely express $\wt{\pi}_v$ as $L(\tau_{1,v} |\cdot|^{e_1}, \dots, \tau_{k,v} |\cdot|^{e_k}, \wt{\pi}_{0,v})$, the Langlands quotient of a standard module (see Sect.~\ref{atr}.) Then by \cite[Theorem~1]{Yam14},
\[
L(s,\wt{\pi}_v) = L(s,\wt{\pi}_{0,v}) \prod_{i=1}^k L_{\mathrm{GJ}}(s+e_i,\tau_{i,v}) L_{\mathrm{GJ}}(s-e_i,\tau_{i,v}^{\vee}),
\]
where $L_{\mathrm{GJ}}$ denotes the local $L$-factor of general linear groups defined by Godement and Jacquet (\cite{GJ72}). 

Define 
\[
s_0 \coloneqq \begin{cases} 1, & \text{if } \ve=0, \\ \frac{1}{2}, & \text{if } \ve=1. \end{cases}
\]
Note that $e_1 < \frac{1}{2}$ by the non-trivial Ramanujan bound \cite[Corollary~10.1]{CKPS04}; consequently, the product $\prod_{i=1}^k L_{\mathrm{GJ}}(s+e_i,\tau_{i,v}) L_{\mathrm{GJ}}(s-e_i,\tau_{i,v}^{\vee})$ is holomorphic at $s=s_0$ by \cite[Remark~3.2.4]{Jac79}. Furthermore, $L(s,\wt{\pi}_{0,v})$ is holomorphic at $s=s_0$ by \cite[Lemma~7.2]{Yam14}. Therefore, $L(s,\wt{\pi}_v)$ is holomorphic at $s=s_0$.

Observe that 
\[
L(s,\wt{\pi}) = L^S(s,\wt{\pi}) \prod_{v \in S} L(s,\wt{\pi}_v).
\]
Since the finite product $\prod_{v \in S} L(s,\wt{\pi}_v)$ is holomorphic at $s=s_0$, if $L(s,\wt{\pi})$ has a pole at $s=1$, it must arise from a pole of $L^S(s,\wt{\pi})$. Similarly, if $L(s,\wt{\pi})$ is nonzero at $s=\frac{1}{2}$, then $L^S(s,\wt{\pi})$ must also be nonzero at $s=\frac{1}{2}$. This completes the proof of the implication (i) $\Rightarrow$ (ii).

Next, we prove (ii) $\Rightarrow$ (iii). Since $\wt{\pi}$ is $\mu_{\ve}^{\TT}$-generic, by Lemma~\ref{lg}, there is an irreducible $\mu_{\ve}$-generic cuspidal representation $\pi$ of $\H_n^{\ve}(\A)$, which occurs as a sub-representation of $\mf{J}(\wt{\pi}\otimes \mathrm{sgn}_{\TT})$. Choose a decomposable element $\varphi=\otimes_v \varphi_v$ in $\pi$ and $f_s=\otimes_v f_{s,v} \in \Ind_{\B^{\ve}(\A)}^{\H^{\ve}(\A)}(|\cdot|^{s-\frac{1}{2}})$.

For a place $v$ of $F$, put 
\begin{align*}
\overline{L(s,\pi_v)} &= 
\begin{cases} 
L(s,\pi_v)\cdot \zeta_{v}(2s)^{-1}, &\text{if } \ve=0, \\ 
L(s,\pi_v), &\text{if } \ve=1, 
\end{cases} 
\\[10pt] 
\overline{L^{S}(s,\pi)} &= 
\begin{cases} 
L^S(s,\pi)\cdot \zeta^S(2s)^{-1}, &\text{if } \ve=0, \\ 
L^S(s,\pi), &\text{if } \ve=1. 
\end{cases}
\end{align*}

By \eqref{dec} and Proposition~\ref{p2}, for $\mathrm{Re}(s) \gg 0$, there is a finite set $S_1$ of places of $F$ containing all archimedean places of $F$ such that
\begin{equation}\label{ie}
    \I(\varphi,f_s) = \overline{L^{S_1}(s,\pi)}\cdot \prod_{v \in S_1}Z_v(W_{\psi_v}^{\ve}(\varphi_v),f_{s,v}).
\end{equation} 
We can enlarge $S_1$ so that $S \subset S_1$ and \eqref{ie} holds. By Proposition~\ref{p3}, for each $v \in S_1$, we can choose $W_v^0 \in \mc{W}_{\psi_v}^{\ve}(\pi_v)$ and a standard section $f_{s,v}^0\in \Ind_{\B^{\ve}(F_v)}^{\H^{\ve}(F_v)}(|\cdot|^{s-\frac{1}{2}})$ such that $Z_v(W_v^0 ,f_{s,v}^0)$ is holomorphic and nonzero at $s=s_0$. Take $\varphi^0=\otimes_v (\varphi^0)_v \in \pi$ such that $(\varphi^0)_v=\varphi_v$ for all $v \notin S_1$ and $W_{\psi_v}^{\ve}((\varphi^0)_v)=W_v^0$ for all $v \in S_1$. We also take $f_s'=\otimes_v (f_s')_v \in \Ind_{\B^{\ve}(\A)}^{\H^{\ve}(\A)}(|\cdot|^{s-\frac{1}{2}})$ such that $(f_s')_v=f_{s,v}$ for all $v \notin S_1$ and $(f_{s}')_v=f_{s,v}^0$ for all $v \in S_1$. (When $\ve=1$, we simply take $f_s'=|\cdot|^{s-\frac{1}{2}}$.)

Then for such a choice $\varphi^0$ and $f_s'$,
\[
    \I(\varphi^0,f_s') = \overline{L^{S_1}(s,\pi)}\cdot \prod_{v \in S_1}Z_v(W_{v}^{0},f_{s,v}^0)
\]
holds in the sense of meromorphic continuation. Note that $\overline{L^{S_1}(s,\pi)} = \overline{L^{S}(s,\pi)}\cdot \prod_{v \in S_1 \setminus S}\overline{L(s,\pi_v)}$. Since $L(s,\pi_v)$ is holomorphic and nonzero at $s=s_0$ by \cite[Corollary~10.1]{CKPS04} for $v \in S_1 \setminus S$, by the assumption, $\overline{L^{S_1}(s,\wt{\pi})} = \overline{L^{S_1}(s,\pi)}$ has a pole at $s=1$ in the case $\ve=0$ and is holomorphic and nonzero at $s=\frac{1}{2}$ in the case $\ve=1$. Choose a $\wt{\varphi}^0\in \wt{\pi}$ such that $\mf{J}(\wt{\varphi}^0\otimes \mathrm{sgn}_{\TT})=\varphi^0$. 

When $\ve=0$, $\I(\varphi^0,f_s')$ has a pole at $s=1$ and this must come from $E(f_s',h)$ and thus is simple. Therefore, 
\[
\int_{[\H^{0}]} (\varphi^0)^{\psi}(h)\cdot \mathrm{Res}_{s=1}(E(f_s',h)) \, dh \ne 0.
\]
Since $E(f_s',h)$ is a Siegel Eisenstein series, the residue of $E(f_s',h)$ at $s=1$ is constant. Then, $\mc{B}_{n-1,\psi}^{0,\TT}(\wt{\varphi}^0) = \int_{[\H^{0}]} (\varphi^0)^{\psi}(h) \, dh \ne 0$.

When $\ve=1$, $\I(\varphi^0,f_s')$ is holomorphic and nonzero at $s=\frac{1}{2}$. Note that $E(f_s',h)=f_s'(h)=|h|^{s-\frac{1}{2}}$. Therefore, 
\begin{align*}
\I(\varphi^0,f_s') \big|_{s=1/2} 
&= \int_{[\H^{\ve}]} (\varphi^0)^{\psi}(h)\cdot E(f_s',h) \big|_{s=1/2} \, dh \\
&= \int_{[\H^{\ve}]} (\varphi^0)^{\psi}(\varphi)(h) \, dh = \mc{B}_{n-1+\ve,\psi}^{\ve,\TT}(\wt{\varphi}^0).
\end{align*}
We proved the (ii) $\Rightarrow$ (iii) direction in both cases.

\noindent The direction (iii) $\Rightarrow$ (iv) is a consequence of Theorem~\ref{a1} and the direction (iv) $\Rightarrow$ (v) is obvious.

\noindent Lastly, we prove the (v) $\Rightarrow$ (i) direction. Suppose that $\Theta_{\psi,V_n^{\ve},W_{n-1+\ve}}^{\TT}(\wt{\pi})\ne0$. We know that $\Theta_{\psi,V_n^{\ve},W_{t}}(\wt{\pi}\otimes \mathrm{sgn}_{\TT})=0$ for $t<n-1+\ve$ by Corollary~\ref{c1}. Then the completed $L$-function $L(s,\wt{\pi})= \prod_{v} L(s,\wt{\pi}_v)$ has a pole at $s=0$ in the case $\ve=0$ and is holomorphic and nonzero at $s=\frac{1}{2}$ in the case $\ve=1$ by \cite[Theorem~2]{Yam14}. Furthermore, in the case $\ve=0$, it has a pole at $s=1$ by the functional equation \cite[Theorem~9.1]{Yam14} and the fact $L(s,\wt{\pi})=L(s,\wt{\pi}^{\vee})$ (\cite[Proposition~5.4]{Yam14}). This completes the proof. 
\end{proof}

Using Theorem~{\ref{t1}}, we can prove Theorem~{\ref{t2}} as follows.
\begin{proof}[Proof of Theorem~\ref{t2}] 
Choose an irreducible $\mu_{\ve}^{\TT}$-generic cuspidal automorphic representation $\wt{\pi}$ of $\G_n^{\ve}(\A)$ such that $\pi$ belongs to both $\text{Res}(\wt{\pi})$ and $\mf{J}(\wt{\pi}\otimes \text{sgn}_{\TT})$. (The existence of such \( \wt{\pi} \text{ and }  \TT\in \mf{T} \) is guaranteed by \cite[Proposition~2.4]{HKK25} in the case $\ve=0$ and the case $\ve=1$ is similar.) Then we have $L(s,\wt{\pi})=L(s,\pi)$ and $L^S(s,\wt{\pi})=L^S(s,\pi)$. Therefore, the direction (i)$\mapsto$(ii) follows from that of Theorem~{\ref{t1}}. By our choice of $\wt{\pi}$, it is easy to verify that $\mc{B}_{n-1+\ve,\psi}^{\ve,\TT}(\wt{\pi})\ne0$ is equivalent to $\mc{B}_{n-1+\ve,\psi}^{\ve}(\pi)\ne0$. Therefore, the proof for the directions (ii)$\mapsto$(iii)$\mapsto$(iv)$\mapsto$(v) is just the application of Theorem~\ref{a1} and Theorem~\ref{t1}. Therefore, we only prove (v)$\mapsto$(i) direction. By Theorem~{\ref{t1}}, it suffices to show that $\Theta_{\psi,V_n^{\ve},W_{n-1+\ve}}(\pi)\ne 0$  implies $\Theta_{\psi,V_n^{\ve},W_{n-1+\ve}}(\wt{\pi}\otimes \text{sgn}_{\TT'})\ne0$ for some $\TT' \in \mf{T}$.

Suppose that $\Theta_{\psi,V_n^{\ve},W_{n-1+\ve}}(\pi)\ne 0$. Then there is some $\phi \in \mc{S}(V_n^{\ve} \otimes Y_{n-1+\ve}^*)(\A)$ and $\varphi \in \pi$ such that $\theta_{\psi,V_n^{\ve},W_{n-1+\ve}}(\phi,\varphi) \ne 0$. Choose an element $\wt{\varphi} \in \wt{\pi}$ such that $\Res(\wt{\varphi})=\varphi$. It can be expressed as a sum of pure tensors, namely,
$\wt{\varphi}=\sum_{i=1}^\ell\wt{\varphi}_i$, where each $\wt{\varphi}_i$ is of the form $\wt{\varphi}_{i}=\otimes_v \wt{\varphi}_{i,v}$. Since $\Res(\wt{\varphi})=\sum_{i=1}^\ell\Res(\wt{\varphi}_i)=\varphi$ and $\theta_{\psi,V_n^{\ve},W_{n-1+\ve}}(\phi,\varphi)\ne 0$, there is a $1\le j \le \ell$ such that $\theta_{\psi,V_n^{\ve},W_{n-1+\ve}}(\phi,\text{Res}(\wt{\varphi}_j))  \ne 0$. Let us denote this particular $\wt{\varphi}_j$ by $\wt{\varphi}'$.

There is a finite set $S_0$ including all archimedean places of $F$ such that $\wt{\varphi}'$ is right $\big(\prod_{v \notin S_0} \mu_2(F_v)\big)\cdot \e$-invariant. For each $v \in S_0$, decompose $\wt{\varphi}'_{v}=\wt{\varphi}'_{v,1}+\wt{\varphi}'_{v,2}$ (one of $\wt{\varphi}'_{v,i}$ might be zero) such that $\wt{\varphi}'_{v,i}$ is in $(-1)^{i+1}$-eigenspace of $\e$ in $\wt{\pi}_v$ for each $i=1,2$. Therefore, we can write $\wt{\varphi}'=\sum_{m=1}^k \wt{\varphi}'_m$ such that for each $\wt{\varphi}'_m=\otimes_v (\wt{\varphi}'_m)_v$, where $(\wt{\varphi}'_m)_v$ is an eigen-vector of $\e$ for each $v \in S_0$ and $\e$-invariant for each $v \notin S_0$.

Again, since $\theta_{\psi,V_n^{\ve},W_{n-1+\ve}}(\phi,\text{Res}(\wt{\varphi}'))  \ne 0$ and $\Res(\wt{\varphi}')=\sum_{m=1}^k \Res(\wt{\varphi}'_m)$, there is a $1\le k_0 \le k$ such that $\theta_{\psi,V_n^{\ve},W_{n-1+\ve}}(\phi,\text{Res}(\wt{\varphi}'_{k_0}))  \ne 0$. We may assume that $k_0=1$.

Since $\wt{\varphi}'_1=\otimes_v \wt{\varphi}'_{1,v} \in \wt{\pi}$, denote $\t_0=(\t_{0,v}) \in \prod_{v \in S_0}\mu_2(F_v)$  given by \[\t_{0,v} =
\begin{cases} 1, & \text{ if } \wt{\varphi}_{1,v} \text{ is an eigenvector of $\e$ with the eigenvalue 1}  \\ -1, & \text{ if } \wt{\varphi}_{1,v} \text{ is an eigenvector of $\e$ with the eigenvalue $-1$}.
\end{cases}
\]

Let $\TT_0$ be a subset of $S_0$ consisting of place $v$ such that $\t_{0,v}=-1$. Since $(\e\cdot \wt{\varphi}'_1)(\textbf{1})=\wt{\varphi}'_1(\textbf{1})$, we see that the number of elements in $S$ is even and so $\TT_0 \in \mf{T}$. 

In a similar way, we can find $\phi'=\otimes_v \phi'_v \in \mc{S}(V_n^{\ve} \otimes Y_{n-1+\ve}^*)(\A)$ such that  $\theta_{\psi,V_n^{\ve},W_{n-1+\ve}}(\phi',\text{Res}(\wt{\varphi}_1'))  \ne 0$ and there is some $\TT_1\in \mf{T}$ such that $\phi'_v$ is an eigenvector of $\e$ with eigenvalue $-1$ for $v \in \TT_1$ and $\e$-invariant for $v\notin \TT_1$. Denote $\TT'=\TT_0 \cup \TT_1- (\TT_0\cap\TT_1)$. Then $|\TT'|=|\TT_0|+|\TT_1|-2\cdot |\TT_0\cap\TT_1|$ is even, we see that $\TT' \in \mf{T}$.

Then for $\wt{\varphi}' =\wt{\varphi}'_1\otimes \text{sgn}_{\TT'}\in \wt{\pi} \otimes \text{sgn}_{\TT'}$, 
\begin{align*}
\theta_{\psi,V_n^{\ve},W_{n-1+\ve}}(\phi',\wt{\varphi}')(h') 
&= \int_{[\H_n^{\ve}]} \int_{[\mu_2]} \theta_{\psi,V_n^{\ve},W_{n-1+\ve}}(\phi';h\cdot(\t \cdot \ve),h')  \wt{\varphi}'(h\cdot(\t \cdot \ve)) \, d\t \, dh \\
&= \int_{[\H_n^{\ve}]} \theta_{\psi,V_n^{\ve},W_{n-1+\ve}}(\phi';h,h')\cdot \wt{\varphi}'(h) \, dh \\
&= \theta_{\psi,V_n^{\ve},W_{n-1+\ve}}(\phi',\text{Res}(\wt{\varphi}'))(h')\\ &= \theta_{\psi,V_n^{\ve},W_{n-1+\ve}}(\phi',\text{Res}(\wt{\varphi}_1'))(h') \ne 0.
\end{align*} 
Therefore, $\Theta_{\psi, V_n^{\ve},W_{n-1+\ve}}(\wt{\pi}\otimes \text{sgn}_{\TT'})\ne0$ and this completes the proof of Theorem~\ref{t2}.
\end{proof}
\begin{rem}\label{non-ge}
In the proofs of Theorems \ref{t1} and \ref{t2}, the implications (iii) $\Rightarrow$ (iv) $\Rightarrow$ (v) $\Rightarrow$ (i) remain valid even when $\wt{\pi}$ and $\pi$ are not generic. In the implication (i) $\Rightarrow$ (ii), we specifically utilized the almost tempered property of the local components of globally generic representations (see Remark \ref{Ram}). Consequently, the chain of implications (iii) $\Rightarrow$ (iv) $\Rightarrow$ (v) $\Rightarrow$ (i) $\Rightarrow$ (ii) remains valid even if $\wt{\pi}$ and $\pi$ are not globally generic, provided that they are almost tempered at every place.
\end{rem}

As consequences of Theorems \ref{t1} and \ref{t2}, we obtain the following corollaries.

\begin{cor}\label{co1}
Let $\wt{\pi}$ be an irreducible $\mu_{\ve}$-generic cuspidal representation of $\G_n^{\ve}(\A)$. If $\Theta_{\psi,V_n^{\ve},W_{n-1+\ve}}(\wt{\pi}\otimes \mathrm{sgn}_{\TT_0})$ is nonzero for some $\TT_0 \in \mf{T}$, there exists some $\TT \in \mf{T}$ such that $\Theta_{\psi,V_n^{\ve},W_{n-1+\ve}}(\wt{\pi}\otimes \mathrm{sgn}_{\TT})$ is nonzero and $(\psi,\lambda_{\ve})$-generic.
\end{cor}

\begin{cor}\label{co2}
Let $\pi$ be an irreducible $\mu_{\ve}$-generic cuspidal representation of $\H_n^{\ve}(\A)$. If $\Theta_{\psi,V_n^{\ve},W_{n-1+\ve}}(\pi)$ is nonzero, then it must be $(\psi,\lambda_{\ve})$-generic.
\end{cor}

Combining Proposition \ref{c1}, Corollary \ref{a3}, and Corollary \ref{co1}, we obtain the following theorem.

\begin{thm}\label{p5}
Let $\wt{\pi}$ be an irreducible $\mu_{\ve}$-generic cuspidal representation of $\G_n^{\ve}(\A)$. Suppose there is some $\TT_0 \in \mf{T}$ such that $\Theta_{\psi,V_n^{\ve},W_{k}}(\wt{\pi}\otimes \mathrm{sgn}_{\TT_0})$ is nonzero and $\Theta_{\psi,V_n^{\ve},W_{i}}(\wt{\pi}\otimes \mathrm{sgn}_{\TT_0})=0$ for all $i<k$. Then there is some $\TT_1 \in \mf{T}$ such that $\Theta_{\psi,V_n^{\ve},W_{k}}(\wt{\pi}\otimes \mathrm{sgn}_{\TT_1})$ is $(\psi,\lambda_{\ve})$-generic and cuspidal.
\end{thm}

\begin{proof}
By Proposition \ref{c1}, $\Theta_{\psi,V_n^{\ve},W_{i}}(\wt{\pi}\otimes \mathrm{sgn}_{\TT})=0$ for all $i<n-1+\ve$ and $\TT \in \mf{T}$, because $\big(\Theta_{\psi,V_n^{\ve},W_{i}}(\wt{\pi}\otimes \mathrm{sgn}_{\TT})\big)_v$ is the local theta lift of $(\wt{\pi}\otimes \mathrm{sgn}_{\TT})_v$. Therefore, if $\Theta_{\psi,V_n^{\ve},W_{n-1+\ve}}(\wt{\pi}\otimes \mathrm{sgn}_{\TT_0}) \neq 0$ for some $\TT_0 \in \mf{T}$, there exists some $\TT_1 \in \mf{T}$ such that $\Theta_{\psi,V_n^{\ve},W_{n-1+\ve}}(\wt{\pi}\otimes \mathrm{sgn}_{\TT_1})$ is $(\psi,\lambda_{\ve})$-generic by Corollary \ref{co1}. Furthermore, by the Rallis tower property of global theta liftings, it is cuspidal. 

Suppose that $\Theta_{\psi,V_n^{\ve},W_{n-1+\ve}}(\wt{\pi}\otimes \mathrm{sgn}_{\TT})=0$ for all $\TT \in \mf{T}$. Since $\pi$ is $\mu_{\ve}$-generic, by Corollary \ref{a3}, there exists some $\TT_2 \in \mf{T}$ such that $\Theta_{\psi,V_n^{\ve},W_{n+\ve}}(\wt{\pi}\otimes \mathrm{sgn}_{\TT_2})$ is nonzero, cuspidal, and $(\psi,\lambda_{\ve})$-generic.
\end{proof}

Similarly, we can prove the following theorem using Corollary \ref{co2}. We omit the proof.

\begin{thm}\label{p6}
Let $\pi$ be an irreducible $\mu_{\ve}$-generic cuspidal representation of $\H_n^{\ve}(\A)$. If $\Theta_{\psi,V_n^{\ve},W_{k}}(\pi)$ is nonzero and $\Theta_{\psi,V_n^{\ve},W_{i}}(\pi)=0$ for all $i<k$, then $\Theta_{\psi,V_n^{\ve},W_{k}}(\pi)$ is $(\psi,\lambda_{\ve})$-generic and cuspidal.
\end{thm}

\section{The global theta lifts from $\wt{\J}_{n}$ to $\G_{k}^{\ve}$}\label{sec:6}
Throughout this section, we assume that $F$ is a number field. Let $\{V_i^{\ve}\}_{i \ge 1}$ be the Witt tower of quadratic spaces containing $V^{\ve}$. We consider the global Weil representation $\omega_{\psi,W_n,V_k^{\ve}} \coloneqq \bigotimes_v \omega_{\psi_v,W_{n,v},V_{k,v}^{\ve}}$ of $\wt{\J}_n(\A) \times \G_k^{\ve}(\A)$. 

Let $W'$ and $Y'$ denote the subspaces of $W_n$ spanned by $\{ f_n, f_n^* \}$ and $\{ f_n \}$, respectively. Then $\wt{Y}_n \coloneqq Y_{n-1}^* \oplus Y'$ forms a maximal isotropic subspace of $W_{n}$. We can then decompose $\omega_{\psi,W_n,V_k^{\ve}}$ as a tensor product $\omega_{\psi,W_n,\overline{V_{k-1+\ve}}} \otimes \omega_{\psi,W_n,V^{\ve}}$ realized in the mixed Schwartz-Bruhat space $\mc{S}(X_{k-1+\ve}^* \otimes W_n)(\A) \otimes \mc{S}(V^{\ve} \otimes \wt{Y}_n)(\A)$.

There is an equivariant map $\theta_{\psi,W_n,V_k^{\ve}} \colon \mc{S}(X_{k-1+\ve}^* \otimes W_n)(\A) \otimes \mc{S}(V^{\ve} \otimes \wt{Y}_n)(\A) \to \mc{A}( \wt{\J}_n \times \G_k^{\ve})$ given by the theta series
\[
\theta_{\psi,W_n,V_k^{\ve}}(\phi;\wt{h'},g) \coloneqq \sum_{(\mathbf{x},\mathbf{y}) \in (X_{k-1+\ve}^* \otimes W_n)(F) \oplus (V^{\ve} \otimes \wt{Y}_n)(F)}\omega_{\psi,W_n,V_k^{\ve}}(\wt{h'},g)(\phi)(\mathbf{x},\mathbf{y}).
\]

For $\varphi \in \mc{A}(\wt{\J}_n)$, take $\ve=1$ if $\varphi$ is genuine and $\ve=0$ if $\varphi$ is non-genuine. When $\ve=1$, for any $\phi \in \mc{S}(X_{k-1+\ve}^* \otimes W_n)(\A) \otimes \mc{S}(V^{\ve} \otimes \wt{Y}_n)(\A)$, the expression $\theta_{\psi,W_n,V_k^{\ve}}(\phi;(h',\ve'),g) \cdot \varphi((h',\ve'))$ is independent of $\ve'$. Consequently, by evaluating at $\ve'=1$, we simplify the notation and denote  $\theta_{\psi,W_n,V_k^{\ve}}(\phi;(h',1),g)$ and $\varphi((h',1))$ by $\theta_{\psi,W_n,V_k^{\ve}}(\phi;h',g)$ and $\varphi(h')$, respectively.

For $\phi \in \mc{S}(X_{k-1+\ve}^* \otimes W_n)(\A) \otimes \mc{S}(V^{\ve} \otimes \wt{Y}_n)(\A)$, we define
\[
\theta_{\psi,W_n,V_k^{\ve}}(\phi,\varphi)(g) \coloneqq \int_{\J_n(F) \bs \J_n(\A)} \theta_{\psi,W_n,V_k^{\ve}}(\phi;h',g)\varphi(h') \, dh', \quad g \in \G_n^{\ve}(\A).
\]

For an irreducible cuspidal automorphic representation $\sigma$ of $\wt{\J}_n(\A)$, we define the space of global theta lifts as
\[
\Theta_{\psi,W_n,V_k^{\ve_{\sigma}}}(\sigma) \coloneqq \left\{ \theta_{\psi,W_n,V_k^{\ve_{\sigma}}}(\phi,\varphi) \;\middle|\; \phi \in \omega_{\psi,W_n,V_k^{\ve_{\sigma}}}, \ \varphi\in \sigma \right\}.
\] 
Then $\Theta_{\psi,W_n,V_k^{\ve_{\sigma}}}(\sigma)$ is an automorphic representation of $\G_k^{\ve_{\sigma}}(\A)$ (or of $\H_k^{\ve_{\sigma}}(\A)$ upon restriction). We define an automorphic representation $\tau_{\sigma}$ of $\wt{\J}_{n,n-1}'(\A)$ by
\begin{equation}\label{tau}
\tau_{\sigma} \coloneqq \begin{cases} 
\omega_{\psi_{cd},W'}, & \text{if } \ve_{\sigma}=0, \\ 
\II, & \text{if } \ve_{\sigma}=1. 
\end{cases}
\end{equation}
(Here, $\omega_{\psi_{cd},W'}$ is a representation acting on $\operatorname{Im}(\theta_{\psi_{cd},W'}')\,$; see \eqref{gtheta_prime}.)

For $k \in \{n-\ve_{\sigma}, n+1-\ve_{\sigma}\}$, consider the Fourier--Jacobi period $\mc{FJ}_{k,\psi}^{\lambda_{\ve_{\sigma}}}$, which acts on the space
\[
\begin{cases}
\sigma \boxtimes \tau_{\sigma} \boxtimes \nu_{\psi^{-1},W'}^{\lambda_{\ve_{\sigma}}}, & \text{if } k=n-\ve_{\sigma}, \\ 
\sigma, & \text{if } k=n+1-\ve_{\sigma}.
\end{cases}
\]

The following theorem states that the non-vanishing of $\mc{FJ}_{k-1+\ve_{\sigma},\psi}^{\lambda_{\ve_{\sigma}}}$ is deeply intertwined with the properties of the global theta lift.
\begin{thm}\label{b10} 
Let $\sigma$ be an irreducible cuspidal representation of $\wt{\J}_n(\A)$. For $k>n+1-\ve_{\sigma}$, if $\Theta_{\psi^{-1},W_{n},V_k^{\ve_{\sigma}}}(\sigma)$ is nonzero, then it is not generic. For $k \in \{n-\ve_{\sigma}, n+1-\ve_{\sigma}\}$, the non-vanishing of $\mc{FJ}_{k-1+\ve_{\sigma},\psi}^{\lambda_{\ve_{\sigma}}}$ is equivalent to the condition that $\Theta_{\psi^{-1},W_{n},V_k^{\ve_{\sigma}}}(\sigma)$ is nonzero and $\mu_{\ve_{\sigma}}$-generic.
\end{thm}

In the case where $\sigma$ is non-genuine (i.e., $\ve_\sigma = 0$) and $V_k^0$ is of type $(1,1)$ for $k \in \{n, n+1\}$, a similar result was established in \cite[Theorem~2.2]{GRS97} and in the subsequent discussion following \cite[Proposition~2.7]{GRS97}. Although our definition of the period $\mc{FJ}_{n-1, \psi}^{1}$ differs from the period $\mc{P}_{\psi^{-1}}$ appearing in \cite[pp.~85, 102]{GRS97}, the theorem implies that the non-vanishing of $\mc{FJ}_{n-1, \psi}^{1}$ is equivalent to the non-vanishing of $\mc{P}_{\psi^{-1}}$. 

When $\sigma$ is genuine (i.e., $\ve_\sigma = 1$), the result was established in \cite[Propositions~2 and 3]{Fu95}, with the exception of the case $k = n-1$. Below, we present a unified proof that not only covers the remaining cases from both references but also extends the results to spaces $V_k^0$ of arbitrary type $(d,c)$.

\begin{proof}
For convenience, we shall simply denote $\ve_\sigma$ by $\ve$. Using an argument analogous to the proof of Theorem \ref{t2}, one can demonstrate that the $\mu_\ve$-genericity of an automorphic representation of $\G_k^\ve$ is equivalent to that of its restriction to $\H_k^\ve$. Therefore, we shall focus on proving the case where $\Theta_{\psi^{-1},W_n,V_k^\ve}(\sigma)$ is regarded as a representation of $\H_k^\ve(\A)$.

For $f \in \Theta_{\psi^{-1},W_n,V_k^{\ve}}(\sigma)$, let us compute its Whittaker-Bessel period 
\[
W_{\psi}^{\ve}(f) = \int_{\U_{k}(F)\backslash \U_{k}(\A)} \mu_{\ve}^{-1}(u)f(u) \, du.
\] 
Write 
\[
f(h) = \int_{[\J_n]} \Theta_{\psi^{-1},W_n,V_k^{\ve}}(\phi;h',h)\varphi(h') \, dh'
\] 
for some $\phi \in \mc{S}(X_{k-1+\ve}^* \otimes W_n)(\A) \otimes \mc{S}(V^{\ve} \otimes \wt{Y}_n)(\A)$ and $\varphi \in \sigma$. 

Suppose that $\phi=\phi_1 \otimes \phi_2$, where $\phi_1\in \mc{S}(X_{k-1+\ve}^* \otimes W_n)(\A)$ and $\phi_2\in \mc{S}(V^{\ve} \otimes \wt{Y}_n)(\A)$. We use the following basis to represent the elements of $\H_k^{\ve}$ as matrices:
\[
\begin{cases}
\{e_1,\dots,e_{k-1},e,e',e_{k-1}^*,\dots,e_1^*\}, \quad &\text{when } \ve=0, \\ 
\{e_1,\dots,e_{k},e,e_{k}^*,\dots,e_1^*\}, \quad &\text{when } \ve=1. 
\end{cases}
\] 
Similarly, we represent the elements of $\J_{n}$ as matrices with respect to the basis $\{f_1,\dots,f_{n},f_{n}^*,\dots,f_{1}^*\}$.

For $z \in Z_{k-1+\ve}$, define
\[
v(z) = \begin{pmatrix}z & 0 &0 \\ 0 & \mathrm{Id}_{2-\ve} & 0\\ 0 & 0 & z^{\star}\end{pmatrix} \in \H_k^{\ve},
\] 
for $s \in S_{k-1+\ve}$, define
\[
n(s) = \begin{pmatrix} \mathrm{Id}_{k-1+\ve} & 0 & s \\ 0 & \mathrm{Id}_{2-\ve} & 0 \\ 0 & 0 & \mathrm{Id}_{k-1+\ve}\end{pmatrix} \in \H_k^{\ve},
\]
and for $m \in M_{k-1+\ve,2-\ve}$, define
\[
l(m) = \begin{pmatrix} \mathrm{Id}_{k-1+\ve} & m & -\frac{1}{2}mm' \\ 0 & \mathrm{Id}_{2-\ve} & m' \\ 0 & 0 & \mathrm{Id}_{k-1+\ve}\end{pmatrix} \in \H_k^{\ve}.
\]
(\emph{Caveat:} The definition of $v(z)$ here differs slightly from the one provided in the proof of Theorem \ref{a1}.)

Under the decomposition $\U_k = n(S_{k-1+\ve}) \cdot l(M_{k-1+\ve,2-\ve}) \cdot v(Z_{k-1+\ve})$, the Haar measure factors as $du = ds \, dm \, dz$ according to our conventions in Section~\ref{subsec:2-1}.

Therefore, we can write
\begin{align*}
W_{\psi}^{\ve}(f) &= \int_{[M_{k-1+\ve,2-\ve}]}\int_{[Z_{k-1+\ve}]}\mu_{\ve}^{-1}(l(m)v(z)) \\
&\quad \times \int_{[S_{k-1+\ve}]}\mu_{\ve}^{-1}(n(s))f\big(n(s)l(m)v(z)\big) \, ds \, dz \, dm.
\end{align*}

Let us first compute the partial Fourier coefficient $W^f(h)$ defined by
\[
W^f(h) \coloneqq \int_{[S_{k-1+\ve}]}\mu_{\ve}^{-1}(n(s))f(n(s)h) \, ds, \quad h\in \H_k^{\ve}(\A).
\]
We have
\begin{align*}
W^f(h) &= \int_{[S_{k-1+\ve}]}\mu_{\ve}^{-1}(n(s))\left(\int_{[\J_n]} \Theta_{\psi^{-1},W_n,V_k^{\ve}}(\phi;h',n(s)h)\varphi(h') \, dh'\right) ds \\
&= \int_{[S_{k-1+\ve}]}\mu_{\ve}^{-1}(n(s))\Bigg(\int_{[\J_{n}]}\sum_{\substack{\mathbf{x}\in (W_n(F))^{k-1+\ve}\\ \mathbf{y} \in (V^{\ve} \otimes \wt{Y}_n)(F)}} \\
&\quad \big(\omega_{\psi^{-1},W_n,V_k^{\ve}}(h',n(s)h)(\phi_1\otimes \phi_2)\big)(\mathbf{x};\mathbf{y})\varphi(h') \, dh'\Bigg) ds \\
&= \int_{[\J_{n}]}\sum_{\substack{\mathbf{x} \in \WW_{0}\\ \mathbf{y} \in (V^{\ve} \otimes \wt{Y}_n)(F)}}\big(\omega_{\psi^{-1},W_n,V_k^{\ve}}(h',h)(\phi_1\otimes \phi_2)\big)(\mathbf{x};\mathbf{y})\varphi(h') \, dh',
\end{align*}
where 
\[
\WW_0 = \left\{\mathbf{x}\in (W_{n}(F))^{k-1+\ve} \;\middle|\; \mathrm{Gr}(\mathbf{x}) = \begin{pmatrix} 0 & \cdots & 0 \\ \vdots  & \ddots & \vdots \\ 0 & \cdots & 0\end{pmatrix}\right\}.
\]
(The last equality follows from \eqref{a7}.)

Given any $\mathbf{x}=(x_1,\dots,x_{k-1+\ve})\in \WW_{0}$ and $\mathbf{y} \in (V^{\ve} \otimes \wt{Y}_n)(F)$, if the set $\{x_1,\dots,x_{k-1+\ve}\}$ is linearly dependent, then an argument analogous to the proof of Theorem \ref{a1} yields
\begin{equation*}
\begin{split}
&\int_{[Z_{k-1+\ve}]} \mu_{\ve}^{-1}(l(m)v(z)) \int_{[\J_{n}]} \big(\omega_{\psi^{-1},W_n,V_k^{\ve}}(h',l(m)v(z)h)(\phi_1\otimes \phi_2)\big)(\mathbf{x};\mathbf{y}) \varphi(h') \, dh' \, dz = 0.
\end{split}
\end{equation*}
When $k>n+1-\ve$, the set $\{x_1,\dots,x_{k-1+\ve}\}$ is necessarily linearly dependent because the dimension of a maximal isotropic subspace of $W_n$ is $n$. Consequently, $W_{\psi}^{\ve}(f)=0$. This completes the proof of the first statement.

When $k=n-\ve \text{ or }n+1-\ve$, only the linearly independent sets $\{x_1,\dots,x_{k-1+\ve}\}$ in $\WW_0$ contribute to non-trivial summands in $W^f(h)$. By the Witt extension theorem, there is only one orbit for the $\J_n(F)$-action on the linearly independent subsets in $\WW_{0}$. Choose its representative as $\w_0=(f_1,\dots,f_{k-1+\ve})\in \WW_0$ and define 
\[
\R_{k-1+\ve}'' \coloneqq \{h'\in \J_n \mid h'\w_0=\w_0 \}.
\]
Then
\[
\R_{k-1+\ve}'' = \begin{cases}
\begin{pmatrix} \mathrm{Id}_{n-2} & 0 & *&*&* \\ &  1 &  *&*&* \\  & & \mathrm{Id}_2 &*&*\\  &  & & 1 & 0\\ & & & & \mathrm{Id}_{n-2} \end{pmatrix} \in \J_{n}, & k=n-\ve, \\ \\ 
\begin{pmatrix}\mathrm{Id}_{n-1} &0&*&* \\ &1&*&* \\&&1&0 \\ &&&\mathrm{Id}_{n-1}\end{pmatrix} \in \J_{n}, & k=n+1-\ve.
\end{cases}
\]
Note that 
\begin{equation*}
\begin{split}
&\sum_{\mathbf{x} \in \WW_0}\omega_{\psi^{-1},W_n,\overline{V_{k-1+\ve}}}(\phi_1)(\mathbf{x})  = \sum_{h_0' \in \R_{k-1+\ve}''(F) \bs \J_n(F)}  \big(\omega_{\psi^{-1},W_n,\overline{V_{k-1+\ve}}}(h_0',\mathbf{1})(\phi_1)\big)(\w_0)
\end{split}
\end{equation*}
and the theta function on $\J_n(\A)$ given by
\[
\sum_{\mathbf{y} \in (V^{\ve} \otimes \wt{Y}_n)(F)}\big(\omega_{\psi^{-1},W_n,V^{\ve}}(h',\1)\phi_2\big)(\mathbf{y})
\]
is $\J_n(F)$-invariant. Thus,
\begin{align*}
W_{\psi}^{\ve}(f) 
&= \sum_{\mathbf{y} \in (V^{\ve} \otimes \wt{Y}_n)(F)} \int_{[Z_{k-1+\ve}]} \int_{[M_{k-1+\ve,2-\ve}]} \sum_{h_0' \in \R_{k-1+\ve}''(F) \bs \J_n(F)} \int_{[\J_{n}]} \mu_{\ve}^{-1}(l(m)v(z)) \\
&\quad \times \big(\omega_{\psi^{-1},W_n,V_k^{\ve}}(h_0'h',l(m)v(z))(\phi_1\otimes \phi_2)\big)(\w_0;\mathbf{y})\varphi(h') \, dh' \, dm \, dz \\
&= \sum_{\mathbf{y} \in (V^{\ve} \otimes \wt{Y}_n)(F)} \int_{[Z_{k-1+\ve}]} \int_{[M_{k-1+\ve,2-\ve}]}  \int_{\R_{k-1+\ve}''(F) \bs \J_{n}(\A)} \mu_{\ve}^{-1}(l(m)v(z)) \\
&\quad \times \big(\omega_{\psi^{-1},W_n,V_k^{\ve}}(h',l(m)v(z))(\phi_1\otimes \phi_2)\big)(\w_0;\mathbf{y})\varphi(h') \, dh' \, dm \, dz.
\end{align*}

Using the basis $\{e,e'\}$ of $V^{\ve}$, we identify $V^{\ve} \otimes \wt{Y}_n$ with $(\wt{Y}_n)^{2-\ve}$. (Note that $e'=0$ when $\ve=1$.) Then for each $\mathbf{y}=(y_1,y_{2-\ve}) \in \wt{Y}_n(F)^{2-\ve}$ (i.e., $\mathbf{y}=e \otimes y_1 + e'\otimes y_{2-\ve}$), utilizing \eqref{a9}, we have
\begin{align*}
&\int_{[M_{k-1+\ve,2-\ve}]}\mu_{\ve}^{-1}(l(m)) \cdot \big(\omega_{\psi^{-1},W_n,V_k^{\ve}}(h',l(m)h)(\phi_1\otimes \phi_{2})\big)(\w_0;\mathbf{y}) \, dm \\
&= \int_{[M_{k-1+\ve,2-\ve}]}\mu_{\ve}^{-1}(l(m)) \cdot \psi^{-1}\left( \sum_{i=1}^{k-1+\ve}  m_{i,1} \langle y_1,f_i \rangle_{W_n} + m_{i,{2-\ve}} \langle y_{2-\ve},f_i\rangle_{W_n} \right) \\
&\quad \times \big(\omega_{\psi^{-1},W_n,V_k^{\ve}}(h',h)(\phi_1\otimes \phi_2)\big)(\w_0;\mathbf{y}) \, dm.
\end{align*}

Observe that $\mu_{\ve}^{-1}(l(m))=\psi^{-1}(m_{k-1+\ve,1})$ and that the integral of a character over a compact group is nonzero if and only if the character is trivial. Accordingly, the above integral is nonzero if and only if $\mathbf{y}=(f_{k-1+\ve}^*+c_1 \cdot f_{k+\ve}, \ c_{2-\ve} \cdot f_{k+\ve}) \in \wt{Y}_n(F)^{2-\ve}$ for some $c_1,c_{2-\ve} \in F^{2-\ve}$. (Here, $f_{k+\ve}=0$ if $k=n+1-\ve$.) Put $\y_{c_1,c_{2-\ve}} = (f_{k-1+\ve}^*+c_1 \cdot f_{k+\ve}, \ c_{2-\ve} \cdot f_{k+\ve}) \in \wt{Y}_n(F)^{2-\ve}$.

Define $\y_{0} \in (V^{\ve}(F))^{k-1+\ve}$ and $y_{c_1,c_{2-\ve}}\in V^{\ve}(F)$ by
\[
\y_{0}=(0,\dots,0,e), \quad y_{c_1,c_{2-\ve}}=c_1 \cdot e +c_{2-\ve} \cdot e'.
\] 
Then, using the basis $\{f_1^*,\dots,f_{n-1}^*,f_n\}$ of $\wt{Y}_n$, we can express $\y_{c_1,c_{2-\ve}} \in (V^{\ve}(F))^{n}$ as
\[ 
\y_{c_1,c_{2-\ve}} = \begin{cases} (\y_{0},y_{c_1,c_{2-\ve}}), & k=n-\ve, \\ \y_0, & k=n+1-\ve. \end{cases}
\]

For the Weil representation $\omega_{\psi^{-1},W_{n},V^{\ve}}$, when $k=n+1-\ve$, we use the model $\mc{S}(Y_{n}^* \otimes V^{\ve})(\A)$. When $k=n-\ve$, we use the mixed model $\mc{S}(Y_{n-1}^* \otimes V^{\ve})(\A)\otimes \mc{S}(Y' \otimes V^{\ve})(\A)$, regarding it as the tensor product representation $\omega_{\psi^{-1},W_{n-1},V^{\ve}}\otimes \omega_{\psi^{-1},W',V^{\ve}}$.
In this case, we assume that $\phi_2$ is a pure tensor product $\phi_{2,1}\otimes \phi_{2,2}\in \mc{S}( Y_{n-1}^* \otimes V^{\ve})(\A)\otimes \mc{S}(Y' \otimes V^{\ve})(\A)$.

For $s \in S_{k-1+\ve}$, define
\[
n'(s) = \begin{cases}
\begin{pmatrix} \mathrm{Id}_{n-1} &  & s \\  & \mathrm{Id}_{2} &  \\  &  & \mathrm{Id}_{n-1}\end{pmatrix} \in \J_n, & k=n-\ve, \\ \\ 
\begin{pmatrix} \mathrm{Id}_{n} & s  \\   & \mathrm{Id}_{n}\end{pmatrix} \in \J_n, & k=n+1-\ve. 
\end{cases}
\]
By \eqref{a7}, for $s\in S_{k-1+\ve}(\A)$, we have
\begin{align*}
&\omega_{\psi^{-1},W_n,V_k^{\ve}}(n'(s),\1)(\phi_1\otimes \phi_2)(\w_0;\y_{c_1,c_{2-\ve}})= \phi_1(\w_0) \otimes \big(\omega_{\psi^{-1},W_{n},V^{\ve}}(n'(s),\1)(\phi_2)\big)(\y_{c_1,c_{2-\ve}}) \\
&= \psi^{-1}\left(\frac{1}{2}\mathrm{tr}\big(\mathrm{Gr}(\y_{0}) s \varpi_{k-1+\ve} \big)\right) \phi_1(\w_0) \phi_2(\y_{c_1,c_{2-\ve}}) = \psi^{-1}(\lambda_{\ve} s_{k-1+\ve,1}) (\phi_1\otimes \phi_2)(\w_0;\y_{c_1,c_{2-\ve}}).
\end{align*}

For $m \in M_{n-1,2}$, define
\[
l'(m) = \begin{pmatrix} \mathrm{Id}_{n-1} & m & -\frac{1}{2}mm' \\ 0 & \mathrm{Id}_{2} & m' \\ 0 & 0 & \mathrm{Id}_{n-1}\end{pmatrix} \in \J_n.
\]
Applying \eqref{a11} with the substitutions $V=W_n$, $W=V^{\ve}$, and $l=n$, we obtain, for $m\in M_{n-1,2}(\A)$,
\begin{align*}
&\omega_{\psi^{-1},W_n,V_{n-\ve}^{\ve}}(l'(m),\1)(\phi_1\otimes \phi_2)(\w_0;\y_{c_1,c_{2-\ve}}) = \phi_1(\w_0) \omega_{\psi^{-1},W_n,V^{\ve}}(l'(m),\1)(\phi_2)(\y_{c_1,c_{2-\ve}}) \\
&= \phi_1(\w_0) \phi_{2,1}(\y_{0}) \phi_{2,2}\big((m_{n-1,1} e + y_{c_1,c_{2-\ve}})\otimes f_n^*\big) \times \psi^{-1} \big(2\lambda_{\ve}m_{n-1,2} (c_1+m_{n-1,1})\big).
\end{align*}

For $\alpha \in \J_{n,n-1}'$, define the map $\k' \colon \J_{n,n-1}' \to \J_n$ by
\[
\k'(\alpha) = \begin{pmatrix}\mathrm{Id}_{n-1} & 0 &0 \\ 0 & \alpha & 0\\ 0 & 0 & \mathrm{Id}_{n-1}\end{pmatrix} \in \J_n.
\] 
Then \[\R_{k-1+\ve}'' = \begin{cases} \k'(\J_{n,n-1}') \cdot l'(M_{n-1,2})\cdot n'(S_{n-1}), & \text{if } k=n-\ve, \\ n'(S_{n}), & \text{if } k=n-\ve+1. \end{cases}
\]
Define a map $v' \colon Z_{k-1+\ve} \to \J_n$ as follows: for $z \in Z_{k-1+\ve}$,
\[
v'(z) \coloneqq \begin{cases}
\begin{pmatrix} z & & \\ &\mathrm{Id}_2& \\ && z^{\star}\end{pmatrix} \in \J_n, & \text{if } k=n-\ve, \\ \\ 
\begin{pmatrix} z & \\ & z^{\star}\end{pmatrix} \in \J_n, & \text{if } k=n-\ve+1. \end{cases}
\]
The Haar measures $dz$,
$d\alpha$, $dm$, $ds$, $dr''$, and $du'$ on the respective groups $Z_{k-1+\ve}$, $\J_{n,n-1}'$, $M_{n-1,2}$, $S_{k-1+\ve}$, $\R_{k-1+\ve}''$, and $\U_n'$ are normalized to be consistent with the conventions in Section~\ref{subsec:2-1}. In particular, they satisfy the following compatible integration formulas:
\[
dr'' = \begin{cases} 
d\alpha \, dm \, ds, & \text{if } k=n-\ve, \\ 
ds, & \text{if } k=n-\ve+1, 
\end{cases} \quad \text{and} \quad du' = dz \, dr''.
\]
From now on, we regard these groups as subgroups of $\J_n$.

From \eqref{a5}, it is straightforward to verify that for $z\in Z_{k-1+\ve}$:
\begin{itemize}
    \item $\begin{aligned}[t] 
    &\big(\omega_{\psi^{-1},W_n,V_k^{\ve}}(h',v(z))(\phi_1\otimes \phi_2)\big)(\w_0;\y_{c_1,c_{2-\ve}})= \big(\omega_{\psi^{-1},W_n,V_k^{\ve}}(z^{-1} h' ,\1)(\phi_1\otimes \phi_2)\big)(\w_0;\y_{c_1,c_{2-\ve}}),
    \end{aligned}$
    
    \item $\mu_{\ve}(v(z)) = \mu_{k-1+\ve}'(z).$
\end{itemize}

Assume that 
\[
\phi_{2,2} = \begin{cases} 
\phi_{3,1}\otimes \phi_{3,2} \in \mc{S}(Y' \otimes \langle e \rangle)(\A) \otimes \mc{S}(Y' \otimes \langle e' \rangle)(\A), & \text{if } \ve=0, \\ 
\phi_{3,1} \in \mc{S}(Y' \otimes \langle e \rangle)(\A), & \text{if } \ve=1. 
\end{cases}
\]

Then for $\alpha \in \J_{n,n-1}'(\A)$ and $m \in M_{n-1,2}(\A)$,
\begin{align*}
&\sum_{(c_1,c_{2-\ve}) \in F^{2-\ve}}  \big(\omega_{\psi^{-1},W_n,V_{n-\ve}^{\ve}}( \alpha m,\1)(\phi_1\otimes \phi_2)\big)(\w_0;\y_{c_1,c_{2-\ve}}) \\ 
&= \begin{cases} 
\phi_1(\w_0) \phi_{2,1}(\y_0) \theta_{\psi_{c}^{-1},W'}(\phi_{3,1})(\alpha m) \theta_{\psi_{cd},W'}'(\phi_{3,2})(\alpha), & \text{if } \ve=0, \\ 
\phi_1(\w_0) \phi_{2,1}(\y_0) \theta_{\psi_{d}^{-1},W'}(\phi_{3,1})(\alpha m), & \text{if } \ve=1. 
\end{cases}
\end{align*}
(For the definition of $\theta_{\psi_{\lambda},W'}$, see Remark \ref{1dim}.)

For any $h'\in \J_n(\A)$, there exists an integer $\mf{i}(h') \ge 1$ such that
\[
\omega_{\psi^{-1},W_n,V_{n-\ve}^{\ve}}(h',\1)(\phi_1 \otimes \phi_{2,1}\otimes \phi_{3,1}\otimes \phi_{3,2-\ve}) = \sum_{i=1}^{\mf{i}(h')}\phi_{1,h'}^{i}\otimes \phi_{2,h'}^{i} \otimes \phi_{3,h'}^{i}\otimes \phi_{4-\ve,h'}^{i},
\]
where for each $1\le i \le \mf{i}(h')$, $\phi_{1,h'}^{i}\otimes \phi_{2,h'}^{i} \otimes \phi_{3,h'}^{i}\otimes \phi_{4-\ve,h'}^{i}  \in $
\begin{equation*}
    \begin{split}
\begin{cases} 
\mc{S}(W_{n} \otimes X_{n-1+\ve}^*)\otimes\mc{S}(Y_{n-1}^* \otimes V^{\ve})\otimes\mc{S}(Y' \otimes e)\otimes \mc{S}(Y' \otimes e'), & \text{if } \ve=0, \\ 
\mc{S}(W_{n} \otimes X_{n-1+\ve}^*)\otimes\mc{S}(Y_{n-1}^* \otimes V^{\ve})\otimes\mc{S}(Y' \otimes e), & \text{if } \ve=1. \end{cases}
\end{split}
\end{equation*}
Note that $\mf{i}(\1)=1$ and $\phi_1 \otimes \phi_{2,1}\otimes \phi_{3,1}\otimes \phi_{3,2-\ve} = \phi_{1,\1}^{1}\otimes \phi_{2,\1}^{1} \otimes \phi_{3,\1}^{1}\otimes \phi_{4-\ve,\1}^{1}$. 
 
Define 
\[
\mc{F} \coloneqq \begin{cases} F^{2-\ve}, & \text{if } k=n-\ve, \\ \emptyset, & \text{if } k=n-\ve+1. \end{cases}
\]
When $\ve=1$, we regard $\theta_{\psi_{cd},W'}'(\phi_{4-\ve,h'}^{i})$ as the constant function $1$ on $ [\J_{n,n-1}']$.
 Combining these observations, we obtain
\begin{align*}
&W_{\psi}^{\ve}(f) = \sum_{(c_1,c_{2-\ve}) \in \mc{F}} \int_{\R_{k-1+\ve}''(F) \bs \J_{n}(\A)}\int_{[Z_{k-1+\ve}]}\mu_{\ve}^{-1}(v(z)) \\
&\quad \times \big(\omega_{\psi^{-1},W_n,V_k^{\ve}}(h',v(z))(\phi_1\otimes \phi_2)\big)(\w_0;\y_{c_1,c_{2-\ve}}) \varphi(h') \, dh' \, dz \\
&= \sum_{(c_1,c_{2-\ve}) \in \mc{F}} \int_{\R_{k-1+\ve}''(\A) \bs \J_{n}(\A)}\int_{[Z_{k-1+\ve}]}\mu_{\ve}^{-1}(v(z))\int_{[\R_{k-1+\ve}'']}\varphi(r''h') \\
&\quad \times \big(\omega_{\psi^{-1},W_n,V_k^{\ve}}(z^{-1}r''h',\1)(\phi_1\otimes \phi_2)\big)(\w_0;\y_{c_1,c_{2-\ve}})  \, dr''\, dz \, dh' \\
&= \sum_{(c_1,c_{2-\ve}) \in \mc{F}} \int_{\R_{k-1+\ve}''(\A) \bs \J_{n}(\A)}\int_{[Z_{k-1+\ve}]}\mu_{\ve}^{-1}(v(z))\int_{[\R_{k-1+\ve}'']}\varphi(z\wt{r''}z^{-1}h') \\
&\quad \times \big(\omega_{\psi^{-1},W_n,V_k^{\ve}}(\wt{r''}z^{-1}h',\1) \big)(\phi_1\otimes \phi_2)(\w_0;\y_{c_1,c_{2-\ve}}) \, d\wt{r''} \, dz \, dh' \\
&= \sum_{(c_1,c_{2-\ve}) \in \mc{F}} \int_{\R_{k-1+\ve}''(\A) \bs \J_{n}(\A)}\int_{[Z_{k-1+\ve}]}\mu_{\ve}^{-1}(v(z))\int_{[\R_{k-1+\ve}'']}\varphi(z\wt{r''}h') \\
&\quad \times \big(\omega_{\psi^{-1},W_n,V_k^{\ve}}(\wt{r''}h',\1)\big)(\phi_1\otimes \phi_2)(\w_0;\y_{c_1,c_{2-\ve}}) \, d\wt{r''} \, dz \, dh' \\
&= \begin{cases}
\begin{aligned}
    &\int_{\R_{n-1}''(\mathbb{A}) \bs \J_{n}(\mathbb{A})} \sum_{i=1}^{\mf{i}(h')} \phi_{1,h'}^{i}(\w_0) \phi_{2,h'}^{i}(\y_0) \\
    &\quad \times \Bigg(\int_{[\J_{n,n-1}']} \int_{[M_{n-1,2}]} \int_{[S_{n-1}]} \int_{[Z_{n-1}]} (\mu_{n-1}')^{-1}(z) \varphi(z\alpha msh') \\
    &\quad \times \theta_{\psi_{\lambda_{\ve}}^{-1},W'}(\phi_{3,h'}^{i})(\alpha ms) \theta_{\psi_{cd},W'}(\phi_{4-\ve,h'}^{i})(\alpha) \, d\alpha \, dm \, ds \, dz \Bigg) \, dh', 
\end{aligned} & \text{if } k = n - \ve, \\[6ex]
\begin{aligned}
    &\int_{\R_{n}''(\mathbb{A}) \bs \J_{n}(\mathbb{A})} \int_{[S_{n}]} \int_{[Z_{n}]} (\mu_{n}' \psi_{\lambda_{\ve}})^{-1}(zs) \varphi(zsh') \\
    &\quad \times \big(\omega_{\psi^{-1},W_n,V_{n-\ve+1}^{\ve}}(h',\mathbf{1})\big)(\phi_1 \otimes \phi_2)(\w_0;\y_{0}) \, ds \, dz \, dh',
\end{aligned} & \text{if } k = n - \ve + 1.
\end{cases} \\[10ex]
&= \begin{cases}
\begin{aligned}
    &\int_{\R_{n-1}''(\mathbb{A}) \bs \J_{n}(\mathbb{A})} \sum_{i=1}^{\mf{i}(h')} \phi_{1,h'}^{i}(\w_0) \phi_{2,h'}^{i}(\y_0) \\
    &\quad \times \mc{FJ}_{n-1,\psi}^{\lambda_{\ve}}\big(\mf{R}(h')(\varphi),\theta_{\psi_{cd},W'}'(\phi_{4-\ve,h'}^{i}),\theta_{\psi_{\lambda_{\ve}}^{-1},W'}(\phi_{3,h'}^{i})\big) \, dh', 
\end{aligned} & \text{if } k=n-\ve, \\[6ex]
\begin{aligned}
    &\int_{\R_{n}''(\mathbb{A}) \bs \J_{n}(\mathbb{A})} \big(\omega_{\psi^{-1},W_n,V_{n-\ve+1}^{\ve}}(h',\mathbf{1})(\phi_1\otimes \phi_{2})\big)(\w_0;\y_{0}) \\
    &\quad \times \mc{FJ}_{n,\psi}^{\lambda_{\ve}}\big(\mf{R}'(h')(\varphi)\big) \, dh',
\end{aligned} & \text{if } k=n-\ve+1.
\end{cases}
\end{align*}

In the second equality, we unfold the integral and apply the identity $\omega(h', v(z)) = \omega(z^{-1}h', \1)$. In the third equality, we use the change of variables $r''=z\wt{r''}z^{-1}$, which is valid because $v'(Z_{k-1+\ve})$ normalizes $\R_{k-1+\ve}''$. In the fourth equality, we shift the variable $h' \mapsto zh'$, leaving the Haar measure invariant. In the fifth equality, we expand the Weil representation into a finite sum, and in the final equality, we recognize the resulting inner integrals as the exact definitions of the Fourier-Jacobi periods.

Therefore, if $W_{\psi}^{\ve}(f) \ne 0$ for some $f \in \Theta_{\psi^{-1},W_n,V_k^{\ve}}(\sigma)$, it is clear that $\mc{FJ}_{k-1+\ve_{\sigma},\psi}^{\lambda_{\ve}}$ is not identically zero. Conversely, suppose that $\mc{FJ}_{k-1+\ve_{\sigma},\psi}^{\lambda_{\ve}}$ is not identically zero. Then there is some $\varphi' \in \sigma$ such that $\mc{FJ}_{k-1+\ve_{\sigma},\psi}^{\lambda_{\ve}}(\varphi')\ne0$. Choose a finite set $S$ containing all archimedean places such that for $v \notin S$, $(\w_{0,v},\y_{0,v})\in \WW_{0}(\OO_v) \times (V^{\ve} \otimes \wt{Y}_n)(\OO_v)$. 

For each $v \notin S$, we simply set the local Schwartz functions $\phi_{1,v}$ and $\phi_{2,v}$ to be the characteristic functions of the standard $\mc{O}_v$-lattices $(X_{k-1+\ve}^* \otimes W_n)(\OO_v)$ and $V^{\ve} \otimes \wt{Y}_n(\OO_v)$, respectively. Since $(\w_{0,v},\y_{0,v})$ lies in these lattices, the evaluation of $\phi_{1,v} \otimes \phi_{2,v}$ at this point is exactly $1$, ensuring that the infinite product is well-defined and non-vanishing.

For the places $v \in S$, we choose $\phi_{1,v}$ and $\phi_{2,v}$ to be highly concentrated around $\w_{0,v}$ and $\y_{0,v}$. Specifically, for non-archimedean $v \in S$, we take them to be the normalized characteristic functions of arbitrarily small compact open neighborhoods of $\w_{0,v}$ and $\y_{0,v}$. For archimedean $v \in S$, we choose them to be non-negative smooth bump functions tightly supported within an infinitesimal neighborhood of $\w_{0,v}$ and $\y_{0,v}$.

We then define the global Schwartz functions as $\phi_1 = \bigotimes_v \phi_{1,v}$ and $\phi_2 = \bigotimes_v \phi_{2,v}$. By shrinking the supports of the local functions strictly at the places $v \in S$, the global function $\phi = \phi_1 \otimes \phi_2$ essentially vanishes unless the shifted argument is extremely close to $(\w_0, \y_0)$. Recall that $R_{k-1+\ve}''$ is precisely the stabilizer of $(\w_0, \y_0)$. Thus, this choice of $\phi$ localizes the integration variable $h'$ in the outer integral over $R_{k-1+\ve}''(\A) \bs J_n(\A)$ to an arbitrarily small neighborhood of the identity coset. 

As the support shrinks, this localization forces the entire global integral to converge to a nonzero multiple of the inner period $\mc{FJ}_{k-1+\ve_{\sigma},\psi}^{\lambda_{\ve}}(\varphi') \ne 0$. Therefore, we can guarantee the existence of an appropriate $\phi_0 = \phi_1 \otimes \phi_2 \in \Omega_{\psi^{-1},W_n,V_{k}^{\ve}}$ such that $W_{\psi}^{\ve}\big(\Theta_{\psi^{-1},W_n,V_k^{\ve}}(\phi_0,\varphi')\big) \ne 0$. This completes the proof.\end{proof}
From Theorem~\ref{b10}, we deduce the following two corollaries. Since their proofs closely parallel those of Corollaries~\ref{a2} and \ref{a3}, we omit the details.

\begin{cor}\label{e2} 
Let $\sigma$ be an irreducible cuspidal representation of $\wt{\J}_n(\A)$. Assume that $\sigma_v$ is $(\psi_v,\lambda_{\ve_{\sigma}})$-generic for some non-archimedean place $v$. Then the non-vanishing of $\mc{FJ}_{n-1,\psi}^{\lambda_{\ve_{\sigma}}}$ on $\sigma \boxtimes \tau_{\sigma} \boxtimes \nu_{\psi^{-1},W_{n,n-1}'}^{\lambda_{\ve_{\sigma}}}$ is equivalent to the condition that $\Theta_{\psi^{-1},W_n,V_{n-\ve_{\sigma}}^{\ve_{\sigma}}}(\sigma)$ is nonzero, cuspidal, and $\mu_{\ve_{\sigma}}$-generic.
\end{cor}

\begin{cor}\label{e3} 
Let $\sigma$ be an irreducible cuspidal $(\psi,\lambda_{\ve_{\sigma}})$-generic automorphic representation of $\wt{\J}_n(\A)$. Then $\Theta_{\psi^{-1},W_n,V_{n+1-\ve_{\sigma}}^{\ve_{\sigma}}}(\sigma)$ is nonzero and $\mu_{\ve_{\sigma}}$-generic. Furthermore, if $\Theta_{\psi^{-1},W_n,V_{n-\ve_{\sigma}}^{\ve_{\sigma}}}(\sigma) = 0$, then $\Theta_{\psi^{-1},W_n,V_{n+1-\ve_{\sigma}}^{\ve_{\sigma}}}(\sigma)$ is cuspidal and irreducible.
\end{cor}

\begin{rem}
In the case where $\H_n^0 = \SO_{2n}$ is of type $(1,1)$, it is stated in \cite[Theorem~2.2]{GRS97} that if $\mc{P}_{\psi^{-1}}$ is nonzero on $\sigma$, then $\sigma$ must be generic. However, no proof is provided for this claim, and we suspect that it may not hold in general. To ensure the validity of the analogous assertion within our context, we have explicitly imposed a local genericity condition in Corollary~\ref{e2}.
\end{rem}
\bb

\section{The relation of $L_{\psi}(s,\sigma\times \chi_{\sigma})$ and the Fourier-Jacobi period $\mc{FJ}_{n-1,\psi}^{1}$}\label{sec:7}
Throughout this section, we assume that the quadratic space $V_n^0$ is of type $(d, 1)$ for some $d \in F^\times$. For the quadratic space $V_n^1$, we assume that its discriminant is trivial (i.e., $\mathrm{disc}(V_n^1) = 1$). For convenience, we let $W'$ denote the subspace $W_{n,n-1}$ of $W_n$ spanned by $\{f_n, f_n^*\}$, and we set $\J' \coloneqq \J_{n,n-1}'$. We establish the following results.
\begin{thm}\label{prop2}
Let $\sigma$ be an irreducible $(\psi,1)$-generic cuspidal representation of $\wt{\J}_n(\A)$. Let $S$ be a finite set of places of $F$ containing all archimedean places of $F$, such that for every $v \notin S$, both $\sigma_v$ and $\psi_{v}$ are unramified. Then the following conditions are equivalent:
\begin{enumerate}
    \item $\begin{cases} L_{\psi}(s,\sigma \times \chi_{d} ) \text{ has a pole at } s=1, & \text{if } \sigma \text{ is non-genuine}, \\ L_{\psi}(s,\sigma ) \text{ has a pole at } s=\frac{3}{2}, & \text{if } \sigma \text{ is genuine}. \end{cases}$
    \item $\begin{cases} L_{\psi}^S(s,\sigma \times \chi_{d} ) \text{ has a pole at } s=1, & \text{if } \sigma \text{ is non-genuine}, \\ L_{\psi}^S(s,\sigma) \text{ has a pole at } s=\frac{3}{2}, & \text{if } \sigma \text{ is genuine}. \end{cases}$
    \item $\mc{FJ}_{n-1,\psi}^{1}$ is nonzero on 
    $\begin{cases} \sigma \boxtimes \omega_{\psi_d,W'} \boxtimes \nu_{\psi^{-1},W'}, & \text{if } \sigma \text{ is non-genuine}, \\ \sigma\boxtimes \nu_{\psi^{-1},W'}, & \text{if } \sigma \text{ is genuine}. \end{cases}$
    \item $\Theta_{\psi^{-1},W_n,V_{n-\ve_{\sigma}}^{\ve_{\sigma}}}(\sigma)$ is nonzero and $\mu_{\ve_{\sigma}}$-generic.
    \item $\Theta_{\psi^{-1},W_n,V_{n-\ve_{\sigma}}^{\ve_{\sigma}}}(\sigma)$ is nonzero.
\end{enumerate}
\end{thm}

\begin{rem}
In the case where $\sigma$ is non-genuine (i.e., $\ve_{\sigma} = 0$) and $V_n^0$ is of type $(1, 1)$, if the period $\mc{FJ}_{n-1, \psi}^{1}$ in (iii) is replaced by the period $\mc{P}_{\psi^{-1}}$ appearing in \cite[p.\,85]{GRS97}, the equivalence of (ii), (iii), and (iv) in Theorem~\ref{prop2} is established by \cite[Theorems~2.1 and 2.2]{GRS97}. Furthermore, when $\sigma$ is strongly generic, Theorem~\ref{prop2} was proved in \cite[Proposition~2.7]{GRS97}. Thus, Theorem~\ref{prop2} provides a significant extension of \cite[Proposition~2.7]{GRS97}, broadening the scope from \textit{strongly generic} representations of $\wt{\J}_n$ to arbitrary \textit{generic} representations. It is also worth mentioning that when $\sigma$ is genuine (i.e., $\ve_{\sigma} = 1$), the implication (v) $\Rightarrow$ (iv) was established in \cite[Theorem~4.2]{JS07} via a different methodology.
\end{rem}
Let $\B'=\T'\U'$ be the Borel subgroup of $\J'$, where $\T'$ is the split maximal torus of $\B'$ and $\U'$ is the unipotent radical of $\B'$. 
For $s \in \CC$, let the character of $\B'$ be given by
\[
|\cdot|^{s} \colon \begin{pmatrix} a & \\ & a^{-1}\end{pmatrix} \mapsto |a|^{s}.
\] 
Then $|\cdot|^{2}$ corresponds to the modulus character $\delta_{\B'}$ of $\B'$.
 
Define $\wt{\B}'(\A) = \wt{\T}'(\A)\U'(\A) = (pr)^{-1}(\B'(\A))$. By composing a character $\chi$ of $\B'(\A)$ with $pr$, we may view $\chi$ as a character of $\wt{\B}'(\A)$.

Let $I'(s,\chi)$ be the unnormalized induced representation $\ind_{\B'(\A)}^{\J'(\A)}(\chi |\cdot|^{2s})$ of $\J'(\A)$ induced from the character $\chi |\cdot|^{2s}$ of $\B'(\A)$. By definition, $I'(s,\chi)(\A)$ consists of all smooth functions $f_s \colon \J'(\A) \to \CC$ satisfying
\[
f_s(bg) = \chi(t) |t|^{2s} f_s(g), \quad b = tu \in \B'(\A) = \T'(\A)\U'(\A).
\]

We also define $\wt{I}'(s,\chi)$ as $\ind_{\wt{\B}'(\A)}^{\wt{\J}'(\A)}(\gamma_{\psi} \chi |\cdot|^{2s})$ the unnormalized induced representation of $\wt{\J}'(\A)$ induced from the character $\chi |\cdot|^{2s}$ of $\wt{\B}'(\A)$. (Recall that $\gamma_{\psi}$ is the Weil factor associated to $\psi$.) By definition, $\wt{I}'(s,\chi)(\A)$ consists of all smooth functions $\wt{f}_s \colon \wt{\J}'(\A) \to \CC$ satisfying
\[
\wt{f}_s(\wt{b} \wt{g}) = \ve' \gamma_{\psi}(t) \chi(t) |t|^{2s} \wt{f}_s(\wt{g}), \quad \wt{b} = (tu,\ve') \in \B'(\A)\times \mu_2, \ (t,\ve') \in \T'(\A) \times \mu_2.
\]

The Eisenstein series attached to holomorphic sections $f_s \in I'(s,\chi)(\A)$ and $\wt{f}_s \in \wt{I}'(s,\chi)(\A)$ are defined, respectively, by
\begin{align*}
E(f_s,g) &= \sum_{\alpha \in \B'(F)\bs \J'(F)}f_s(\alpha g), \quad \text{for } g \in \J'(\A), \\
E(\wt{f}_s,\wt{g}) &= \sum_{\wt{\alpha} \in \wt{\B}'(F)\bs \wt{\J}'(F)}\wt{f}_s(\wt{\alpha} \wt{g}), \quad \text{for } \wt{g} \in \wt{\J}'(\A).
\end{align*}
We define
\[
w' = \begin{pmatrix}
 & 1 &  &     \\
\mathrm{Id}_{n-1} &  &  &     \\
 &  &    & \mathrm{Id}_{n-1}  \\
   &  & 1 &   
\end{pmatrix} \in M_{2n \times 2n}(F),
\]
and define the conjugation map $j$ by
\begin{align*}
j(\wt{h}') \coloneqq \begin{cases} w'\wt{h}'(w')^{-1}, \quad  &\text{if } \wt{h}' \in \J_n(\A), \\ (w'h'(w')^{-1}, \ve'), \quad  &\text{if } \wt{h}'=(h',\ve') \in \wt{\J}_n(\A).\end{cases}
\end{align*}

Let $\sigma$ be an irreducible $(\psi,1)$-generic cuspidal automorphic representation of $\wt{\J}_n(\A)$. Then for any $\varphi \in \sigma$, the function $\wt{h}' \mapsto \varphi(j(\wt{h}'))$ belongs to $\sigma$. Indeed, since $w' \in \J_n(F)$ and $\varphi$ is left $\wt{\J}_n(F)$-invariant, we have 
\[
\varphi(j(\wt{h}')) = \varphi(w' \wt{h}' (w')^{-1}) = \varphi(\wt{h}' (w')^{-1}).
\]
Because $\sigma$ is an automorphic representation, it is stable under right translation, ensuring that this function remains in $\sigma$. We denote this function by $\varphi_{w'}\in \sigma$, so that $\varphi_{w'}(\wt{h}') \coloneqq \varphi(j(\wt{h}')).$

For $f_s \in I'(s,\chi)(\A)$ and $\wt{f}_s \in \wt{I}'(s,\chi)(\A)$, the global integrals $\J_{\psi}^{\chi}(\cdot, \cdot,\wt{f}_s)$ and $\wt{\J}_{\psi}^{\chi}(\cdot,\cdot,f_s)$ that represent the partial $L$-functions of $\sigma$ are defined as follows.

If $\sigma$ is non-genuine, then for $\varphi \in \sigma$, $\phi \in \Omega_{\psi^{-1},W'}$, and $\wt{f}_s \in \wt{I}'(s,\chi)(\A)$, we define
\[
\J_{\psi}^{\chi}(\varphi,\phi,\wt{f}_s) \coloneqq \int_{[\J']} \int_{[\U_{n,n-1}']} \varphi(j(u'g)) \theta_{\psi^{-1},W'}(\phi)(u'\wt{g}) E(\wt{f}_s,\wt{g}) (\mu_{n-1}')^{-1}(u') \, du' \, dg.
\]

If $\sigma$ is genuine, then for $\varphi \in \sigma$, $\phi \in \Omega_{\psi^{-1},W'}$, and $f_s \in I'(s,\chi)(\A)$, we define
\[
\wt{\J}_{\psi}^{\chi}(\varphi,\phi,f_s) \coloneqq \int_{[\J']} \int_{[\U_{n,n-1}']} \varphi(j(u'\wt{g})) \theta_{\psi^{-1},W'}(\phi)(u'\wt{g}) E(f_s,g) (\mu_{n-1}')^{-1}(u') \, du' \, dg.
\]
(Here, we regard $\J'$ as a subgroup of $\J_n$ and $\wt{g}$ is any element in $(pr)^{-1}(g) \subset \wt{\J}_n(\A)$.  Since $\theta_{\psi^{-1},W'}(\phi)(u'\wt{g}) E(\wt{f}_s,\wt{g})$ and $\varphi(j(u'\wt{g})) \theta_{\psi^{-1},W'}(\phi)(u'\wt{g})$ are non-genuine functions in their respective cases, the above integrals do not depend on the choice of $\wt{g}$.)

The rapid decay property of $\varphi$ ensures that the integrals $\J_{\psi}^{\chi}$ and $\wt{\J}_{\psi}^{\chi}$ are absolutely convergent for all $s \in \CC$, except at the poles of $E(f_s, g)$ and $E(\wt{f}_s, \wt{g})$, respectively.

There is a basic identity, analogous to Proposition~\ref{p1}, that relates $\J_{\psi}^{\chi}$ and $\wt{\J}_{\psi}^{\chi}$ to the global zeta integrals $Z_{\psi,\chi}'$ and $\wt{Z}_{\psi,\chi}'$. To introduce these zeta integrals, we first establish some notation.

For $x \in \mathbb{G}_a$, define
\[
\bar{x} = \begin{pmatrix}
\mathrm{Id}_{n-2}   & 0 & 0 & 0 & 0 & 0 \\ 
 & 1 & x & 0 & 0 & 0 \\  
 & & 1 & 0 & 0 & 0 \\  
 & & 0 & 1 & x^* & 0 \\ 
 & & & & 1 & 0 \\ 
 & & & & & \mathrm{Id}_{n-2} 
\end{pmatrix} \in \U_{n,n-1}'.
\]

Let $R \subset \U_{n,n-1}'$ be the unipotent subgroup consisting of all matrices of the form
\[
R = \left\{ \begin{pmatrix} \mathrm{Id}_{n-1} & r & & \\ & 1 & & \\ & & 1 & r^* \\ & & & \mathrm{Id}_{n-1} \end{pmatrix} \;\middle|\; r\in M_{n-1,1} \text{ such that the bottom row of } r \text{ is }\textbf{0} \right\}.
\]
If $\sigma$ is non-genuine, for $\varphi \in \sigma$, $\phi \in \mc{S}(Y'(\mathbb{A}))$, and $\wt{f}_s \in \wt{I}'(s,\chi)(\mathbb{A})$, we set
\begin{equation*}
\begin{split}
&Z_{\psi,\chi}'(\varphi,\phi,\wt{f}_s) \coloneqq \int_{\U'(\mathbb{A}) \bs \J'(\mathbb{A})} \int_{R(\mathbb{A})} \int_{\mathbb{G}_a(\mathbb{A})} W_{\varphi}^{1}\big(j(r\bar{x}g)\big) \big(\Omega_{\psi^{-1},W'}(\wt{g})(\phi)\big)(x) \wt{f}_s(\wt{g}) \, dx \, dr \, dg.
\end{split}
\end{equation*}

If $\sigma$ is genuine, for $\varphi \in \sigma$, $\phi \in \mc{S}(Y'(\mathbb{A}))$, and $f_s \in I'(s,\chi)(\mathbb{A})$, we set
\begin{equation*}
\begin{split}
&\wt{Z}_{\psi,\chi}'(\varphi,\phi,f_s) \coloneqq \int_{\U'(\mathbb{A}) \bs \J'(\mathbb{A})} \int_{R(\mathbb{A})} \int_{\mathbb{G}_a(\mathbb{A})} W_{\varphi}^{1}\big(j(r\bar{x}\wt{g})\big) \big(\Omega_{\psi^{-1},W'}(\wt{g})(\phi)\big)(x) f_s(g) \, dx \, dr \, dg.
\end{split}
\end{equation*}

The following theorem relates $\J_{\psi}^{\chi}$ and $\wt{\J}_{\psi}^{\chi}$ to the global zeta integrals $Z_{\psi,\chi}'$ and $\wt{Z}_{\psi,\chi}'$.

\begin{thm}[{\cite[Theorems~2.1 and 2.3]{GRS98}}]\label{secid} 
Let $\sigma$ be an irreducible $(\psi,1)$-generic cuspidal automorphic representation of $\wt{\J}_n(\A)$. For $\varphi\in \sigma$, $\phi \in \mc{S}(Y'(\A))$, and holomorphic sections $f_s\in I'(s,\chi)(\A)$ and $\wt{f}_s\in \wt{I}'(s,\chi)(\A)$, the integrals $Z_{\psi,\chi}'(\varphi,\phi,\wt{f}_s)$ and $\wt{Z}_{\psi,\chi}'(\varphi,\phi,f_s)$ are absolutely convergent except at the poles of $E(f_s,g)$ and $E(\wt{f}_s,\wt{g})$. For $\Re(s)\gg 0$, we have
\begin{align*}
\J_{\psi}^{\chi}(\varphi,\phi,\wt{f}_s) &= Z_{\psi,\chi}'(\varphi,\phi,\wt{f}_s), \quad &&\text{if } \sigma \text{ is non-genuine,} \\ 
\wt{\J}_{\psi}^{\chi}(\varphi,\phi,f_s) &= \wt{Z}_{\psi,\chi}'(\varphi,\phi,f_s), \quad &&\text{if } \sigma \text{ is genuine.}
\end{align*}
\end{thm}

Motivated by the definitions of the global zeta integrals, we define the local zeta integrals as their local counterparts. For an arbitrary place $v$ of $F$, choose a $(\psi_v,1)$-generic $\sigma \in \Irr(\wt{\J}_n(F_v))$. Then for any $W \in W_{\psi_v}^{1}(\sigma)$, $\phi \in \mc{S}(Y'(F_v))$, and holomorphic sections $f_{s,v} \in I'(s,\chi)(F_v)$ and $\wt{f}_{s,v} \in \wt{I}'(s,\chi)(F_v)$, the local integrals $Z_{\psi_v,\chi_v}'(W,\phi,\wt{f}_{s,v})$ and $\wt{Z}_{\psi_v,\chi_v}'(W,\phi,f_{s,v})$ are defined as
\begin{align*}
&Z_{\psi_v,\chi_v}'(W,\phi,\wt{f}_{s,v}) \coloneqq \int_{\U'(F_v) \bs \J'(F_v)} \int_{R(F_v)} \int_{\mathbb{G}_a(F_v)} W\big(j(r\bar{x}g)\big) \big(\Omega_{\psi_v^{-1},W'} (\wt{g})(\phi)\big)(x) \wt{f}_{s,v}(\wt{g}) \, dx \, dr \, dg, \\
&\quad \text{if } \sigma \text{ is non-genuine,} \\[15pt]
&\wt{Z}_{\psi_v,\chi_v}'(W,\phi,f_{s,v})\coloneqq \int_{\U'(F_v) \bs \J'(F_v)} \int_{R(F_v)} \int_{\mathbb{G}_a(F_v)} W\big(\wt{\jmath}(r\bar{x}\wt{g})\big) \big(\Omega_{\psi_v^{-1},W'} (\wt{g})(\phi)\big)(x) f_{s,v}(g) \, dx \, dr \, dg, \\
&\quad \text{if } \sigma \text{ is genuine.}
\end{align*}

\begin{prop}[{\cite[Lemmas~3.4 and 3.5]{GRS98}}]
Let $v$ be a place of $F$. Suppose that $\sigma \in \Irr(\wt{\J}_n(F_v))$ is $(\psi_v,1)$-generic. Then for all $W \in W_{\psi_v}^{1}(\sigma)$ and $\phi \in \mc{S}(Y'(F_v))$, and for any holomorphic sections $f_{s,v} \in I'(s,\chi)(F_v)$ and $\wt{f}_{s,v} \in \wt{I}'(s,\chi)(F_v)$, the integrals $Z_{\psi_v,\chi_v}'(W,\phi,\wt{f}_{s,v})$ and $\wt{Z}_{\psi_v,\chi_v}'(W,\phi,f_{s,v})$ are absolutely convergent for $\Re(s) \gg 0$. Furthermore, they admit meromorphic continuations to the whole complex plane.
\end{prop}

By Theorem~\ref{secid}, for pure tensors $\varphi=\otimes_v \varphi_v\in \sigma$ and $\phi=\otimes_v \phi_v \in \mc{S}(Y'(\A))$, and holomorphic decomposable sections $f_s=\otimes_v f_{s,v}\in I'(s,\chi)(\A)$ and $\wt{f}_s=\otimes_v \wt{f}_{s,v}\in \wt{I}'(s,\chi)(\A)$, we have the Euler product expansions
\begin{align*}
\J_{\psi}^{\chi}(\varphi,\phi,\wt{f}_s) &= \prod_v Z_{\psi_v,\chi_v}'(W_{\psi_v}^{1}(\varphi_v),\phi_v,\wt{f}_{s,v}), \quad \text{for } \Re(s)\gg 0, \\ 
\wt{\J}_{\psi}^{\chi}(\varphi,\phi,f_s) &= \prod_v \wt{Z}_{\psi_v,\chi_v}'(W_{\psi_v}^{1}(\varphi_v),\phi_v,f_{s,v}), \quad \text{for } \Re(s)\gg 0.
\end{align*}

For a non-archimedean place $v$ of $F$, we say that $\phi_v \in \mc{S}(Y'(F_v))$ is unramified if 
\[
\phi_v(x) = \begin{cases} 1, & \text{if } x \in Y'(\mc{O}_v), \\ 0, & \text{otherwise}. \end{cases}
\] 
The following proposition provides the computation of the local zeta integrals with unramified data for $\Re(s) \gg 0$.

\begin{prop}[{\cite[Theorems~3.1 and 3.2]{GRS98}}]\label{unram} 
Let $v$ be a non-archimedean place of $F$. Suppose that $\sigma \in \Irr(\wt{\J}_n(F_v))$ is $(\psi_v,1)$-generic and unramified. Choose $W^0 \in W_{\psi_v}^{1}(\sigma)$, and holomorphic sections $f_{s,v}^0 \in I'(s,\chi)(F_v)$ and $\wt{f}_{s,v}^0 \in \wt{I}'(s,\chi)(F_v)$ such that $W^0(k')=1$ and $f_{s,v}^0(k')=\wt{f}_{s,v}^0(k')=1$ for all $k' \in \K_v'$. Suppose that $\phi^0 \in \mc{S}(Y'(F_v))$ is unramified. Then for $\Re(s) \gg 0$, we have
\begin{align*}
Z_{\psi_v,\chi_v}'(W^0,\phi^0,\wt{f}_{s,v}^0) &= \frac{L_{\psi_v}(2s-\frac{1}{2},\sigma\times \chi)}{\zeta_{v}(4s-1)}, \quad &&\text{if } \sigma \text{ is non-genuine,} \\ 
\wt{Z}_{\psi_v,\chi_v}'(W^0,\phi^0,f_{s,v}^0) &= \frac{L_{\psi_v}(2s-\frac{1}{2},\sigma\times \chi)}{\zeta_{v}(2s)}, \quad &&\text{if } \sigma \text{ is genuine.} 
\end{align*}
\end{prop}

We also require a non-vanishing result for the local zeta integrals. 

\begin{prop}[{\cite[Propositions~3.6 and 3.7]{GRS98}}]\label{nonv}
Let $v$ be a place of $F$ and $\lambda \in F_v^{\times}$. Suppose that $\sigma \in \Irr(\wt{\J}_n(F_v))$ is $(\psi_v,1)$-generic. Given $s_0\in \CC$, there exist $W_0 \in W_{\psi_v}^{1}(\sigma)$, $\phi_0 \in \mc{S}(Y'(F_v))$, and $\K_v'$-finite sections $f_{s_0,v}\in I'(s,\chi)(F_v)$ and $\wt{f}_{s_0,v}\in \wt{I}'(s,\chi)(F_v)$ such that $Z_{\psi_v,\chi_v}'(W_0,\phi_0,\wt{f}_{s_0,v}) \neq 0$ and $\wt{Z}_{\psi_v,\chi_v}'(W_0,\phi_0,f_{s_0,v}) \neq 0$. 
\end{prop}
Now we are ready to prove Theorem~\ref{prop2}.

\begin{customproof0}
We first prove the implication (i) $\Rightarrow$ (ii). Define 
\[
s_{\sigma} \coloneqq \begin{cases} 1, & \text{if } \sigma \text{ is non-genuine,} \\ \frac{3}{2}, & \text{if } \sigma \text{ is genuine,} \end{cases} \quad \chi_{\sigma} \coloneqq \begin{cases} \chi_d, & \text{if } \sigma \text{ is non-genuine,} \\ \II, & \text{if } \sigma \text{ is genuine.} \end{cases}
\] 
Suppose, for the sake of contradiction, that $L_{\psi}^S(s,\sigma\times \chi_{\sigma})$ is holomorphic at $s=s_{\sigma}$. Since 
\[
L_{\psi}(s,\sigma\times \chi_{\sigma}) = L_{\psi}^S(s,\sigma\times \chi_{\sigma}) \prod_{v \in S} L_{\psi}(s,\sigma_v\times \chi_{\sigma,v}),
\] 
the pole of $L_{\psi}(s,\sigma\times \chi_{\sigma})$ must originate from $L_{\psi}(s,\sigma_v\times \chi_{\sigma,v})$ for some place $v \in S$. We can uniquely express $\sigma_v$ as the Langlands quotient $L(\tau_{1,v} |\cdot|^{e_1}, \cdots , \tau_{k,v} |\cdot|^{e_k} , \sigma_{0,v})$. 

When $\wt{\J}_n=\J_n$ (i.e., $\sigma$ is non-genuine), then $e_1 < \frac{1}{2}$ by the non-trivial Ramanujan bound \cite[Corollary~10.1]{CKPS04}; consequently, the product \[\prod_{i=1}^k L_{\mathrm{GJ}}(s+e_i,\tau_{i,v}) L_{\mathrm{GJ}}(s-e_i,\tau_{i,v}^{\vee})\] is holomorphic at $s=s_0$ by \cite[Remark~3.2.4]{Jac79}. By \cite[Theorem~1]{Yam14},
\[
L(s,\sigma_v\times \chi_{\sigma,v}) = L(s,\sigma_{0,v}\times \chi_{\sigma,v}) \prod_{i=1}^k L_{\mathrm{GJ}}(s+e_i,\sigma_{i,v}\times \chi_{\sigma,v}) L_{\mathrm{GJ}}(s-e_i,\sigma_{i,v}^{\vee}\times \chi_{\sigma,v}),
\]
where $L_{\mathrm{GJ}}$ denotes the local $L$-factor of general linear groups defined by Godement and Jacquet (\cite{GJ72}). Since $e_1 < \frac{1}{2}$, the product $\prod_{i=1}^k L_{\mathrm{GJ}}(s+e_i,\tau_{i,v}\times \chi_{\sigma,v}) L_{\mathrm{GJ}}(s-e_i,\tau_{i,v}^{\vee}\times \chi_{\sigma,v})$ is holomorphic at $s=1$ by \cite[Remark~3.2.4]{Jac79}, and $L(s,\sigma_{0,v}\times \chi_{\sigma,v})$ is holomorphic at $s=1$ by \cite[Lemma~7.2]{Yam14}. Therefore, $L(s,\sigma_v\times \chi_{\sigma,v})$ must be holomorphic at $s=1$, which leads to a contradiction.

Similarly, when $\wt{\J}_n$ is a metaplectic group (i.e., $\sigma$ is genuine), using functoriality from metaplectic groups to general linear groups (\cite[Theorem~11.2]{GRS11}), we  obtain $e_1\le \frac{1}{2}$. Proceeding with arguments analogous to those above, we again derive a contradiction. This establishes the implication (i) $\Rightarrow$ (ii). 

Next, we prove the implication (ii) $\Rightarrow$ (iii). Assume that (ii) holds. Then by Propositions~\ref{unram} and \ref{nonv}, there exist $\varphi \in \sigma$, $\phi \in \mc{S}(Y'(\A))$, and holomorphic sections $f_s\in I'(s,\II)(\A)$ and $\wt{f}_s \in \wt{I}'(s,\chi_{d})(\A)$ such that 
\begin{align*}
&\J_{\psi}^{\chi_d}(\varphi,\phi,\wt{f}_s) \text{ has a pole at } s=\frac{3}{4}, \quad &&\text{if } \sigma \text{ is non-genuine,} \\ 
&\wt{\J}_{\psi}^{\II}(\varphi,\phi,f_s) \text{ has a pole at } s=1, \quad &&\text{if } \sigma \text{ is genuine.}
\end{align*}
The poles of these global zeta integrals necessarily arise from the corresponding Eisenstein series involved in the integration. According to \cite[Proposition~6.2, Theorem~7.1]{GQT14}, there exist a $1$-dimensional orthogonal space $\wt{V}$ with discriminant $d$, nonzero constants $c_1, c_2 \in \CC$, and a Schwartz function $\phi' \in \mc{S}(\wt{V} \otimes Y')(\A)$ such that
\begin{align}
\label{res1}\mathrm{Res}_{s=3/4} E(\wt{f}_s, \wt{g}) &= c_1 \int_{[\mu_2]} \theta_{\psi, \wt{V}, W'}(\phi')((\mathbf{t} \cdot \mathbf{e}), \wt{g}) \, d\mathbf{t}, && \text{if } \sigma \text{ is non-genuine,} \\ \label{res2}
\mathrm{Res}_{s=1} E(f_s, g) &= c_2, && \text{if } \sigma \text{ is genuine.}
\end{align}
(Specifically, to derive \eqref{res1}, we apply the parameters $(m,n)=(3,1)$ for \cite[Proposition~6.2]{GQT14} and $(m,n)=(1,1)$ for \cite[Theorem~7.1]{GQT14}. Similarly, to derive \eqref{res2}, we substitute $(m,n)=(4,1)$ into \cite[Proposition~6.2]{GQT14} and $(m,n)=(0,1)$ into \cite[Theorem~7.1]{GQT14}. We note that the induction used in \cite{GQT14} is normalized.)

In view of Remark~\ref{1dim}, we shall identify $\mc{S}(\wt{V} \otimes Y')$ and $\theta_{\psi, \wt{V}, W'}'(\phi')$ with $\mc{S}(Y')$ and $\theta_{\psi_d, W'}'(\phi')$, respectively.

If $\sigma$ is non-genuine, by applying an appropriate right translation action of $\O(\wt{V})(\mathbb{A})$, there exists $\phi'' \in \mc{S}(Y')(\mathbb{A})$ such that 
\begin{equation*}
\begin{split}
&\mc{FJ}_{n-1,\psi}^{1}(\varphi_{w'}\otimes \theta_{\psi_d,W'}'(\phi'')\otimes \phi) \\
&= \int_{[\J']} \int_{[\U_{n,n-1}']} \varphi(j(u'h')) \theta_{\psi^{-1},W'}(\phi)(u'\wt{h}')\theta_{\psi_d,W'}'(\phi'')(\wt{h}')(\mu_{n-1}')^{-1}(u') \, du' \, dh' \neq 0.
\end{split}
\end{equation*}

If $\sigma$ is genuine, we obtain
\begin{equation*}
\begin{split}
&\mc{FJ}_{n-1,\psi}^{1}(\varphi_{w'}\otimes \phi)= \int_{[\J']} \int_{[\U_{n,n-1}']} \varphi(j(u'\wt{h}')) \theta_{\psi^{-1},W'}(\phi)(u'\wt{h}')  (\mu_{n-1}')^{-1}(u') \, du' \, dh' \neq 0.\quad \quad \quad \
\end{split}
\end{equation*}

This establishes the implication (ii) $\Rightarrow$ (iii). The implication (iii) $\Rightarrow$ (iv) is a consequence of Theorem~\ref{b10}, and (iv) $\Rightarrow$ (v) is trivial. Therefore, it suffices to show that (v) implies (i). If $\Theta_{\psi^{-1},W_n,V_{n-\ve}^{\ve}}(\sigma)$ is nonzero, then by \cite[Lemma~10.2]{Yam14}, $L_{\psi}(s,\sigma \times \chi_{\sigma})$ has a pole at $s=s_{\sigma}$. (We remark that there is a typo in the statement of \cite[Lemma~10.2]{Yam14}; the first condition stated there should be the holomorphy at $-s_m+\frac{1}{2}$.) This proves the implication (v) $\Rightarrow$ (i), completing the proof. \qed
\end{customproof0}

\begin{rem}\label{non-ge1}
In the proof of Theorem~\ref{prop2}, the chain of implications (iii) $\Rightarrow$ (iv) $\Rightarrow$ (v) $\Rightarrow$ (i) remains valid even when $\sigma$ is not globally generic. In the implication (i) $\Rightarrow$ (ii), we specifically utilized the almost tempered property of the local components of globally generic representations (see Remark~\ref{Ram}). Consequently, the chain of implications (iii) $\Rightarrow$ (iv) $\Rightarrow$ (v) $\Rightarrow$ (i) $\Rightarrow$ (ii) remains valid even if $\sigma$ is not globally generic, provided that it is almost tempered at every place.
\end{rem}

From Theorem~{\ref{prop2}}, we have the following corollary.
\begin{cor}\label{co11}Let $\sigma$ be an irreducible $(\psi,1)$-generic cuspidal representation of $\wt{\J}_n(\A)$. If $\Theta_{\psi^{-1},W_n,V_{n-\ve_{\sigma}}^{\ve_{\sigma}}}(\sigma)$ is nonzero, then $\Theta_{\psi^{-1},W_n,V_{n-\ve_{\sigma}}^{\ve_{\sigma}}}(\sigma)$ is $\mu_{\ve_{\sigma}}$-generic.
\end{cor}

Combining Proposition~{\ref{con}}, Corollary~\ref{e3} and Corollary~\ref{co11}, we get the following theorem. 

\begin{thm}{\label{q1}}Let $\sigma$ be an irreducible $(\psi,1)$-generic cuspidal representation of $\wt{\J}_n(\A)$. If $\Theta_{\psi,W_n,V_{k_0}^{\ve_{\sigma}}}(\sigma)$ is nonzero and $\Theta_{\psi^{-1},W_n,V_{i}^{\ve_{\sigma}}}(\sigma)=0$ for all $i<k_0$, then $\Theta_{\psi^{-1},W_n,V_{k_0}^{\ve_{\sigma}}}(\sigma)$ is $\mu_{\ve_{\sigma}}$-generic and cuspidal.
\end{thm}

\bb

\section{Application to the non-tempered GGP conjecture} \label{sec8}

Throughout this section, we assume that $F$ is a number field. For the sake of clarity, we shall adopt the standard notation $\SO_{2n+1}$, $\SO_{2n}$, $\Sp_{2n}$, and $\Mp_{2n}$ in place of $\H_{n}^1$, $\H_n^0$, $\J_n$, and $\wt{\J}_n$, respectively. For their respective subgroups, we use $\SO_{2}$, $\SO_{3}$, $\Mp_{2}$, and $\Sp_{2}$ in place of $\H^0$, $\H^1$, $\wt{\J}_{n,n-1}'$, and $\J_{n,n-1}'$.

We now apply our previous results to establish certain cases of the global Gan--Gross--Prasad (GGP) conjecture. As illustrated in Diagram~\ref{diag:logical_flow}, the proof of the main theorem is assembled as follows. Theorems~\ref{a1} and~\ref{b10} relate the non-vanishing of the relevant special periods to the non-vanishing and genericity of suitable global theta lifts. The local vanishing results of Section~\ref{sec:3}, together with a local genericity hypothesis, determine the first occurrence index by excluding any prior occurrences in the theta tower. Theorems~\ref{t2} and~\ref{prop2} then translate these period conditions into poles or non-vanishing statements for the relevant $L$-functions. Finally, these analytic conditions are interpreted in terms of global Arthur parameters, yielding one direction non-tempered Gan--Gross--Prasad cases. The converse direction is established by reversing these implications and applying the global descent theory.

\begin{figure}[htbp]
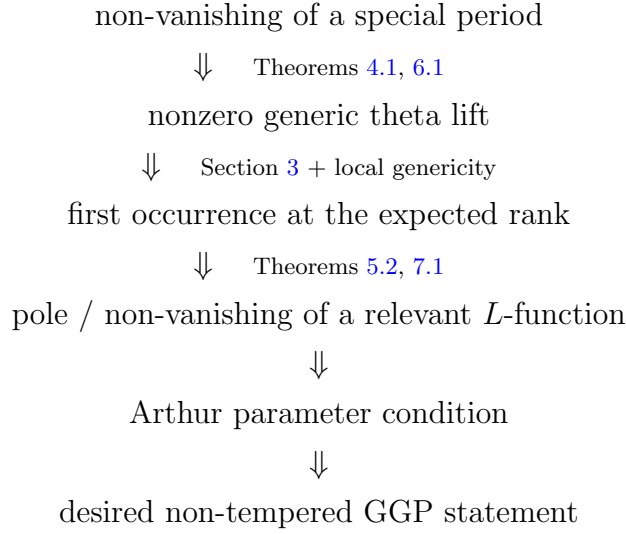

\centering
\[
\begin{array}{c}
\text{non-vanishing of a special period}
\\[0.4em]
\Downarrow \quad \text{\scriptsize Theorems~\ref{a1}, \ref{b10}}
\\[0.4em]
\text{nonzero generic theta lift}
\\[0.4em]
\Downarrow \quad \text{\scriptsize Section~\ref{sec:3} + local genericity}
\\[0.4em]
\text{first occurrence at the expected rank}
\\[0.4em]
\Downarrow \quad \text{\scriptsize Theorems~\ref{t2}, \ref{prop2}}
\\[0.4em]
\text{pole / non-vanishing of a relevant } L\text{-function}
\\[0.4em]
\Downarrow
\\[0.4em]
\text{Arthur parameter condition}
\\[0.4em]
\Downarrow
\\[0.4em]
\text{desired non-tempered GGP statement}
\end{array}
\]
\caption{Skeleton of the logical flow of the proof.}
\label{diag:logical_flow}
\end{figure}
\bb

To formulate the global GGP conjecture precisely, we first recall the notion of global $A$-parameters.

\subsection{Global $A$-parameters and $L$-packets}\label{sec:A-param}

Let $G_k$ be one of the groups $\SO_{2k}$, $\SO_{2k+1}$, $\Sp_{2k}$, or $\Mp_{2k}$. In this section, we recall the construction of global $A$-parameters and their associated packets for $G_k$.

\subsubsection{Discrete Global $A$-parameters}\label{subsec:discrete-A}
For a non-negative integer $d$, let $[d]$ denote the unique $(d+1)$-dimensional irreducible representation of $\SL_2(\CC)$. By convention, we set $[-1]$ to be $0$. A global $A$-parameter for $G_k$ is defined as a formal finite direct sum:
\begin{equation}\label{art}
M = \bigoplus_{i=1}^{r} M_i \boxtimes [d_i],
\end{equation}
where each $M_i$ is an irreducible cuspidal unitary automorphic representation of $\GL_{n_i}(\A)$. Note that each $M_i$ is almost tempered. For simplicity, we write $M_i$ for $M_i \boxtimes [0]$ and $[d]$ for $\II \boxtimes [d]$. 

A global $A$-parameter $M$ of $G_k$ is called \textbf{tempered} if $d_i=0$ for all $1 \le i \le r$. 

We say that a global $A$-parameter $M$ of $G_k=\SO_{2k}$ or $\Sp_{2k}$ (resp. $G_k=\SO_{2k+1}$ or $\Mp_{2k}$) is \textbf{discrete} if the following conditions hold:
\begin{itemize}
    \item Each $d_i \ge -1$ is an integer satisfying the dimension identity:
    \[
    \sum_{i=1}^r n_i(d_i+1) = \begin{cases} 2k, & \text{if } G_k=\SO_{2k}, \\ 2k+1, & \text{if } G_k=\Sp_{2k} \end{cases} 
    \qquad \text{(resp. } 2k, \text{ if } G_k=\SO_{2k+1} \text{ or } \Mp_{2k});
    \]
    \item $M_i$ is of orthogonal type (i.e., $L(s, M_i, \mathrm{Sym}^2)$ has a pole at $s=1$) if $d_i$ is even (resp. odd);
    \item $M_i$ is of symplectic type (i.e., $L(s, M_i, \wedge^2)$ has a pole at $s=1$) if $d_i$ is odd (resp. even);
    \item The central characters $\omega_{M_i}$ of $M_i$ satisfy $\prod_{i=1}^r \omega_{M_i}^{d_i+1} = \chi_d$ (resp. $1$);
    \item The pairs $(M_i, d_i)$ are distinct.
\end{itemize}
Under these conditions, $M$ is said to be \textbf{orthogonal} (resp. \textbf{symplectic}).

\subsubsection{Relevant $A$-parameters}
Let $M$ and $N$ be two discrete global $A$-parameters. We can write them in the form:
\begin{align*}
M &= \bigoplus_{i=1}^{t} M_{i} \boxtimes [n_i], \\
N &= \bigoplus_{i=1}^{t} M_{i} \boxtimes [m_i],
\end{align*}
where the $M_{i}$ are irreducible unitary cuspidal automorphic representations of $\GL_{r_i}(\A)$, and $n_{i}, m_{i} \ge -1$ are integers. (Recall that $[-1] = 0$.)

We say that $M$ and $N$ are \textbf{relevant} if there exists a permutation $p$ of $\{1,2,\dots, t\}$ such that
\[ 
M = \bigoplus_{i=1}^{t} M_{i} \boxtimes [n_{p(i)}] 
\]
and $|n_{p(i)} - m_i| = 1$ for all $1 \le i \le t$. In this case, one of $M$ and $N$ is symplectic, and the other is orthogonal.

\subsubsection{Local $L$-parameters and $L$-packets}\label{subsec:local-packets}

Given a discrete global $A$-parameter $M$, its local component at a place $v$ is given by $M_v = \bigoplus_{i=1}^r M_{i,v} \boxtimes [d_i]$. Let $W(F_v)$ denote the Weil group of $F_v$, and define the Weil--Deligne group as:
\[
WD(F_v) \coloneqq
\begin{cases}
W(F_v) \times \SL_2(\CC), & \text{if } v \text{ is non-archimedean,}\\[4pt]
W(F_v), & \text{if } v \text{ is archimedean.}
\end{cases}
\]
Via the local Langlands correspondence for general linear groups, we regard each component $M_{i,v}$ as a representation of $WD(F_v)$ into $\GL_{n_i}(\CC)$, so that $M_v$ constitutes a representation of $WD(F_v) \times \SL_2(\CC)$, which we call a local $A$-parameter. We say that $M_v$ is \textbf{tempered} if $d_i = 0$ for all $1 \le i \le r$. 

The local $A$-parameter induces a local $L$-parameter $\wt{M}_v \colon WD(F_v) \to {}^L G_k$ via the diagonal map:
\[ 
w \longmapsto \left(w, \mathrm{diag}(|w|^{1/2}, |w|^{-1/2})\right). 
\]
Explicitly, the induced local $L$-parameter admits the decomposition:
\begin{equation}
\wt{M}_v = \bigoplus_{i=1}^r \bigoplus_{j=0}^{d_i} M_{i,v} |\cdot|^{\frac{d_i}{2}-j}.
\end{equation}
Note that if $M_v$ is tempered, then $\wt{M}_v$ coincides with $M_v$, since the $\SL_2(\CC)$-factor is trivial.

Associated with $\wt{M}_v$ is a local $L$-packet $\Pi_{\wt{M}_v}$ consisting of irreducible smooth genuine representations of $G_k(F_v)$ (or of its pure inner forms if $G_k \in \{\SO_{2k+1}, \SO_{2k}\}$), all of which share the same $\GL_1$-twisted $\gamma$-factors as $\wt{M}_v$. Furthermore, if $M_v$ is tempered, then every member of $\Pi_{\wt{M}_v}$ is almost tempered. In addition, if the $M_{i,v}$ are all tempered representations of $WD(F_v)$ (i.e., the image is bounded), then every member of $\Pi_{\wt{M}_v}$ is tempered (see \cite{Art13, AG17, GI18, Ish24, Li24}). Note that when $G_k = \Mp_{2k}$, the packet $\Pi_{\wt{M}_v}$ depends on the choice of a non-trivial additive character $\psi_v$.
\subsubsection{Global $L$-packets and Satake parameters}\label{subsec:global-packets}
The global $L$-packet $\Pi_M$ is a set of irreducible (genuine) automorphic representations of $G_k(\A)$ (or of its pure inner forms) that occur in the discrete automorphic spectrum. It is characterized as a subset of the restricted tensor product $\sideset{}{'}\bigotimes_v \Pi_{\wt{M}_v}$ satisfying certain global conditions. For any $\pi \in \Pi_M$, the local component $\pi_v$ is unramified for almost all places $v$. Let $\{c_{1,i,v}, \dots, c_{n_i,i,v}\}$ be the Satake parameters of the unramified representation $M_{i,v}$. Then the Satake parameter of $\pi_v$ is given by the multiset
\begin{equation}\label{sat}
\bigcup_{i=1}^{r} \bigcup_{m=1}^{n_i} \left\{ c_{m,i,v} q_v^{-\frac{d_i}{2}}, c_{m,i,v} q_v^{-\frac{d_i}{2}+1}, \dots, c_{m,i,v} q_v^{\frac{d_i}{2}} \right\}.
\end{equation}
It follows that all elements of the global packet $\Pi_M$ are \textbf{nearly equivalent}; that is, they are isomorphic at almost all places $v$. When $G_k=\Mp_{2k}$, the packet $\Pi_M$ depends on the choice of $\psi$ because the local $L$-packets $\Pi_{\wt{M}_v}$ depend on $\psi_v$.

If $\pi \in \Pi_M$ is a globally generic cuspidal representation, then $M$ coincides with the functorial lift of $\pi$ to the general linear group. For $G_k \in \{\SO_n, \Mp_{2n}\}$, this is proved in \cite[Theorem~B.1]{AHKO25}, and the approach therein applies equally well to other classical groups.
\subsection{$L$-functions of $A$-parameters}
\subsubsection{Definitions of local and global $L$-functions}
We define the $L$-functions for discrete global $A$-parameters as follows. Let $\Pi_1$ and $\Pi_2$ be irreducible cuspidal automorphic representations of $\GL_n(\A)$ and $\GL_m(\A)$, respectively, and let $d_1, d_2$ be non-negative integers. For each place $v$ of $F$, we define the following local $L$-functions:\begin{align*}
&L(s, (\Pi_{1,v} \times \Pi_{2,v}) \boxtimes [d_1])  \coloneqq \prod_{i=0}^{d_1} L\left(s + \frac{d_1}{2} - i, \Pi_{1,v} \times \Pi_{2,v}\right), \\[10pt]
&L(s, (\Pi_{1,v} \boxtimes [d_1]) \times (\Pi_{2,v} \boxtimes [d_2])) \coloneqq \prod_{k=0}^{\min\{d_1, d_2\}} L(s, (\Pi_{1,v} \times \Pi_{2,v}) \boxtimes [d_1 + d_2 - 2k]).
\end{align*}
Here, $L(s, \Pi_{1,v} \times \Pi_{2,v})$ denotes the local Rankin--Selberg (or Artin) $L$-function, where $\Pi_{1,v}$ and $\Pi_{2,v}$ are viewed as representations of the Weil--Deligne group $WD(F_v)$ via the local Langlands correspondence. For the trivial character $\II_v$ of $\GL_1(F_v)$, we define:
\[ 
L(s, \Pi_{1,v} \boxtimes [d_1]) \coloneqq \prod_{i=0}^{d_1} L\left(s + \frac{d_1}{2} - i, \Pi_{1,v} \times \II_v\right). 
\]

Let $\Sym^2$ and $\wedge^2$ denote the symmetric and exterior square representations, respectively. For $\gamma \in \{\Sym^2, \wedge^2\}$, we set:
\[ 
L(s, \gamma(\Pi_{1,v}) \boxtimes [d]) \coloneqq L(s, (\gamma \circ \Pi_{1,v}) \boxtimes [d]), 
\]
where $\Pi_{1,v}$ is again viewed as a representation of $WD(F_v)$. Taking the product over all places $v$, we define the completed $L$-functions as $L(s, \cdot) \coloneqq \prod_v L(s, \cdot_v)$. Although the existence of $\gamma(\Pi_1)$ as an automorphic representation is not known in general, we formally define its completed $L$-function by:
\[ 
L(s, \gamma(\Pi_1) \boxtimes [d]) \coloneqq \prod_v L(s, \gamma(\Pi_{1,v}) \boxtimes [d]). 
\]
When $d_1 = d_2 = 0$, the $L$-functions $L(s, \Pi_1 \times \Pi_2)$ and $L(s, \gamma(\Pi_1))$ coincide with the standard completed Rankin--Selberg and symmetric/exterior square $L$-functions, respectively. We also employ the partial $L$-functions $L^S(s, \cdot) \coloneqq \prod_{v \notin S} L(s, \cdot_v)$ for a finite set of places $S$ such that $\Pi_{1,v}$ and $\Pi_{2,v}$ are unramified for $v \notin S$.

For a cuspidal representation $\rho$ of $\GL_n$ and a representation $W$ of $\SL_2(\CC)$, we define the symmetric and exterior square decompositions as follows:
\begin{align*} 
\Sym^2(\rho \boxtimes W) &\coloneqq \left( \Sym^2(\rho) \boxtimes \Sym^2(W) \right) \oplus \left( \wedge^2(\rho) \boxtimes \wedge^2(W) \right),\\
\wedge^2(\rho \boxtimes W) &\coloneqq \left( \Sym^2(\rho) \boxtimes \wedge^2(W) \right) \oplus \left( \wedge^2(\rho) \boxtimes \Sym^2(W) \right).
\end{align*}

Consider two discrete global $A$-parameters $M = \bigoplus_{\alpha} V_{\alpha}$ and $N = \bigoplus_{\beta} W_{\beta}$. We then define the following exterior and symmetric square parameters for $M$:
\begin{align*}
    \wedge^2(M) &\coloneqq \bigoplus_{\alpha} \wedge^2(V_{\alpha}) \oplus \bigoplus_{\alpha < \beta} (V_{\alpha} \otimes V_{\beta}), \\
    \Sym^2(M) &\coloneqq \bigoplus_{\alpha} \Sym^2(V_{\alpha}) \oplus \bigoplus_{\alpha < \beta} (V_{\alpha} \otimes V_{\beta}).
\end{align*}

Furthermore, the product $L$-function associated with $M$ and $N$ is defined by:
\begin{align*}
    L(s, M \times N) &\coloneqq \prod_{\alpha, \beta} L(s, V_{\alpha} \times W_{\beta}).
\end{align*}
Based on these structural decompositions, the $L$-functions $L(s, \Sym^2(M))$ and $L(s, \wedge^2(M))$, as well as their partial analogues, are defined accordingly through the products of the $L$-functions of their constituent irreducible components.

\subsubsection{Extended central $L$-value}
Suppose that $M$ and $N$ are relevant discrete $A$-parameters, with $M$ orthogonal and $N$ symplectic. It has been established that the ratio
\begin{equation}\label{clv}
L(s, M, N) \coloneqq \frac{L(s + \frac{1}{2}, M \times N)}{L(s + 1, M, \wedge^2) L(s + 1, N, \Sym^2)}
\end{equation}
is holomorphic at $s=0$ (cf. \cite[Theorem~9.7]{GGP20}). Consequently, we may define the \textbf{extended central $L$-value} $L(0, M, N)$ as the evaluation of $L(s, M, N)$ at $s=0$.

\subsection{The global GGP conjecture for some non-tempered cases}\label{relev}
In this section, we establish certain non-tempered cases of the global GGP conjecture. We begin by formulating the conjecture for the Bessel and Fourier--Jacobi cases (cf. \cite[Conjecture~9.1]{GGP20}).

First, consider the Bessel case. Let $V_m \subset V_n$ be a pair of non-degenerate orthogonal spaces over $F$ of dimensions $m$ and $n$, respectively. Let $V_m^{\perp}$ denote the orthogonal complement of $V_m$ in $V_n$, and assume that $V_m^{\perp}$ is a split orthogonal space with Witt index $k-1$. We let $\SO(V_n)$ and $\SO(V_m)$ denote the connected components of their respective isometry groups.

Let $V_m' \subset V_n'$ be another pair of non-degenerate orthogonal spaces of dimensions $m$ and $n$. If $V_m'^{\perp} \simeq V_m^{\perp}$ as orthogonal spaces, we say that $\SO(V_n') \times \SO(V_m')$ is a \textbf{relevant pure inner form} of $\SO(V_n) \times \SO(V_m)$. For a global $A$-parameter $M \times N$ of $\SO(V_n) \times \SO(V_m)$, let $\Pi_{M \times N}^{\mathrm{rel}} \subset \Pi_M \times \Pi_N$ denote the set of automorphic representations of $\SO(V_n') \times \SO(V_m')$, where the pair $\SO(V_n') \times \SO(V_m')$ ranges over all relevant pure inner forms. 

While the Bessel period $\mc{B}_{k,\psi}^{\ve}$ was previously defined for automorphic forms on the quasi-split pair $\SO(V_n) \times \SO(V_m)$, it extends analogously to any relevant pure inner form. We shall retain the same notation for this extension.

Next, we consider the Fourier--Jacobi case. Let $W_m \subset W_n$ be a pair of symplectic spaces over $F$ of dimensions $2m$ and $2n$, respectively. We decompose $W_n = W_m \oplus W_m^{\perp}$ and assume that $W_m^{\perp}$ is a split symplectic space with Witt index $k$. Let $\Sp(W_m)$ and $\Sp(W_n)$ be the corresponding isometry groups, and let $\Mp(W_m)$ and $\Mp(W_n)$ denote their respective double covers. 

The Fourier--Jacobi period $\mc{FJ}_{k,\psi}^{\lambda}$ is defined for automorphic forms on the pair $\Mp(W_n) \times \Sp(W_m)$ or $\Sp(W_n) \times \Mp(W_m)$. For a global $A$-parameter $M \times N$ associated with such a pair, we similarly let $\Pi_{M \times N}^{\mathrm{rel}} \coloneqq \Pi_M \times \Pi_N$ denote the associated set of representations. We note that, unlike the orthogonal case, the pure inner forms of $\Sp(W_n)$ and $\Mp(W_n)$ are unique (up to isomorphism), and thus the set $\Pi_{M \times N}^{\mathrm{rel}}$ consists of the product of the global packets on the groups themselves.

\begin{con1}[{\cite[Conjecture~9.1]{GGP20}}]
Let $(G_n, G_m)$ be one of the following pairs of classical groups:
\[
(G_n, G_m) \in \{ (\Sp(W_n), \Mp(W_m)), (\Mp(W_n), \Sp(W_m)) \}.
\]
Let $M \times N$ be a discrete global $A$-parameter of $G_n \times G_m$. Consider an irreducible (genuine) discrete automorphic representation $\pi_1 \boxtimes \pi_2$ of $G_n(\A) \times G_m(\A)$ associated with the $A$-parameter $M \times N$. Then the following assertions hold:

\begin{enumerate}
    \item Suppose that $n-m$ is odd.
    \begin{enumerate}
        \item If the Bessel period $\mc{B}_{k,\psi}^{\ve}$ on $\pi_1 \boxtimes \pi_2$ is nonzero, then $(M, N)$ is a relevant pair.
        \item Assume that $(M, N)$ is a relevant pair. Then the following conditions are equivalent:
        \begin{enumerate}
            \item[(1)] There exists an irreducible discrete automorphic representation $\pi_1' \boxtimes \pi_2'$ in $\Pi_{M \times N}^{\mathrm{rel}}$ such that $\mc{B}_{k,\psi}^{\ve}$ on $\pi_1' \boxtimes \pi_2'$ is nonzero.
            \item[(2)] $L(0, M, N)$ is nonzero.
        \end{enumerate}
    \end{enumerate}

    \item Suppose that $n-m$ is even.
    \begin{enumerate}
        \item If the Fourier--Jacobi period $\mc{FJ}_{k,\psi}^1$ is nonzero on $\pi_1 \boxtimes \pi_2 \boxtimes \nu_{\psi^{-1},W_{m}}^{\lambda}$, then $(M, N)$ is a relevant pair.
        \item Assume that $(M, N)$ is a relevant pair. Then the following conditions are equivalent:
        \begin{enumerate}
            \item[(1)] There exists an irreducible discrete automorphic representation $\pi' = \pi_1' \boxtimes \pi_2'$ in $\Pi_{M \times N}^{\mathrm{rel}}$ such that $\mc{FJ}_{k,\psi}^1$ is nonzero on $\pi_1' \boxtimes \pi_2' \boxtimes \nu_{\psi^{-1},W_{m}}^{\lambda}$.
            \item[(2)] $L(0, M, N)$ is nonzero.
        \end{enumerate}
    \end{enumerate}
\end{enumerate}
\end{con1}
For the remainder of this section, we establish the conjecture in three distinct non-tempered settings, namely:
\[
(G_n, G_m) \in \{ (\SO_{2n}, \SO_3), \ (\Sp_{2n}, \Mp_2), \ (\Mp_{2n}, \Sp_2) \}.
\]

In these respective cases, we take the non-tempered $A$-parameter $N$ of $G_m$ to be $[1]$, $\chi_d \boxtimes [1]$, and $[2]$. We note that when $(G_m,N)=(\SO_3,[1])$ or $(\Sp_2,[2])$, the local component group of $N$ is trivial (a singleton). Consequently, the global $L$-packet $\Pi_N$ consists of a unique element, which is the trivial character $\II$ of $G_m(\A)$ (see \cite[Sect.~25]{GGP12}).

When $(G_m,N)=(\Mp_2,\chi_d \boxtimes [1])$, the local component group of $N$ is $\ZZ / 2\ZZ$, and the global $L$-packet $\Pi_N$ is given by
\[
\Pi_N \coloneqq \big\{ \omega_{\psi_d, W_{n,n-1}'}^S \big\}_{|S| \text{ is even}}
\]
(see Proposition~\ref{p.t} and \cite[Proposition~10.2.4]{Li24}).

For $G_n=\SO_{2n}$ or $\Sp_{2n}$, we take $M$ to be a tempered $A$-parameter relevant to $N$. For $G_n=\Mp_{2n}$, we take $M$ to be a non-tempered $A$-parameter relevant to $N$.

\bb

Regarding the analytic properties of $L$-functions, we rely on the following well-known facts (see \cite{JSS83, Sha88} and \cite[Theorem~1.3]{Gr11}).

\begin{fac} \label{fac1} 
Let $\sigma_1$ and $\sigma_2$ be cuspidal automorphic representations of $\GL_n(\A)$ and $\GL_m(\A)$, respectively. Let $\Sym^2$ and $\wedge^2$ denote the symmetric and exterior square representations, respectively. We also let $L(s,\gamma(\sigma_1))$ denote the completed symmetric or exterior square $L$-function, corresponding to $\gamma=\Sym^2$ or $\wedge^2$. We review some analytic properties of $L(s,\sigma_1 \times \sigma_2)$ and $L(s,\gamma(\sigma_1))$:
\begin{itemize}
    \item $L(s,\sigma_1 \times \sigma_2)$ and $L(s,\gamma(\sigma_1))$ admit meromorphic continuations to the whole complex plane.
    \item $L(s,\sigma_1 \times \sigma_2)$ and $L(s,\gamma(\sigma_1))$ satisfy functional equations relating their values at $s$ and $1-s$.
    \item $L(s,\sigma_1 \times \sigma_2)$ and $L(s,\gamma(\sigma_1))$ are holomorphic for all $s \in \CC \setminus \{0,1\}$.
    \item In the regions $\Re(s) \ge 1$ and $\Re(s) \le 0$, the $L$-functions $L(s,\sigma_1 \times \sigma_2)$ and $L(s, \gamma(\sigma_1))$ are nonzero.
    \item $L(s,\sigma_1 \times \sigma_2)$ has poles at $s=0$ and $s=1$ if and only if $m=n$ and $\sigma_2 \simeq \sigma_1^{\vee}$ (the contragredient dual of $\sigma_1$). In this case, the poles are simple.
    \item If $\sigma_1$ is not self-dual, then $L(s,\gamma(\sigma_1))$ is an entire function.
    \item If $\sigma_1$ is self-dual (i.e., $\sigma_1 \simeq \sigma_1^{\vee}$), exactly one of $L(s,\Sym^2(\sigma_1))$ and $L(s,\wedge^2(\sigma_1))$ has simple poles at $s=0$ and $s=1$.
\end{itemize}
\end{fac} 

We also employ the following elementary facts concerning the decomposition of representations of $\SL_2(\CC)$ for various algebraic manipulations:
\begin{itemize}
    \item $\Sym^2([i]) = [2i] \oplus [2i-4] \oplus [2i-8] \oplus \cdots$
    \item $\wedge^2([i]) = [2i-2] \oplus [2i-6] \oplus [2i-10] \oplus \cdots$
    \item $[i] \otimes [j] = \bigoplus_{k=0}^{\min\{i, j\}} [i+j-2k]$
\end{itemize}
In what follows, we denote $\mc{B}_{k,\psi}^{\ve}$ and $\mc{FJ}_{k,\psi}^1$ simply by $\mc{B}$ and $\mc{FJ}$, respectively, since the omitted data is clear from the context.

\subsubsection{$\SO_{2n} \times \SO_3$ case}
\begin{thm}\label{thm:main-SO}
For $n \ge 2$, let $M$ be a discrete global $A$-parameter of $\SO_{2n}$ and let $N = [1]$ be the non-tempered discrete global $A$-parameter of $\SO_3$. Then the following assertions hold:
\begin{enumerate}
    \item Suppose that $\pi \in \Pi_M$ is cuspidal and $\pi_v$ is $\mu_{0,v}$-generic for some finite place $v$ of $F$. If the Bessel period $\mc{B}$ on $\pi$ (i.e., on $\pi \boxtimes \II$) is nonzero, then $M$ is tempered and the pair $(M, N)$ is relevant.
    
    \item Suppose that $M$ is tempered. Then there exists $\pi \in \Pi_M$ such that $\mc{B}(\pi) \neq 0$ if and only if $L(0,M,N) $ is nonzero.
\end{enumerate}
\end{thm}

\begin{proof}
We first prove (i). Suppose that $\mc{B}$ is nonzero on $\pi$. Choose a Witt tower of symplectic spaces $\{W_k\}_{k\ge1}$ of discriminant $d$. Then by Theorem~\ref{t2} and Remark~\ref{non-ge}, the global theta lift $\Theta_{\psi,V_n^{0},W_{n-1}}(\pi)$ of $\pi$ to $\Sp_{2n-2}(\A)$ is nonzero and $(\psi,c)$-generic. 
Furthermore, Proposition~\ref{c1} implies that $\Theta_{\psi,V_n^{0},W_{n-3}}(\pi) = 0$ since the non-vanishing of $\mc{B}$ on $\pi$ implies that $\pi_v$ is of $\mu_{n-1,0}$-type. Therefore, $n(\pi)$, the first occurrence index of $\pi$, is $n-1$ or $n-2$. Suppose that $n(\pi) = n-2$. Then $\Theta_{\psi,V_n^{0},W_{n-2}}(\pi)$ is cuspidal and irreducible. Thus, $(\Theta_{\psi,V_n^{0},W_{n-2}}(\pi))_v$ is the local theta lift $\Theta_{\psi_v,V_{n,v}^{0},W_{n-2,v}}(\pi_{v})$ of $\pi_{v}$ to $\Sp_{2n-4}(F_v)$. However, this contradicts Proposition~\ref{c1} since $\pi_{v}$ is $\mu_{0,v}$-generic by our assumption. Therefore, we must have $n(\pi) = n-1$.

Put $\sigma = \Theta_{\psi, V_n^0, W_{n-1}}(\pi)$. By Theorem~\ref{a1}, $\sigma$ is a $(\psi, c)$-generic cuspidal automorphic representation of $\Sp_{2n-2}(\A)$. Let $N'$ denote the discrete tempered $A$-parameter associated with $\sigma$. Then, by the results of Kudla on the theta correspondence for unramified representations (\cite{Ku86}), the local parameters satisfy $M_v = \II_{\GL_1, v} \oplus \chi_{d, v} \cdot N_v'$ for almost all places $v$ (cf. \eqref{sat}). 

It follows that $\pi$ belongs to both $\Pi_M$ and $\Pi_{\II_{\GL_1} \oplus \chi_d \cdot N'}$. Since any two distinct global $A$-packets are disjoint, we deduce that $M = \II_{\GL_1} \oplus \chi_d \cdot N'$. Consequently, $M$ is tempered and the pair $(M, N)$ is relevant, which completes the proof of assertion (i).

Next, we prove (ii). 

\textit{Proof of $(\Rightarrow)$:} Suppose that there exists $\pi_1 \in \Pi_M$ such that $\mc{B}$ is nonzero on $\pi_1$. Since $M$ is a tempered $A$-parameter, $\pi_1$ must be cuspidal and almost tempered. Then by Theorem~\ref{t2} and Remark~\ref{non-ge}, $L(s,\pi_1)$ has a simple pole at $s=1$. Consequently, by Proposition~\ref{eql}, it follows that $L(s,M)$ also has a pole at $s=1$. Since $M$ is an orthogonal parameter, it is straightforward to check that $\operatorname{ord}_{s=1} \big( L(s,\wedge^2(M)) \big) = 0$. Furthermore, since $\Sym^2(N) = \II_{\GL_1} \boxtimes [2]$, we have $\operatorname{ord}_{s=1} \big( L(s,\Sym^2(N)) \big) = -2$.
 
On the other hand, we have 
\[
L(s+1/2, M \times N) = L(s+1, M) \cdot L(s, M).
\]
Since $\operatorname{ord}_{s=1} \big( L(s,M) \big) = -1$, we obtain $\operatorname{ord}_{s=0} \big( L(s+1/2, M \times N) \big) = -2$. 
Therefore, by the definition of the adjoint $L$-function $L(s, M, N)$, we deduce
\begin{equation}\label{ord}
    \operatorname{ord}_{s=0} \big( L(s, M, N) \big) = -2 - (-2) = 0.
\end{equation}
This proves the forward direction of (ii).

\textit{Proof of $(\Leftarrow)$:} From the calculation above, we see that
\[
\operatorname{ord}_{s=0} \big( L(s,M,N) \big) = 2 \cdot \operatorname{ord}_{s=1} \big( L(s,M) \big) + 2.
\]
Therefore, if $L(0,M,N) \neq 0$ (meaning the order is $0$), then
\[
\operatorname{ord}_{s=1} \big( L(s,M) \big) = -1,
\]
which implies that $L(s,M)$ has a simple pole at $s=1$. 

On the other hand, let $\Pi$ be the isobaric automorphic representation of $\GL_{2n}(\A)$ corresponding to $M$. By the global descent of $\Pi$ from $\GL_{2n}(\A)$ to $\SO_{2n}(\A)$ (see \cite[Theorem~11.2]{GRS11}), we obtain a globally $\mu_0$-generic cuspidal automorphic representation $\pi_1$ of $\SO_{2n}(\A)$. 
Since $\pi_1$ has the $A$-parameter $M$, the $L$-function $L(s,\pi_1)$, which is equal to $L(s,M)$ by Proposition~\ref{eql}, has a pole at $s=1$. 
Hence, by Theorem~\ref{t2}, $\mc{B}$ is nonzero on $\pi_1$. 
This completes the proof of (ii). 
\end{proof}

\begin{rem}
When $\SO_{2n}$ is split, \cite[Theorem~1.6]{JL14} shows that $\mu_{0,v}$-genericity implies that $M$ is tempered and $\pi$ is cuspidal. Therefore, in this split case, we do not need to assume that $\pi$ is cuspidal. However, for quasi-split $\SO_{2n}$, this implication is currently unknown, and thus we have made this assumption explicit. It is also remarkable that Hazeltine, Liu, and Lo \cite[Theorem~4.6]{HLL24} proved the enhanced version of the Shahidi conjecture for local $A$-packets of $\Sp_{2n}$ and split $\SO_{2n+1}$.
\end{rem}
\subsubsection{$\Sp_{2n} \times \Mp_2$ case}
\begin{thm}\label{thm:main-Sp}
For $n \ge 1$, let $M$ be a discrete $A$-parameter of $\Sp_{2n}$ and, for $d\in F^{\times}$, let $N = \chi_d \boxtimes [1]$ be a non-tempered discrete $A$-parameter of $\Mp_2$. Then the following assertions hold:
\begin{enumerate}
    \item Suppose that $\sigma \in \Pi_M$ and $\sigma_v$ is $(\psi,1)$-generic for some finite place $v$ of $F$. If there is some $\rho \in \Pi_N$ such that the Fourier--Jacobi period $\mc{FJ}$ is nonzero on $\sigma \boxtimes \rho \boxtimes \nu_{\psi^{-1},W_{n,n-1}'}$, then the pair $(M,N)$ is relevant.
    
    \item Suppose that $M$ is tempered. Then there exist $\sigma_1 \in \Pi_M$ and $\rho_1 \in \Pi_N$ such that $\mc{FJ}$ is nonzero on $\sigma_1 \boxtimes \rho_1 \boxtimes \nu_{\psi^{-1},W_{n,n-1}'}$ if and only if $L(0,M,N) \neq 0$.
\end{enumerate}
\end{thm}

\begin{proof}
Before proceeding with the proof, we first note that since $\omega_{\psi_d,W_{n,n-1}'}$ is the direct sum of all members of $\Pi_N$, for any $\sigma \in \Pi_M$, the condition that there exists some $\rho \in \Pi_N$ such that $\mc{FJ}$ is nonzero on $\sigma \boxtimes \rho \boxtimes \nu_{\psi^{-1},W_{n,n-1}'}$ is equivalent to $\mc{FJ}$ being nonzero on $\sigma \boxtimes \omega_{\psi_d,W_{n,n-1}'} \boxtimes \nu_{\psi^{-1},W_{n,n-1}'}$.

We first prove (i). By \cite[Theorem~1.3]{Liu11}, the local genericity of $\sigma_v$ implies the temperedness of $M$. Therefore, $\sigma$ is cuspidal.

Suppose that $\mc{FJ}$ is nonzero on $\sigma \boxtimes \rho \boxtimes \nu_{\psi^{-1},W_{n,n-1}'}$ for some $\rho \in \Pi_N$. Choose a Witt tower of even-dimensional orthogonal spaces $\{V_{k}^0\}_{k\ge1}$ of type $(d,1)$. By Theorem~\ref{prop2}, the global theta lift $\Theta_{\psi^{-1},W_n,V_{n}^0}(\sigma)$ of $\sigma$ to $\SO_{2n}(\A)$ is nonzero and $\mu_{0}$-generic. 
Furthermore, Proposition~\ref{con} implies that $\Theta_{\psi^{-1},W_n,V_{n-2}^0}(\sigma)=0$ because $\mc{FJ}$ is nonzero on $\sigma \boxtimes \omega_{\psi_d,W_{n,n-1}'} \boxtimes \nu_{\psi^{-1},W_{n,n-1}'}$.
Therefore, the first occurrence index $n(\sigma)$ of $\sigma$ is $n$ or $n-1$. Suppose that $n(\sigma)=n-1$. Then $\Theta_{\psi^{-1},W_n,V_{n-1}^0}(\sigma)$ is cuspidal and irreducible. Thus, $(\Theta_{\psi^{-1},W_n,V_{n-1}^0}(\sigma))_v$ is the local theta lift $\Theta_{\psi_v^{-1},W_{n,v},V_{n-1,v}^0}(\sigma_v)$ of $\sigma_{v}$ to $V_{n-1,v}^0$. However, this contradicts Proposition~\ref{con} since $\sigma_{v}$ is $(\psi,1)$-generic by our assumption. Therefore, we must have $n(\sigma)=n$.

Set $\pi = \Theta_{\psi^{-1}, W_n, V_n^0}(\sigma)$, and let $M'$ denote the $A$-parameter associated with $\pi$. According to Theorem~\ref{b10}, $\pi$ is an irreducible $\mu_0$-generic cuspidal automorphic representation of $\SO_{2n}(\A)$, which implies that $M'$ is a discrete  tempered parameter of $\SO_{2n}(\A)$. By the results of Kudla on the theta correspondence for unramified representations \cite{Ku96, Ku86} and the disjointness of global $A$-packets, it follows that $M = \chi_d \oplus (\chi_d \otimes M')$ (cf. \eqref{sat}). Since $M'$ is tempered, the pair $(M, N)$ is relevant. This completes the proof of assertion~(i).

Next, we prove (ii). 

\textit{Proof of $(\Rightarrow)$:} Suppose that there exist $\sigma_1 \in \Pi_M$ and $\rho_1 \in \Pi_N$ such that $\mc{FJ}$ is nonzero on $\sigma_1 \boxtimes \rho_1 \boxtimes \nu_{\psi^{-1},W_{n,n-1}'}$. By Theorem~\ref{prop2} and Remark~\ref{non-ge1}, $L_{\psi}(s,\sigma_1 \times \chi_d)$ has a pole at $s=1$, and consequently, $L(s,M \times \chi_d)$ also has a pole at $s=1$ by Proposition~\ref{eql}. Furthermore, since $M$ is a tempered orthogonal $A$-parameter, we have $\operatorname{ord}_{s=1} \big( L(s,\wedge^2(M)) \big) = 0$.
Note that 
\[
\Sym^2(N) = \II \boxtimes \Sym^2([1]) = [2],
\]
and therefore, $\operatorname{ord}_{s=1} \big( L(s,\Sym^2(N)) \big) = -2$.
On the other hand,  
\begin{equation*}
\begin{split}
&\operatorname{ord}_{s=0} \big( L(s+1/2,M \times N) \big)  = \operatorname{ord}_{s=0} \big( L(s,M \times \chi_d) \cdot L(s+1,M \times \chi_d) \big) \\
& = 2 \cdot \operatorname{ord}_{s=0} \big( L(s+1,M \times \chi_d) \big) = -2.
\end{split}
\end{equation*}
Therefore, $L(0,M,N) \neq 0$. This proves the forward direction of (ii).

\textit{Proof of $(\Leftarrow)$:} Let $\Pi$ be the isobaric automorphic representation of $\GL_{2n}(\A)$ corresponding to $M$. By the global descent of $\Pi$ from $\GL_{2n}(\A)$ to $\Sp_{2n}(\A)$ (see \cite[Theorem~11.2]{GRS11}), we obtain a globally $(\psi,1)$-generic cuspidal automorphic representation $\sigma_1$ of $\Sp_{2n}(\A)$. From the calculation above, since $L(0,M,N) \neq 0$, the $L$-function $L_{\psi}(s,\sigma_1\times \chi_d) = L(s,M\times \chi_d)$ has a pole at $s=1$. Then, by Theorem~\ref{prop2}, $\mc{FJ}$ is nonzero on $\sigma_1 \boxtimes \omega_{\psi_d,W_{n,n-1}'} \boxtimes \nu_{\psi^{-1},W_{n,n-1}'}$. This proves the backward direction of (ii).
\end{proof}
\begin{rem}
    The non-tempered local GGP conjecture \cite[Conjecture~6.1]{GGP20} predicts that $\rho_1 \in \Pi_N$ in (ii) should be $\omega_{\psi_d,W_{n,n-1}'}^{\emptyset}$.
\end{rem}

\subsubsection{$\Mp_{2n} \times \Sp_2$ case}
\begin{thm}\label{thm:main-Mp}
For $n \ge 1$, let $M$ be a discrete $A$-parameter of $\Mp_{2n}$ and let $N= [2]$ be a non-tempered discrete $A$-parameter of $\Sp_2$. Then the following assertions hold:
\begin{enumerate}
    \item Suppose that $\sigma \in \Pi_{M}$ is cuspidal and $\sigma_v$ is $(\psi_v,1)$-generic for some finite place $v$ of $F$. If $\mc{FJ}$ is nonzero on $\sigma \boxtimes \nu_{\psi^{-1},W_{n,n-1}'}$, then $M = M' \oplus [1]$ for some discrete tempered $A$-parameter $M'$, and the pair $(M,N)$ is relevant.
    
    \item Suppose that $M = M' \oplus [1]$ for some discrete tempered $A$-parameter $M'$. Then there exists a cuspidal $\sigma_1 \in \Pi_M$ such that $\mc{FJ}$ is nonzero on $\sigma_1 \boxtimes \nu_{\psi^{-1},W_{n,n-1}'}$ if and only if $L(0,M,N) \neq 0$.
\end{enumerate}
\end{thm}

\begin{proof}
We first prove (i). Suppose that $\mc{FJ}$ is nonzero on $\sigma \boxtimes \nu_{\psi^{-1},W_{n,n-1}'}$.
Choose a Witt tower of odd-dimensional orthogonal spaces $\{V_{k}^{1}\}_{k \ge 1}$ of discriminant $1$. 
By Theorem~\ref{prop2}, the global theta lift $\Theta_{\psi^{-1},W_n,V_{n-1}^{1}}(\sigma)$ of $\sigma$ to $\SO_{2n-1}(\A)$ is nonzero and $\mu_{0}$-generic. 
Furthermore, Proposition~\ref{con} implies that $\Theta_{\psi^{-1},W_n,V_{n-3}^{1}}(\sigma) = 0$ because $\mc{FJ}$ is nonzero on $\sigma \boxtimes \nu_{\psi^{-1},W_{n,n-1}'}$. 
Hence, the first occurrence index $n(\sigma)$ is either $n-1$ or $n-2$. 
Suppose that $n(\sigma)=n-2$. 
Then $\Theta_{\psi^{-1},W_n,V_{n-2}^{1}}(\sigma)$ is irreducible and cuspidal. 
Consequently, $(\Theta_{\psi^{-1},W_n,V_{n-2}^{1}}(\sigma))_v$ is the local theta lift $\Theta_{\psi_v^{-1},W_{n,v},V_{n-2,v}^{1}}(\sigma_v)$ of $\sigma_v$ to $V_{n-2,v}^{1}$. 
However, this contradicts Proposition~\ref{con}, since $\sigma_v$ is $(\psi_v,1)$-generic by our assumption. 
Therefore, we conclude that $n(\sigma)=n-1$.

Set $\pi = \Theta_{\psi^{-1}, W_n, V_{n-1}^1}(\sigma)$, and let $M'$ denote the $A$-parameter associated with $\pi$. According to Theorem~\ref{b10}, $\pi$ is an irreducible $\mu_1$-generic cuspidal automorphic representation of $\SO_{2n-1}(\A)$, which implies that $M'$ is a discrete tempered $A$-parameter of $\SO_{2n-1}(\A)$. 

By Kudla's results on the theta correspondence for unramified representations \cite{Ku96, Ku86} and the disjointness of global $A$-packets, the global $A$-parameters $M$ and $M'$ satisfy the relation $M = M' \oplus [1]$ (cf. \eqref{sat}). Consequently, $(M, N)$ constitutes a relevant pair, thereby completing the proof of assertion (i).

Next, we prove (ii). 

\textit{Proof of $(\Rightarrow)$:} Suppose that there exists a cuspidal representation $\sigma_1 \in \Pi_M$ such that $\mc{FJ}$ is nonzero on $\sigma_1 \boxtimes \nu_{\psi^{-1},W_{n,n-1}'}$. Since $M = M' \oplus [1]$, by Proposition~\ref{eql}, we have
\[
L_{\psi}(s,\sigma_1) = L(s,M') \cdot L(s-1/2,\II_{\GL_1}) \cdot L(s+1/2,\II_{\GL_1}).
\]
By Theorem~\ref{prop2} and Remark~\ref{non-ge1}, $L(s,\sigma_1)$ has a simple pole at $s=3/2$, and consequently, 
\[
\operatorname{ord}_{s=3/2} \big( L(s,M') \big) = 0.
\]

Note that 
\begin{align*}
\Sym^2(M) &= \Sym^2([1]) \oplus \Sym^2(M') \oplus ([1] \otimes M') \\
&\quad = [2] \oplus \Sym^2(M') \oplus ([1] \otimes M'), \\[5pt]
\wedge^2(N) &= [2].
\end{align*}
Therefore, 
\begin{align*}
\operatorname{ord}_{s=1} \big( L(s,\Sym^2(M)) \big) &= -2 + \operatorname{ord}_{s=1/2} \big( L(s,M') \big) + \operatorname{ord}_{s=3/2} \big( L(s,M') \big), \\
\operatorname{ord}_{s=1} \big( L(s,\wedge^2(N)) \big) &= -2.
\end{align*}
On the other hand, since 
\[
M \otimes N = [1] \oplus [3] \oplus (M' \otimes [2]),
\]
we have
\begin{equation*}
\begin{split}
&\operatorname{ord}_{s=0} \big( L(s+1/2, M \times N) \big)  = -2 - 2 + 2 \cdot \operatorname{ord}_{s=3/2} \big( L(s,M') \big) + \operatorname{ord}_{s=1/2} \big( L(s,M') \big).
\end{split}
\end{equation*}
Since $\operatorname{ord}_{s=3/2} \big( L(s,M') \big) = 0$, we deduce that $L(0,M,N) \neq 0$. This proves the forward direction of (ii).

\textit{Proof of $(\Leftarrow)$:} The calculation above shows that $L(0,M,N) \neq 0$ implies $\operatorname{ord}_{s=3/2} \big( L(s,M') \big) = 0$. Let $\Pi$ be the isobaric automorphic representation $\II_{\GL_2} \times M_1 \times \cdots \times M_k$ of $\GL_{2n}(\A)$ corresponding to $M = M' \oplus [1]$. By the global descent of $\Pi$ from $\GL_{2n}(\A)$ to $\Mp_{2n}(\A)$ (see \cite[Theorem~11.2]{GRS11}), we obtain a globally $(\psi,1)$-generic cuspidal automorphic representation $\sigma_1$ of $\Mp_{2n}(\A)$. 

Note that 
\[
\operatorname{ord}_{s=3/2} \big( L(s,M) \big) = -1 + \operatorname{ord}_{s=3/2} \big( L(s,M') \big) = -1.
\]
Since $\sigma_1$ has the $A$-parameter $M$, the $L$-function $L_{\psi}(s,\sigma_1) = L(s,M)$ has a pole at $s=3/2$. Then, by Theorem~\ref{prop2}, $\mc{FJ}$ is nonzero on $\sigma_1 \boxtimes \nu_{\psi^{-1},W_{n,n-1}'}$. This proves the backward direction of (ii).
\end{proof}

\subsection*{Acknowledgements}
We are deeply indebted to the mathematical legacy of Benedict Gross and Stephen Rallis, whose pioneering works have profoundly guided the direction of this research. We express our sincere gratitude to Hiraku Atobe, Wee Teck Gan, Eyal Kaplan and Dipendra Prasad for their invaluable comments on an earlier draft of this paper. We also thank Kazuki Morimoto for bringing his  result to our attention. 

The first author thanks KAIST and Catholic Kwandong University for providing excellent working environments. He is also grateful to Dongho Byeon, Youn-Seo Choi, Youngju Choie, Wee Teck Gan, Haseo Ki, Zhengyu Mao, Sug Woo Shin, and Shunsuke Yamana for their constant support and encouragement. This work was supported by a National Research Foundation of Korea (NRF) grant (No. RS-2023-00237811). During the preparation of this manuscript, the authors utilized Google’s Gemini to improve English language clarity and assist with LaTeX formatting.
\appendix
\section{The equality of $L$-functions} \label{appendix:L-factors}

In this appendix, we provide a proof for the identity between the analytic $L$-functions defined by the doubling method and $L$-functions associated with Arthur parameters.

\begin{prop}\label{eql}
Let $G_n$ be one of the classical groups $\SO_{2n}$, $\Sp_{2n}$, $\SO_{2n+1}$, or the metaplectic group $\Mp_{2n}$. In the metaplectic case, fix the additive character $\psi$ used to define the corresponding packets and $L$-factors. Let 
\[
M = \bigoplus_{i=1}^{r} M_{i} \boxtimes [d_i], \quad d_i \ge 0
\]
be a global discrete $A$-parameter of $G_n$, where each $M_i$ is an irreducible unitary cuspidal automorphic representation of $\GL_{n_i}(\A)$. Assume that $M$ satisfies one of the following conditions:
\begin{enumerate}
    \item Each $M_{i}$ is tempered (i.e., everywhere locally tempered);
    \item For every $i$ with $d_i \ge 1$, the representation $M_i$ is tempered.
\end{enumerate}
Then for any $\pi \in \Pi_{M}$ and any automorphic character $\chi$ of $\GL_1(\A)$, we have the equality of completed $L$-functions:
\[
L(s, \pi \times \chi) = L(s, M \times \chi).
\]
\end{prop}
\begin{proof}
 As the global $L$-function identity follows by taking the product over all places $v$, it suffices to show the identity for the local $L$-factors at each place $v$ of $F$:
\[
L(s, \pi_v \times \chi_v) = L(s, M_{v} \times \chi_v),
\]
where $M_{v}$ (and $M_{i,v}$) is viewed as a representation of the Weil--Deligne group $WD(F_v) \times \SL_2(\CC)$ via the local Langlands correspondence (LLC) for $\GL_n$, and $L(s, \pi_v \times \chi_v)$ is the local $L$-factor defined via the doubling method \cite{LR05, Yam14}. We divide the proof into two cases, depending on which condition in the statement $M$ satisfies:
\begin{itemize}
    \item \textbf{Case (i)} 
    
   \subsubsection*{Step 1: The Tempered Case}
Assume that the local $A$-parameter $M_v$ is tempered (i.e., $d_i = 0$ for all $i$). Consequently, its associated local $L$-parameter $\wt{M}_v$ coincides with $M_v$ and is therefore a bounded representation of $WD(F_v)$, since each $M_{i,v}$ is tempered. By the local Langlands correspondence (LLC) for $G_n$ established in \cite{Art13, AG17, GI18, Ish24, Li24}, any representation $\pi_v \in \Pi_{\wt{M}_v}$ is tempered, and we have the following equality of local gamma factors:
$$\gamma(s, \pi_v \times \chi_v, \psi_v) = \gamma(s, \wt{M}_v \times \chi_v, \psi_v).$$
Here, $\gamma(s, \pi_v \times \chi_v, \psi_v)$ is the standard $\gamma$-factor defined by Lapid--Rallis \cite{LR05} via the doubling method, and $\gamma(s, \wt{M}_v \times \chi_v, \psi_v)$ is Tate's $\gamma$-factor as defined in \cite{Ta79}. Recall that the $\gamma$-factor relates to the $L$-factors and the epsilon factor via the identity:
$$\gamma(s, \pi_v \times \chi_v, \psi_v) = \varepsilon(s, \pi_v \times \chi_v, \psi_v) \frac{L(1-s, \pi_v^{\vee} \times \chi_v^{-1})}{L(s, \pi_v \times \chi_v)}.$$
For a tempered representation $\pi_v$, the poles of $L(s, \pi_v \times \chi_v)$ are confined to the half-plane $\operatorname{Re}(s) \le 0$, whereas the poles of $L(1-s, \pi_v^{\vee} \times \chi_v^{-1})$ lie in $\operatorname{Re}(s) \ge 1$. Because there is no cancellation of poles between the numerator and the denominator, the $L$-factor $L(s, \pi_v \times \chi_v)$ is uniquely determined as the inverse of the polynomial part of the numerator of $\gamma(s, \pi_v \times \chi_v, \psi_v)^{-1}$. Therefore, the equality of the gamma factors guarantees the equality of the $L$-factors:
$$L(s, \pi_v \times \chi_v) = L(s, \wt{M}_v \times \chi_v) = L(s, M_v \times \chi_v).$$

    \subsubsection*{Step 2: The General Case}
    Now consider the general case where $M_v = \bigoplus_{i=1}^r M_{i,v} \boxtimes [d_i]$. The associated local $L$-parameter is 
    \[
    \wt{M}_v = \bigoplus_{i=1}^r \bigoplus_{j=0}^{d_i} M_{i,v} |\cdot|^{\frac{d_i}{2}-j}.
    \]
    Since each $M_i$ is self-dual, its local component $M_{i,v}$ is also self-dual. Therefore, we can rewrite $\wt{M}_v$ as:
    \begin{equation*}
    \wt{M}_v = \bigoplus_{i=1}^{e} N_{i,v} |\cdot|^{s_i} \oplus N_{v} \oplus \bigoplus_{i=1}^{e} N_{i,v}^{\vee} |\cdot|^{-s_i},
    \end{equation*}
    where $N_{i,v}$ are irreducible tempered representations, $s_1 \ge \dots \ge s_e > 0$, and $N_v$ is a tempered $L$-parameter. 

    By the definition of $L$-factors for $A$-parameters, we have:
    \begin{align*}
    L(s, M_v \times \chi_v) &= \prod_{i=1}^r \prod_{j=0}^{d_i} L\left(s + \frac{d_i}{2} - j, M_{i,v} \times \chi_v\right) \\
    &= L(s, N_v \times \chi_v) \prod_{i=1}^e \left( L(s+s_i, N_{i,v} \times \chi_v) \cdot L(s-s_i, N_{i,v}^{\vee} \times \chi_v) \right).
    \end{align*}

    On the other hand, by the LLC for $G_n$, the local $L$-packet $\Pi_{\wt{M}_v}$ consists of the Langlands quotients $L(N_{1,v} |\cdot|^{s_1}, \dots, N_{e,v} |\cdot|^{s_e}, \sigma_0)$ where $\sigma_0 \in \Pi_{N_v}$. By the multiplicative property of the doubling $L$-factors \cite[Theorem~7.1]{Yam14}, we have:
    \[
    L(s, \pi_v \times \chi_v) = L(s, \sigma_0 \times \chi_v) \prod_{i=1}^e \left( L(s+s_i, N_{i,v} \times \chi_v) \cdot L(s-s_i, N_{i,v}^{\vee} \times \chi_v) \right).
    \]
    Since $N_v$ is tempered, Step 1 implies $L(s, \sigma_0 \times \chi_v) = L(s, N_v \times \chi_v)$. Thus, 
    \[
    L(s, \pi_v \times \chi_v) = L(s, M_v \times \chi_v).
    \]\vspace{0.1cm}

    \item \textbf{Case (ii)}
    
    Without loss of generality, we may assume that
    \[
    M_v = \Big(\bigoplus_{i=1}^f M_{i,v}\boxtimes [d_i]\Big) \oplus \bigoplus_{j=f+1}^r M_{j,v},
    \]
    where $d_i \ge 1$ and $M_{i,v}$ are tempered self-dual representations for $i \le f$, and $M_{j,v}$ are almost tempered self-dual representations for $j > f$. 

    By the Langlands classification for general linear groups, since each $M_{j,v}$ is almost tempered and self-dual representation, it can be written as the Langlands quotient
    \[
    M_{j,v} = L\big( \tau_{j,1}|\cdot|^{a_{j,1}}, \dots, \tau_{j,k_j}|\cdot|^{a_{j,k_j}}, \tau_{j,0}, \tau_{j,k_j}^{\vee}|\cdot|^{-a_{j,k_j}}, \dots, \tau_{j,1}^{\vee}|\cdot|^{-a_{j,1}} \big),
    \]
    where each $\tau_{j,t}$ ($1 \le t \le k_j$) is a tempered representation of $\GL_{m_{j,t}}(F_v)$ with $1/2 > a_{j,1} \ge \dots \ge a_{j,k_j} > 0$, and $\tau_{j,0}$ is a (possibly trivial) tempered self-dual representation.

    Consequently, the local $L$-parameter $\wt{M}_v$ associated with $M_v$ can be rearranged into the following symmetric form:
    \[
    \wt{M}_v = \bigoplus_{i=1}^{e'} N_{i,v} |\cdot|^{s_i'} \oplus N_{v} \oplus \bigoplus_{i=1}^{e'} N_{i,v}^{\vee} |\cdot|^{-s_i'},
    \]
    where $N_{i,v}$ are irreducible tempered representations, $s_1' \ge \dots \ge s_{e'}' > 0$, and $N_v$ is a tempered $L$-parameter. Here, the positive exponents $\{s_i'\}$ and their corresponding representations $N_{i,v}$ arise either from the $A$-parameter segments $[d_i]$ with $d_i \ge 1$, or from the Langlands exponents $a_{j,t}$ and representations $\tau_{j,t}$ of the almost tempered components $M_{j,v}$ for $j > f$. The central tempered parameter $N_v$ absorbs all the exponent-zero components of the $A$-parameters as well as the central tempered representations $\tau_{j,0}$.

    By the definition of $L$-factors for $A$-parameters and the standard properties of local $L$-factors for Langlands quotients, the $L$-factor of $M_v$ factorizes exactly according to this symmetric decomposition. Specifically, we have:
    \begin{align*}
    L(s, M_v \times \chi_v) &= \prod_{i=1}^f \prod_{l=0}^{d_i} L\left(s + \frac{d_i}{2} - l, M_{i,v} \times \chi_v\right) \cdot \prod_{j=f+1}^r L(s, M_{j,v} \times \chi_v) \\
    &= \prod_{i=1}^f \prod_{l=0}^{d_i} L\left(s + \frac{d_i}{2} - l, M_{i,v} \times \chi_v\right) \\
    &\quad \cdot \prod_{j=f+1}^r \left( L(s, \tau_{j,0} \times \chi_v) \prod_{t=1}^{k_j} L(s+a_{j,t}, \tau_{j,t} \times \chi_v) L(s-a_{j,t}, \tau_{j,t}^{\vee} \times \chi_v) \right) \\
    &= L(s, N_v \times \chi_v) \prod_{i=1}^{e'} \left( L(s+s_i', N_{i,v} \times \chi_v) \cdot L(s-s_i', N_{i,v}^{\vee} \times \chi_v) \right).
    \end{align*}
    Following the exact same procedure as in Step 2 of Case (i), utilizing the multiplicativity of doubling $L$-factors for the standard module associated with $\wt{M}_v$ \cite[Theorem~7.1]{Yam14}, we conclude that:
    \[
    L(s, \pi_v \times \chi_v) = L(s, M_v \times \chi_v).
    \]
\end{itemize}
\end{proof}

\end{document}